\documentclass[a4paper,10pt]{article}

\usepackage{amsthm, amsmath, amssymb, mathrsfs, mathtools,booktabs}  
\usepackage[dvipsnames]{xcolor}     
\usepackage[colorlinks]{hyperref}
 
\usepackage{caption,subcaption}
\newtheorem{theorem}{Theorem}[subsection]
\newtheorem{definition}{Definition}[subsection]
\newtheorem{remark}{Remark}[section]
\newtheorem{proposition}{Proposition}[subsection]
\newtheorem{lemma}{Lemma}[subsection]
\newtheorem{problem}{Problem}

\usepackage{bm}
\newcommand{\vect}[1]{\boldsymbol{\mathbf{#1}}} 

\newcommand{\vertiii}[1]{{\left\vert\kern-0.25ex\left\vert\kern-0.25ex\left\vert #1 
    \right\vert\kern-0.25ex\right\vert\kern-0.25ex\right\vert}}

\newcommand{\zd}{v_{{D}}} 

\newcommand{\zdn}{v_{{D}}^{(n)}}

\newcommand{\zdk}{v_{{D}}^{(k)}}

\newcommand{\ezd}{\tilde{v}_{{D}}} 

\newcommand{\ezdn}{\tilde{v}_{{D}}^{(n)}}

\newcommand{\ezdk}{\tilde{v}_{{D}}^{(k)}}

\newcommand{\eun}{\tilde{u}_{{N}}}
\newcommand{\eud}{\tilde{u}_{{D}}} 

\newcommand{\eudn}{\tilde{u}_{{D}}^{(n)}}

\newcommand{\zo}{v(\Omega)}

\newcommand{\udo}{\tilde{u}_{{D}}}

\newcommand{\udpp}{u^{\prime\prime}_{D}} 
\newcommand{\udn}{u_{{D}}^{(n)}}
\newcommand{\un}{u_{{N}}}
\newcommand{\ud}{u_{{D}}}

\newcommand{\udt}{u_{{D}t}}
\newcommand{\unp}{u^{\prime}_{N}}
\newcommand{\udp}{u^{\prime}_{D}}
\newcommand{\pdp}{p^{\prime}_{D}}
\newcommand{\pnp}{p^{\prime}_{N}} 
\newcommand{\pn}{p_{{N}}}
\newcommand{\pd}{p_{{D}}}
\newcommand{\Wad}{\mathcal{W}_{\text{ad}}}

\newcommand{\sfj}{J}
\newcommand{\sfTheta}{\vect{\Theta}}

\newcommand{\VV}{\vect{V}}
\newcommand{\Vn}{V_n}

\newcommand{\WW}{\vect{W}}

\newcommand{\nn}{\nu}

\newcommand{\dn}[1]{\partial_{\nn}{#1}}

\newcommand{\intD}[1]{\int_{D}{#1}{\, d x}}
\newcommand{\intO}[1]{\int_{\Omega}{#1}{\, d x}}

\newcommand{\intS}[1]{\int_{\Sigma}{#1}{\, d s}}  
\newcommand{\intG}[1]{\int_{\Gamma}{#1}{\, d s}} 
\newcommand{\intGast}[1]{\int_{\Gamma^{\ast}}{#1}{\, d s}}

\usepackage{parskip}
\begin{document} 
\title{{Boundary shape reconstruction with Robin condition: existence result, stability analysis, and inversion via multiple measurements}}
\author{Lekbir Afraites$^{\ast}$ and Julius Fergy T. Rabago$^{\dagger}$}

\date{%
	{\footnotesize
 	%%%
        $^\ast$Mathematics and Interactions Teams (EMI), Faculty of Sciences and Techniques\\%
        Sultan Moulay Slimane University, Beni Mellal, Morocco\\\vspace{-1pt}
        \texttt{l.afraites@usms.ma,\ lekbir.afraites@gmail.com}\\[2ex]	
	$^{\dagger}$Faculty of Mathematics and Physics, Institute of Science and Engineering\\%
         Kanazawa University, Kanazawa 920-1192, Japan\\\vspace{-1pt}
        \texttt{rabagojft@se.kanazawa-u.ac.jp,\ jfrabago@gmail.com}\\[2ex]       
}
    \today
}

\maketitle

\begin{abstract}
This study revisits the problem of identifying the unknown interior Robin boundary of a connected domain using Cauchy data from the exterior region of a harmonic function. 
It investigates two shape optimization reformulations employing least-squares boundary-data-tracking cost functionals. 
Firstly, it rigorously addresses the existence of optimal shape solutions, thus filling a gap in the literature. 
The argumentation utilized in the proof strategy is contingent upon the specific formulation under consideration. 
Secondly, it demonstrates the ill-posed nature of the two shape optimization formulations by establishing the compactness of the Riesz operator associated with the quadratic shape Hessian corresponding to each cost functional. 
Lastly, the study employs multiple sets of Cauchy data to address the difficulty of detecting concavities in the unknown boundary. 
Numerical experiments in two and three dimensions illustrate the numerical procedure relying on Sobolev gradients proposed herein.
\medskip

\end{abstract}

%--------------------------  INTRODUCTION --------------------------
%--------------------------  INTRODUCTION --------------------------
%--------------------------  INTRODUCTION --------------------------
\section{Introduction}
\label{sec:Introduction}

%%%%%%%%%%%%%%%%%\subsection{Background and review of the literature} 
We revisit the problem of identifying a connected domain $\Omega$ with an exterior accessible boundary $\Sigma$ and an unknown, inaccessible (interior) boundary $\Gamma$ via a Cauchy pair of data $(f,g)$ on $\Sigma$ for a harmonic function $u$ in $\Omega$.
On $\Gamma$, $u$ is assumed to satisfy a homogeneous Robin boundary condition.
The problem can essentially be expressed as an overdetermined system of partial differential equations (PDEs)
	\begin{equation}
        \label{eq:gip}
        		  - \Delta u	 = 0 	\quad\text{in\ $\Omega$},\qquad
        			      u = f \quad \text{and}\quad \dn{u} = g 	\quad\text{ on \ $\Sigma$},\qquad
        	\dn{u} + \alpha u = 0	\quad\text{on\ $\Gamma$},
        \end{equation}
where $\alpha$, generally, is assumed to be a fixed non-negative Lipschitz function in $\mathbb{R}^{d}$, $d \in \{2,3\}$, such that $\alpha \geqslant \alpha_{0} > 0$, where $\alpha_{0}$ is a known constant and $\partial_{\nn}{u}$ is the (outward) unit normal derivative of $u$.
In inverse problem framework, \eqref{eq:gip} can be stated as follows:
\begin{problem}\label{eq:main_problem}
	Given the Dirichlet data $f$ on $\Sigma$ and the measured Neumann data
	\[
		g:= \dn{u}\qquad\text{on $\Sigma$},
	\]
	where $u:\Omega \to \mathbb{R}$ solves
	\begin{equation}
        \label{eq:state}
        		  - \Delta u	 = 0 	\quad\text{in\ $\Omega$},\qquad
        			      u = f 	\quad\text{on\ $\Sigma$},\qquad
        	\dn{u} + \alpha u = 0	\quad\text{on\ $ \Gamma $},
        \end{equation}
	determine the shape of the unknown portion of the boundary $\Gamma$.
\end{problem}
Actual application of the problem -- which is popular in engineering sciences, particularly in materials science and biomedical engineering \cite{KurahashiMaruokaIyama2017} -- includes the identification of a cavity or inclusion (or the reconstruction of inaccessible boundary or the detection of interior cracks in a conducting medium) by electrostatic measures or thermal imaging techniques on the external and accessible part $\Sigma$.
In the former case, $u$ is interpreted as the electrostatic potential in a conducting body $\Omega$ of which only the part $\Sigma$ of the boundary is accessible for testing and evaluation.
Problem \ref{eq:main_problem} can then be interpreted as the determination of the shape of the inaccessible portion $\Gamma$ of the boundary from the knowledge of the imposed voltage $u\big|_{\Sigma}$ and the measured resulting current $\dn{u}\big|_{\Sigma}$ on $\Sigma$.
Other applications of this problem (or slightly modified versions) are mentioned in \cite{AlessandriniDelPieroRondi2003,AlessandriniSincich2007,ChaabaneaJaoua1999,CakoniKress2007,FangLu2004,RabagoAzegami2018}.
For a more detailed discussion about the model equation, we refer the readers, for instance, to \cite{FasinoInglese2007,Inglese1997,KaupSantosa1995,KaupSantosaVogelius1996}.
Meanwhile, for some numerical studies on recovering $\Gamma$ from measurements on $\Sigma$, assuming the knowledge of $\alpha$, see \cite{AfraitesRabago2022,CakoniKress2007,CakoniKressSchuft2010a,CakoniKressSchuft2010a,KressRundell2005,Loh1987,Rundell2008,FangLinMa2019,FangZeng2009}.

In Problem \ref{eq:main_problem}, the cases $\alpha = 0$ and $\alpha = \infty$ correspond, respectively, to homogeneous Neumann and Dirichlet boundary conditions, as studied in \cite{KressRundell2005} and \cite{IvanyshynKress2006}. 
Our main focus is on the case where $\alpha \in \mathbb{R}^{+}$, except in subsection \ref{subsection:existence_tracking_Neumann_data}, where $\alpha$ is assumed to be a Lipschitz function with additional properties. 
Thus, unless specified otherwise, $\alpha$ as a positive constant.

Regarding the identifiability issue for the inverse Robin problem, we note the following. 
In \cite{IngleseMariani2004}, Inglese and Mariani established local uniqueness of $\Gamma$ under the assumption that $\Omega$ is a thin rectangular plate. 
In \cite{CakoniKress2007}, Cakoni and Kress highlighted that, generally, for a fixed constant impedance $\alpha$, a single Cauchy pair $(f,g)$ on $\Sigma$ can lead to infinitely many different domains $\Omega$. 
They provided counterexamples demonstrating that a single Cauchy pair is insufficient to identify $\Gamma$, particularly when determining both the shape $\Gamma$ and the impedance $\alpha$ simultaneously. 
Bacchelli further showed in \cite{Bacchelli2009} that two pairs of Cauchy data ensure uniqueness of $\Gamma$ and $\alpha$ simultaneously, provided that the input data are linearly independent and one of them is positive. 
Additionally, Pagani and Pierotti in \cite{PaganiPierotti2009} also established uniqueness results using two measures for the inverse Robin problem. 
On the topic of stability using solutions corresponding to two independent input data, we refer readers to the work of Sincich in \cite{Sincich2010}.
Meanwhile, for cases involving homogeneous Dirichlet or Neumann boundary conditions, one pair of Cauchy data uniquely determines the missing part of the boundary; see \cite[Thm. 2.3]{CakoniKress2007}. 
In the case of Neumann boundary data on $\Gamma$, for example, one can establish that $(f,g) = (u, \dn{u})|_{\Sigma}$ uniquely determines $\Gamma$ provided that $f$ is not constant \cite[Rem. 2.4]{CakoniKress2007}.
%
%

%%%%%%%%%%%%%%%%%\subsection{Motivation} 
The shape inverse problem under consideration is well-known to be ill-posed in the sense of Hadamard \cite{EpplerHarbrecht2005,Fang2022}.
In fact, previous numerical studies \cite{AfraitesRabago2022,RabagoAzegami2018} have highlighted the difficulty of accurately detecting an unknown Robin boundary $\Gamma$, particularly when it includes concave parts (see \cite{AfraitesRabago2022} for illustrative examples, and \cite{CaubetDambrineKateb2013} for a related study).
Here we aim to demonstrate -- through numerical experiments -- that this difficulty can be partially overcome by utilizing multiple boundary measurements.
Employing this strategy significantly improves detection outcomes compared to using a single boundary measurement.
While it may seem intuitive to leverage multiple boundary data for enhanced detection, limited exploration of this idea exists in the context of the present shape inverse problem -- to the best of our knowledge.
Notable exceptions include recent works by Fang \cite{Fang2022}, Rundell \cite{Rundell2008}, a related investigation on an interface problem \cite{GiacominiPantzaTrabelsi2017}, and a study on identifying acoustic sources in a domain \cite{AlvesMartinsRoberty2009}.
This finding underscores our study's primary contribution to the existing literature.
Additionally, we extend our work to include a numerical experiment in three dimensions, highlighting the advantages of our proposed strategy over the classical approach of relying solely on a single boundary measurement.

%%%%%%%%%%%%%%%%%\subsection{Contributions} 
Despite the well-established understanding of the boundary inverse problem for the Laplacian with Robin conditions, to our knowledge, the accurate detection of concave parts or regions on the unknown boundary remains unexplored in previous studies, at least from a numerical perspective.
Additionally, a rigorous proof of the existence of a shaped solution to the classical shape optimization formulation with boundary-data-tracking type cost functionals -- as far as we know  -- is currently lacking in the literature. 
Regarding this issue, it must be noted that for domains with even mildly oscillating boundaries converging to a smooth domain, the Robin condition might not be preserved. 
This makes the existence proof even more delicate.
These observations warrant further investigation of the problem within the context of shape optimization settings.

The present study is the first part of a two-part investigation.
In this work, the investigation that will be carried out covers the following: (ii) existence results for the minimizer of the considered cost functionals; (i) analyses of the shape Hessians of the cost functionals at their respective optimal shapes; and (iii) numerical investigations focusing on the comparison between reconstructions with one or more Cauchy data.
The second part or follow-up work of the study will focus on the following topics: (i) quantitative analyses; (ii) development of a second-order numerical method; and (iii) numerical analysis of the reconstruction's sensitivity to noise, involving the introduction of a suitable regularization term in the functional. 
This will, in fact, include a proposal for a different shape optimization approach based on an augmented Lagrangian method.
On the other hand, we emphasize here that since the shape Hessian is computed only at the minimizer, the analysis is insufficient to develop a second-order method, which, despite its described limitations, is still of interest. 
These research directions are deliberately postponed in a follow-up investigation to maintain the focus of the present study and to avoid making the discussion too extensive.

%%%%%%%%%%%%}}
%--------------------------  PROBLEM SETTING --------------------------
%--------------------------  PROBLEM SETTING --------------------------
%--------------------------  PROBLEM SETTING --------------------------
%%%%%%%%%%%%%%%%%%%%%%%%%%%%%%%%%%%%%%%%%%%%%%%%%%%%%%%%%%%%%%%%%%%%%%%%%%%%%%
%----------------------------------------------------------------------------------------------------------------------------------
%	THE MAIN PROBLEM
%----------------------------------------------------------------------------------------------------------------------------------
%%%%%%%%%%%%%%%%%\subsection{The shape inverse problem} 
Let us briefly discuss some specifications of the shape inverse problem that we aim to address here in the next few lines.
Let $D$ be a (non-empty, open, bounded, and simply connected) planar set of class $\mathcal{C}^{2,1}$, and $\delta > 0$ be a fixed (small) real number.
Define the collection of all admissible (non-empty) inclusions $\omega$ by $\Wad$ as follows
\begin{equation}\label{eq:set_Wad}
\Wad :=  \left\{ \ \omega \Subset D \  \left|
\begin{aligned}
	&\ \ \text{$\omega \in \mathcal{C}^{2,1}$ is an open bounded set}, \\
	&\ \ \text{$d(x,\partial D) > \delta$ for all $x \in \omega$, and}\\
	&\ \ \text{$D\setminus \overline{\omega}$ is open, bounded, and connected.}
\end{aligned}
\right. \
\right\}
\end{equation}
In above set, we emphasize that essentially, the annular domain $D\setminus \overline{\omega}$ -- we simply denote here by $\Omega$ -- has $\mathcal{C}^{2,1}$ boundary $\partial D \cup \partial \omega$.
The said regularity assumption on the boundary is instrumental in the well-posedness of the shape optimization problem(s) that we shall consider here.

Our main problem here now is the following:
\begin{equation}\label{eq:shape_problem}
	\text{Find $\omega \in \Wad$ such that $u=u(\Omega):=u(D\setminus \overline{\omega})$ satisfies \eqref{eq:gip}.}
\end{equation}
We refer to $\Omega^{\ast}=D\setminus \overline{\omega}^{\ast}$, or equivalently $\Gamma^{\ast}=\partial\omega^{\ast}$, as a solution of \eqref{eq:gip} if the pair $(\Omega^{\ast},u(\Omega^{\ast}))$ satisfies \eqref{eq:shape_problem}. 
Hereafter, we understand that $\Sigma=\partial D$ and $\Gamma=\partial\omega$, and we assume, unless otherwise stated, that $f \in H^{5/2}(\Sigma)$ and $f \not\equiv 0$.
Also, we let $g \in H^{3/2}(\Sigma)$ be an admissible boundary measurement corresponding to $f$.
That is, $g$ belongs to the image of the Dirichlet-to-Neumann map $\Lambda_{\Sigma}: f \in H^{5/2}(\Sigma) \mapsto g:=\dn{u} \in H^{3/2}(\Sigma)$, where $u$ solves \eqref{eq:state}. 
\begin{remark}
The regularity assumptions given on the boundary data $f$ and $g$ are more than we can actually expect.
In fact, we can only assume that $\Sigma$ is Lipschitz and that $f \in H^{1/2}(\Sigma)$ and $g \in H^{-1/2}(\Sigma)$.
However, for ease of notations and to simplify many proofs, we assume the announced regularity.
Besides we will carry out a second-order analysis of the problem, so we need sufficient regularity (at least $\mathcal{C}^{2,1}$) for the domain as well as the regularity assumption on the data $f \in H^{5/2}(\Sigma)$ and $g \in H^{3/2}(\Sigma)$ to ensure the existence of the shape derivative $\udp$ of $\ud$ in $H^{1}(\Omega)$.
On a side note, we comment that the regularity assumptions: $ \Gamma $ is of class $ \mathcal{C}^{1,1} $ and $ f \in H^{1/2}(\Sigma) $ (with $ \alpha $ non-negative), are suitable assumptions to deal with the unique solvability of the inverse problem under consideration, see \cite{PaganiPierotti2009}. 
Thus, it is only reasonable to consider domains that are at least of class $ \mathcal{C}^{1,1} $ in the present study.
\end{remark}
%%%%%%%%%%%%%%%%%%%%%%%%%%%%%%%%%%%%%%%%%%%%%%%%%%%%%%%%%%%%%%%%%%%%%%%%%%%%%%
%--------------------------
%	ORGANIZATION OF THE PAPER
%--------------------------

%%%%%%%%%%%%%%%%%\subsection{Organization of the paper} 
The rest of the paper is divided into three main sections. 
Section \ref{sec:theory} discusses the existence of optimal solutions for the two shape optimization formulations with boundary-type data-tracking functionals considered in this context.
Section \ref{sec:second-order_analyses_and_stability_issue} focuses on characterizing the shape Hessian's structure of the cost functionals associated with each formulation at critical points and examines the problem's ill-posedness by analyzing the compactness of the Hessian expression. 
Section \ref{sec:numerics} outlines the numerical algorithm used to solve minimization problems with multiple boundary measurements and presents experiments in two and three dimensions. 
The study concludes with remarks in Section \ref{sec:conclusion}.
%%%%%%%%%%%%%%%%%%%%%%%%%%%%%%%%%%%%%%%%%%%%%%%%%%%%%%%%%%%%%%%%%%%%%%%%%%%%%%
%----------------------------------------------------------------------------------------------------------------------------------
%	FORMULATION AND NOTATIONS
%----------------------------------------------------------------------------------------------------------------------------------
\section{Shape optimization formulations}\label{sec:theory}
\subsection{Tracking the Neumann data in least-squares approach}\label{sec:tracking_the_Neumann_data} 
A way to solve Problem \ref{eq:shape_problem} is to consider the following minimization problem which consists of tracking the Neumann data on the accessible boundary in $L^{2}$ sense:
\begin{equation}\label{eq:shop}
	J_{N}(\Omega) := J_{N}(D\setminus\overline{\omega}) = \frac12 \intS{\left(\dn{\ud} - g \right)^{2}} \to\ \inf,\footnote{We observe that, following \eqref{eq:shape_problem}, it is more appropriate to denote $J_{N}(\omega)$ rather than $J_{N}(\Omega)$ or $J_{N}(D\setminus\overline{\omega})$. Nonetheless, in this context, these notations are interchangeable, and the choice of notation is a matter of personal preference. Additionally, it is more accurate to express $J(\Omega,u(\Omega))$ instead of the reduced functional $J(\Omega)$. However, we have opted for the latter notation since it can be demonstrated that the mapping $\Omega \mapsto u(\Omega)$ is continuous (given suitable conditions, as assumed here, see subsection \ref{subsection:existence_tracking_Neumann_data}), and therefore well-defined. It is worth noting that, due to the unique solvability of \eqref{eq:state_ud}, the mapping $\Omega \mapsto u(\Omega)$ can be defined.}
\end{equation}
where $\ud:\Omega\to\mathbb{R}$ solves the system of PDEs
%%%
\begin{equation}
\label{eq:state_ud}
	- \Delta \ud  		= 0		\quad\text{in $\Omega$},\qquad
	\ud	 			= f		\quad\text{on $\Sigma$},\qquad
	\dn{\ud} + \alpha \ud = 0	\quad\text{on $\Gamma$}.
\end{equation}
In \eqref{eq:shop}, the infimum is naturally taken over a set of admissible domains, such as $\Wad$ defined in \eqref{eq:set_Wad}
%
%
%
%------------------------------------------------------------------------------------------------------------------------------------------------------------------------------------------------------
\begin{remark}[Well-posedness of \eqref{eq:state_ud}]\label{rem:weakforms}
Define the bilinear form $a$ as follows
\begin{equation}
	\label{eq:bilinear_form}
	a({\varphi},{\psi}) = \intO{\nabla \varphi \cdot \nabla \psi} + \intG{ \alpha \varphi \psi},  \quad \text{where ${\varphi},{\psi}\in H^{1}(\Omega)$},
\end{equation}
and let $H_{\Sigma,0}^{1}(\Omega)$ be the Hilbert space $\{\varphi \in H^{1}(\Omega) \mid \varphi = 0 \ \text{on} \ \Sigma\}$.
Then, the variational formulation of \eqref{eq:state_ud} can be stated as follows:
	\begin{equation}\label{eq:state_weak_form_ud}
		\text{Find ${\ud} \in H^{1}(\Omega)$, $\ud = f$ on $\Sigma$, such that $a(\ud,\psi) = 0$, for all $\psi \in H_{\Sigma,0}^{1}(\Omega)$}.
	\end{equation}
Assuming that $f \in H^{1/2}(\Sigma)$ and $\Omega$ is Lipschitz, the existence of unique weak solution $\ud \in H^{1}(\Omega)$ to \eqref{eq:state_weak_form_ud} follows from the application of the Lax-Milgram lemma.
\end{remark}
%------------------------------------------------------------------------------------------------------------------------------------------------------------------------------------------------------
%
%
%
The shape optimization problem given by \eqref{eq:shop} is equivalent to the overdetermined problem \eqref{eq:gip} provided that we have the perfect match of boundary data on the unknown boundary, i.e., $u = f$ and $\dn{u} = g$ on $\Sigma$.
Indeed, if $\omega^{\ast}$ (or equivalently, $\Omega^{\ast}$) solves \eqref{eq:shape_problem}, then $J_{N}(\Omega^{\ast})  = J_{N}(D\setminus \overline{\omega}^{\ast}) = 0$, and it holds that
\begin{equation}\label{eq:argmin}
\omega^{\ast} \in \operatorname{argmin}_{\omega \in \Wad}J_{N}(D\setminus\overline{\omega}).
\end{equation}
Conversely, if $\omega^{\ast}$ solves \eqref{eq:argmin} with $J_{N}(D\setminus\overline{\omega}^{\ast}) = 0$, then it is a solution of \eqref{eq:shape_problem}.

On a side note, we comment that $J_{N}$ requires a high level of regularity for the state $\ud$ to be well-defined. 
This requirement makes it impractical for numerical experiments without guaranteed regularity of the state variables. 
However, as we will demonstrate in the numerical section of this study, it still offers reasonable reconstruction of shapes for Problem \ref{eq:main_problem}.
%%%%%%%%%%%%%%%%%%%%%%%%%%%%%%%%%%%%%%%%%%%%%%%%%%%%%%%%%%%%%%%%%%%%%%%%%%%%%
%%%%%%%%%%%%%%%%%%%%%%%%%%%%%%%%%%%%%%%%%%%%%%%%%%%%%%%%%%%%%%%%%%%%%%%%%%%%%
%%%%%%%%%%%%%%%%%%%%%%%%%%%%%%%%%%%%%%%%%%%%%%%%%%%%%%%%%%%%%%%%%%%%%%%%%%%%%
%--------------------------  EXISTENCE OF OPTIMAL SOLUTION  --------------------------

\subsubsection{Existence of a shape solution}
\label{subsection:existence_tracking_Neumann_data}
In this subsection, we investigate the existence of an optimal solution to \eqref{eq:shop} within planar domains. 
To this end, we make the assumption that the entire boundary of any admissible domain (which will be rigorously defined later) is $\mathcal{C}^{k,1}$ regular. 
This implies that $\partial\Omega$ can be parametrized by a $\mathcal{C}^{k,1}$ function. 
We will specify the value of $k$ for technical clarity. Unless explicitly stated otherwise, this level of regularity will be imposed on any open and bounded set considered in this section.

To proceed, we consider instead of \eqref{eq:shop} the shape optimization problem
\begin{equation}
\label{eq:new_shape_opti}
	\min_{\Omega} J_{N}(\Omega) := \min_{\Omega} \left\{ \frac{1}{2} \intS{(\dn{\ud} - g)^{2}} \right\},
\end{equation}
where $\ud(\Omega)  = \zd(\Omega) + \udo(\Omega) =: \zd + \udo$ such that
 \begin{align}
	&-\Delta \zd = \Delta \udo\ \ \text{in} \ \Omega,\qquad
	\zd = 0\ \ \text{on} \ \Sigma,
	&\dn{\zd} + \alpha \zd = - \dn{\udo} - \alpha \udo\ \ \text{on} \ \Gamma\label{eq:state_zd}.
\end{align}
Here, $\udo \in H^{2}(D)$ is an arbitrarily fixed function such that $\udo = f$ on $\Sigma$, 
and $\ud(\Omega)$ satisfies \eqref{eq:state_ud}. 
The corresponding variational formulation of \eqref{eq:state_zd} can be given as follows:
%%%%%%
	\begin{equation}	\label{eq:weak_vd}
		\text{Find $\zd \in H_{\Sigma,0}^{1}(\Omega)$ such that $a(\zd, \varphi) = -a(\udo, \varphi)$, for all $\varphi \in H_{\Sigma,0}^1(\Omega)$}.
	\end{equation}
%%% 
%
Since \eqref{eq:weak_vd} has a unique solution in $H^{1}(\Omega)$, we can define the mapping $\Omega \rightarrow \zd := \zd(\Omega)$, and denote its graph by
\[
    \mathcal{G} = \{(\Omega,\zd): \text{$\Omega \in \mathcal{O}_\text{ad}$ and $\zd$ solves \eqref{eq:weak_vd}}\},
\]
where $\mathcal{O}_\text{ad}$ represents the set of admissible domains as defined in \eqref{eq:admissible_domain}.
This implies that \eqref{eq:new_shape_opti} is equivalent to minimizing $J_{N}(\Omega,\zo)$ over $\mathcal{G}$. To prove the existence of a solution, we equip $\mathcal{G}$ with a topology ensuring that $\mathcal{G}$ is compact and $J_{{N}}$ is lower semi-continuous.
Compactness implies that for any minimizing sequence in $\mathcal{O}_\text{ad}$, denoted by $\{\Omega^{(n)}\}:=\{\Omega^{(n)}\}_{n=1}^{\infty}$, there exists a subsequence $\{\Omega^{(k)}\}$ converging in some sense to a limit domain $\Omega^{0}$. We then investigate whether $\zd(\Omega^{(n)})$, satisfying \eqref{eq:state_zd} on $\Omega^{(n)}$, converges to the solution $\zd(\Omega^{0})$ of \eqref{eq:state_zd} on $\Omega^{0}$.
It is essential to note that unlike Dirichlet or Neumann problems, Robin problems may have different boundary conditions in the limit, depending on how the domains $\{\Omega^{(n)}\}$ approach the limit domain $\Omega^{0}$ \cite{DancerDaners1997}.
In fact, convergence of solutions is not guaranteed, even for smooth domains. For instance, in \cite[Ex. 5.2]{DancerDaners1997}, it was shown that for domains with mildly oscillating boundaries converging to a smooth domain, the limit solution satisfies the Robin boundary condition with a different Robin term (see \cite[Ex. 5.3]{DancerDaners1997}), where, for constant $\alpha$, the limit function satisfies a Robin boundary condition with a larger Robin coefficient.
So, in general, the solution of the PDE constraint may not converge to the solution for the limit domain when convergence is achieved only for the domain and not for the boundaries (i.e., $\Omega^{(n)} \rightarrow \Omega$ but $\partial \Omega^{(n)}$ does not converge to $\partial \Omega$). 
Nonetheless, the implied convergence ``$\Omega^{(n)} \rightarrow \Omega \Rightarrow \partial \Omega^{(n)} \rightarrow \partial \Omega$'' holds true in the Hausdorff sense for domains with Lipschitz boundaries \cite[Ex. 3.2]{Holzleitner2001} or the cone property \cite{Chenais1975}. 
For details about convergence in the sense of Hausdorff, we refer readers to \cite[Sec. 2.2.3, Def. 2.2.8, p. 30]{HenrotPierre2018}.

Let us now characterize the set $\mathcal{O}_\text{ad}$, and then specify an appropriate topology on it.
We consider the following problem:
	\begin{equation}
	\label{eq:minimization_problem}
	\text{Find $(\Omega^{\ast},\zd(\Omega^{\ast}))\in \mathcal{G}$ such that $J_{N}(\Omega^{\ast},\zd(\Omega^{\ast})) \leqslant J_{N}(\Omega,\zd)$, for all $(\Omega,\zd) \in \mathcal{G}$}.
	\end{equation}	
%%%
To prove the existence of optimal solution to \eqref{eq:minimization_problem}, we will follow the ideas developed in \cite{HKKP2004a,HKKP2003,HaslingerMakinen2003} and apply the tools furnished in \cite{BoulkhemairChakib2007,Boulkhemairetal2008,Boulkhemairetal2013}, while mimicking some arguments from \cite{RabagoAzegami2019b,RabagoAzegami2020} in proving the continuity of the state problems with respect to the domain. 
%
%

%%{\critical{
{{
%%-------------------------------------------------------------------------------------------------------------------------------------------------------------------------------
Hereinafter, we denote by $\mathcal{C}^{k,1}_{1}$ ($k\in\mathbb{N}_{0}$) the space of restrictions to $(0,1]$ of the subspace of $1$-periodic functions in $\mathcal{C}^{k,1}(\mathbb{R};\mathbb{R}^{2})$, and assume that the free boundary $\Gamma=\partial \omega$ can be parametrized by a vector function $\vect{\phi} \in \mathcal{C}^{k,1}(\mathbb{R};\mathbb{R}^{2})$.

We issue the following definition.
\begin{definition}\label{def:assumptions}
	Let $\delta$ be a fixed positive number.
	A vector function $\vect{\phi} \in\mathcal{C}^{1,1}(\mathbb{R};\mathbb{R}^{2})$ is said to be in the set $\mathcal{U}_{\text{ad}}$, a closed and bounded subset of $\mathcal{C}^{1,1}(\mathbb{R};\mathbb{R}^{2})$, if it satisfies the following properties:
\begin{description}
	\item[$(\text{P1})$] $\vect{\phi}$ is injective on $(0,1]$ and is $1$-periodic;
	%%%
	\item[$(\text{P2})$] there exist positive constants $c_{0}$, $c_{1}$, $c_2$, and $c_3$ such that
		\begin{align*}
			|\vect{\phi}(t)| \leqslant c_{0}, \ \
			c_{1} \leqslant |\vect{\phi}'(t)| \leqslant c_2, \ \forall t \in (0,1),
			\quad \text{and} \quad |\vect{\phi}''(t)| \leqslant c_3 \ \ \text{for a.e.}\ \  t \in (0,1),
		\end{align*}
		where $|\cdot|$ denotes the Euclidean norm;
	\item[$(\text{P3})$] $\overline{\Omega} = \overline{\Omega}(\vect{\phi}) \subset D \subset \mathbb{R}^{2}$;
	\item[$(\text{P4})$] there exists a positive constant $\delta_{0} < \delta$ such that $\operatorname{dist}(\Sigma,\Gamma(\vect{\phi})) \geqslant \delta_{0}$.
\end{description}
\end{definition}
Henceforth, we assume that $\vect{\phi} \in \mathcal{U}_{\text{ad}}$.
Accordingly, we define
\begin{equation}
	\label{eq:admissible_domain}
	\mathcal{O}_\text{ad}=\{\Omega=\Omega(\vect{\phi}) \mid \vect{\phi} \in \mathcal{V}_\text{ad}\},
\end{equation}
where $\Omega(\vect{\phi})=\Omega(\Gamma(\vect{\phi}))$, $\mathcal{V}_{\text{ad}} \Subset \mathcal{U}_{\text{ad}}$, and the regularity assumption on $\omega$ is for now relaxed to $\mathcal{C}^{1,1}$.

An example of $\mathcal{V}_\text{ad}$ is the set of elements $\vect{\phi} \in \mathcal{C}^{1,1}((0,1]; \mathbb{R}^{2})$ such that $\vect{\phi}$ satisfies $(\text{P2})$--$(\text{P4})$ (with $\omega \in \Wad$ for instance), except that the last inequality in $(\text{P2})$ is replaced by the Lipschitz condition $|\vect{\phi}'(t_{1}) - \vect{\phi}'(t_2)| \leqslant c_3 |t_{1} - t_2|$, for all $t_{1}, t_2 \in (0,1]$, and the set defined by these conditions is compact in {{$\mathcal{C}^{1,1}$.}}
Another example is any closed subset of $\mathcal{U}_{\text{ad}}$ which is bounded in $\mathcal{C}^{2,\mu}(\mathbb{R};\mathbb{R}^{2})$, for some $\mu \in (0,1]$.
Additionally to \eqref{eq:admissible_domain}, we also consider the larger set 
\[
	\tilde{\mathcal{O}}_\text{ad}:=\{\Omega=\Omega(\vect{\phi})  \mid \vect{\phi} \in \mathcal{U}_{\text{ad}}\}.
\] 

Let us emphasize some important features of the domains in $\mathcal{O}_\text{ad}$. 
The second inequality condition in $(\text{P2})$ essentially prevents impractical oscillations on the free boundary, which we aim to avoid due to the aforementioned issue.
With this remark, we also impose a condition for the fixed accessible boundary $\Sigma$ similar to $(\text{P2})$, but with different constants, to prevent unreasonable shapes for the specimen.

For later use, we note the following assumption:
\begin{equation}\label{eq:assumption_on_fixed_boundary}\tag{B}
\left\{
	\begin{aligned}
	&\text{the boundary $\Sigma$ can be parametrized by a $\mathcal{C}^{1,1}$ ($1$-periodic) function $\vect{\phi}_{0}$ such that}\\
	&\quad 	|\vect{\phi}_{0}(t)| \leqslant b_{0}, \
			b_{1} \leqslant |\vect{\phi}_{0}'(t)| \leqslant b_2, \ \forall t \in (0,1),
			\ \text{and} \ |\vect{\phi}_{0}''(t)| \leqslant b_3 \ \text{for a.e.}\  t \in (0,1),\\
	&\text{for some positive constants $b_{0}$, $b_{1}$, $b_2$, and $b_3$.}
	\end{aligned}\right.
\end{equation}
%

%%%%%%%% ABOUT UNIFORM CONE AND EXTENSION PROPERTY OF THE ADMISSIBLE DOMAINS
Meanwhile, the definition given in \eqref{eq:admissible_domain} implies that every admissible domain $\Omega(\vect{\phi})$ satisfy the well-known \emph{uniform cone property} \cite[Thm. 2.4.7, p. 56, and Rem. 2.4.8, p. 59]{HenrotPierre2018} (see Definition \ref{definition:cone_property} in subsection \ref{subsec:existence_tracking_the_Dirichlet_data} for the cone property) since every admissible domain is essentially a uniformly Lipschitz open set in $\mathbb{R}^{2}$.
	As a consequence, every set in $\mathcal{O}_\text{ad}$ satisfies a very important extension property in the sense of Chenais \cite{Chenais1975} (see Lemma \ref{lemma:extension_operator}). 
	More exactly, for every $k \geqslant 1$, $p > 1$, and domain $\Omega \in \tilde{\mathcal{O}}_\text{ad}$, there exists an extension operator $\mathcal{E}_{\Omega}:W^{k,p}(\Omega) \to W^{k,p}(D)$ such that $\|\mathcal{E}_{\Omega}\varphi\|_{W^{k,p}(D)} \leqslant C\|\varphi\|_{W^{k,p}(\Omega)}$, where $C>0$ is a constant independent of $\Omega$.
%%%
By these properties, we are guaranteed of an extension $\tilde{\varphi} \in H^{k}(D)$, $k \geqslant 1$, from $\Omega$ to $D$ of every function $\varphi \in H^{k}(\Omega)$.
For instance, because $\ud$ is $H^{2}(\Omega)$ regular, the function $\zd$ is also $H^{2}(\Omega)$, and so we can find an extension $\ezd$ of $\zd$ from $\Omega$ to $D$ that is $H^{2}(D)$ regular.
With these considerations, we can now specify the topology that we will use to investigate the existence of an optimal solution to the shape problem described by \eqref{eq:new_shape_opti}.

First, we define the convergence of a sequence $\{\vect{\phi}^{(n)}\} \subset \mathcal{U}_{\text{ad}}$ by
\begin{equation}
	\label{eq:phi_sequence_convergence}
	\vect{\phi}^{(n)} \to \vect{\phi} \qquad \Longleftrightarrow \qquad \text{$\vect{\phi}^{(n)} \to \vect{\phi}$ in the $\mathcal{C}^{1}([0,1])$-topology.}
\end{equation} 
In fact, the latter convergence can be inferred from the properties of elements of $\mathcal{U}_{\text{ad}}$ and by Arzel\`{a}-Ascoli theorem.
Accordingly, we define the convergence of a sequence of domains $\{\Omega^{(n)}\}:=\{\Omega(\vect{\phi}^{(n)})\}:=\{\Omega(\vect{\phi}^{(n)})\}_{n=1}^{\infty} \subset \tilde{\mathcal{O}}_\text{ad}$ by
\begin{equation}
	\label{eq:Omega_sequence_convergence}
	\Omega^{(n)} \to \Omega \qquad \Longleftrightarrow \qquad \vect{\phi}^{(n)} \to \vect{\phi}.
\end{equation}
Meanwhile, we define the convergence of a sequence $\{\zdn\}$ of solutions of \eqref{eq:weak_vd} on $\Omega^{(n)}$ to the solution of \eqref{eq:weak_vd} on $\Omega$ as follows
\begin{equation}
	\label{eq:state_un_sequence_convergence}
	\zdn \to \zd \qquad \Longleftrightarrow \qquad \text{$\mathcal{E}_{{{\Omega^{(n)}}}}{\zdn}=:\ezdn \to \ezd:=\mathcal{E}_{\Omega}{\zd}$ \quad weakly in $H^{1}(D)$}.
\end{equation}
{{{We emphasize here that the extension operator $\mathcal{E}_{{{\Omega^{(n)}}}}$ for $\zdn$ ensures that $\ezdn = \zdn$ on $\Omega^{(n)}$.}}

Finally, the topology we introduce on $\mathcal{G}$ is the one induced by the convergence defined by
\begin{equation}
	\label{eq:convergence}
	(\Omega^{(n)}, \zdn) \to (\Omega, \zd) \qquad \Longleftrightarrow \qquad
		\vect{\phi}^{(n)} \to \vect{\phi} \quad \text{and} \quad \zdn \to \zd.
\end{equation}
%------------------------------
%	REMARK
%------------------------------
\begin{remark}
	\label{rem:convergence_of_domains_and_boundaries}
	As mentioned in passing, the uniform cone property of the admissible sets also implies a compactness property of $\mathcal{O}_\text{ad}$ with respect to the Hausdorff metric (see \cite[Sec. 2.2.3, Def. 2.2.7, p. 30]{HenrotPierre2018}). 
	This property ensures that once the convergence of domains $\Omega^{(n)} \to \Omega$ is true, it also holds that $\partial\Omega^{(n)} \to \partial\Omega$ (specifically, $\Gamma^{(n)} \to \Gamma$) for the free boundaries in the Hausdorff sense \cite[Thm. 2.4.10, p. 59]{HenrotPierre2018} (see also \cite{Holzleitner2001}).
Moreover, from the definition of the Hausdorff metric, one easily finds that if $\Gamma^{(n)} \to \Gamma$ in the Hausdorff sense, as $n \to \infty$, then for any $\varepsilon > 0$ there is a $k_{0}:=k_{0}(\varepsilon) \in \mathbb{N}$ such that $\Gamma_k$ belongs to the $\varepsilon$-neighborhood of $\Gamma$, for all $\mathbb{N} \ni k \geqslant k_{0}$.
\end{remark}
Denoting the convergence of a sequence $\{O^{(k)}\}$ of sets in $\mathbb{R}^{d}$ to $O \subset \mathbb{R}^{d}$ in the Hausdorff metric by $O^{(k)} \overset{\text{H}} \longrightarrow O$, as $k \to \infty$ (see Definition \ref{def:Hausdorff_convergence}), we give a proper statement of the above-mentioned property in the following lemma whose proof can be found in \cite{Holzleitner2001}.
%------------------------------
%	LEMMA: HAUSDORFF CONVERGENCE
%------------------------------
\begin{lemma}
	\label{lem:Hausdorff_convergence}
	Let $\{\vect{\phi}^{(k)}\} \subset \mathcal{U}_{\text{ad}}$ and $\vect{\phi} \in \mathcal{U}_{\text{ad}}$.
	For any sequence of domains $\{\Omega(\vect{\phi}^{(n)})\}$, $\Omega(\vect{\phi}^{(n)}) \in \mathcal{O}_\text{ad}$, there is a subsequence $\{\Omega(\vect{\phi}^{(k)})\} \subset \{\Omega(\vect{\phi}^{(n)})\}$ and a domain  $\Omega(\vect{\phi}) \in \mathcal{O}_\text{ad}$ such that $\Omega(\vect{\phi}^{(k)})  \overset{\text{H}} \longrightarrow \Omega(\vect{\phi})$, and $\Gamma(\vect{\phi}^{(k)}) \overset{\text{H}} \longrightarrow \Gamma(\vect{\phi})$, as $k\to \infty$, where $\Gamma(\vect{\phi}^{(k)})$ and $\Gamma(\vect{\phi})$ are the free boundaries or graphs of $\vect{\phi}^{(k)}$ and $\vect{\phi}$, respectively.
\end{lemma}
\begin{remark}
	Properties $(\text{P1})$--$(\text{P4})$ of $\vect{\phi}$ and $\vect{\phi}^{(n)}$ not only provide a very well behaved convergence of a subsequence of open sets $\{\Omega^{(k)}\}$ of the sequence $\{\Omega^{(n)}\} \subset \mathcal{O}_\text{ad}$ to an open set $\Omega \in \mathcal{O}_\text{ad}$ in the sense of Hausdorff, but convergences in the sense of characteristic functions and in the sense of compacts are also achieved by these subsequence \cite[Thm. 2.4.10, p. 59]{HenrotPierre2018}.
\end{remark}
\begin{theorem}
	\label{prop:existence_of_optimal_solution}
	The minimization problem \eqref{eq:minimization_problem} admits a solution in $\mathcal{G}$.
\end{theorem}
As stated before, the existence proof is reduced to proving the compactness of $\mathcal{G}$ and the lower semi-continuity of $J$.
By the compactness of $\mathcal{V}_\text{ad}$ and the Arzel\`{a}-Ascoli theorem, we see that the convergence $\vect{\phi}^{(n)} \to \vect{\phi}$ already holds.
Hence, it only then remains to show the continuity of \eqref{eq:state_zd} with respect to the domain to complete the proof of compactness of $\mathcal{G}$.
The proof of this continuity is not straightforward, but follows a similar argument in \cite{RabagoAzegami2019b,RabagoAzegami2020} using the tools developed in \cite{BoulkhemairChakib2007,Boulkhemairetal2008}.
%------------------------------------------------------------------
%	 CONTINUITY OF THE STATE PROBLEMS
%------------------------------------------------------------------
\begin{proposition}
	\label{prop:continuity_of_the_state_zd}
	Given the convergence of a sequence of domains stated in \eqref{eq:Omega_sequence_convergence}, we let $\{(\vect{\phi}^{(n)},\zdn)\}$ be a sequence in $\mathcal{G}$ where $\zdn:=\zd(\vect{\phi}^{(n)})$ satisfies \eqref{eq:weak_vd} on $\Omega^{(n)}:=\Omega(\vect{\phi}^{(n)}) \subset {\mathcal{O}}_\text{ad}$ (i.e., $\{\vect{\phi}^{(n)}\} \subset \mathcal{V}_\text{ad}$).
	Then, there exists a subsequence $\{(\vect{\phi}^{(k)},\zdk)\}$ of $\{(\vect{\phi}^{(n)},\zdn)\}$ and elements $\vect{\phi} \in \mathcal{V}_\text{ad}$ and $\zd \in H^{1}(D)$ such that
	$\vect{\phi}^{(k)} \to \vect{\phi}$ and $\ezdk \rightharpoonup \zd$ in $H^{1}(D)$,
	where $\zd=\zd(\vect{\phi})=\ezd|_{\Omega(\vect{\phi})}$ uniquely satisfies \eqref{eq:weak_vd} on $\Omega:=\Omega(\vect{\phi})$.
\end{proposition}
%%%
The proof of the proposition -- which we postpone a bit further below -- relies on three key results listed in the next lemma.
%------------------------------------------------------------------
%	 LEMMA: ESTIMATES
%------------------------------------------------------------------
\begin{lemma}
	\label{lem:estimates}
	We have the following results.
	\begin{itemize}
		\item[(i)] For every $u \in H^{1}_{\Sigma,0}(\Omega)$ and $\Omega \in \tilde{\mathcal{O}}_\text{ad}$, it holds that $\|u\|_{L^{2}(\Omega)} \lesssim |u|_{H^{1}(\Omega)}{{=\|\nabla{u}\|_{L^{2}(\Omega)}}}$\footnote{Here, and in the rest of the discussion, the notation ``$\lesssim$'' means that if $x \lesssim y$, then we can find some constant $c > 0$ such that $x \leqslant c y$. Of course, $y \gtrsim x$ is defined as $x \lesssim y$.}.
		%%%
		\item[(ii)] Let $q \in (\frac12,1]$. Then, for all $\vect{\phi} \in \mathcal{V}_\text{ad}$ and $u \in H^{1}(D)$, we have $\|u\|_{L^{2}(\Gamma(\vect{\phi}))}\lesssim \|u\|_{H^q(D)}$, where $\|\cdot\|_{L^{2}(\Gamma(\vect{\phi}))}$ is the $L^{2}(\Gamma(\vect{\phi}))$-norm and $\|\cdot\|_{H^q(D)}$ denotes the $H^q(D)$-norm.
		%%%
		\item[(iii)] there exists an extension $\ezdn$ of $\zdn$ from $\Omega^{(n)}$ to $D$, and a constant $C_{D}>0$ independent of $n$ such that $\|\ezdn\|_{H^{1}(D)} \leqslant C_{D}$.
	\end{itemize}
\end{lemma}
Supported by the fact that the admissible domains satisfy the uniform cone property, Lemma \ref{lem:estimates}(i) issues a uniform Poincar\'{e} inequality proved in \cite[Cor. 3(ii)]{BoulkhemairChakib2007}.
Lemma \ref{lem:estimates}(ii), on the other hand, is related to the uniform continuity of the trace operator with respect to the domain established in \cite[Thm. 4]{Boulkhemairetal2013}, and Lemma \ref{lem:estimates}(iii) is about an extension of the state variable from $\Omega^{(n)}$ to $D$ whose ${H^{1}(D)}$-norm is bounded above by a positive constant.
This guarantees the existence of subsequence $\{ \ezdn \}$ which weakly converges in $H^{1}(D)$ to a limit denoted by $\ezd$.
Thus, the proof of Proposition \ref{prop:continuity_of_the_state_zd} is completed by showing that the restrictions of $\ezd$ in $\Omega(\vect{\phi})$ is in fact the unique solution to \eqref{eq:weak_vd}. 

In connection to Lemma \ref{lem:estimates}(ii), we note the compactness of the injection $H^{1}(D)$ into $H^q(D)$ for $q \in (1/2, 1)$, i.e., we have
\begin{equation}
	\label{eq:injection}
	 H^{1}(D) \xhookrightarrow{\text{compact}} H^q(D), \qquad \text{for $\frac12 < q < 1$}.
\end{equation}

Meanwhile, the proof of the third statement of Lemma \ref{lem:estimates} which make use of the first two estimates given in the lemma as seen in the following argumentations.
%%% PROOF OF LEMMA 4(iii)
\begin{proof}[Proof of Lemma \ref{lem:estimates}]
The notation $(\cdot)^{(n)}$ is here (and throughout the proof) understood as $(\cdot)(\vect{\phi}^{(n)})$ .
We first show the boundedness of $\{\|\ezdn \|_{H^{1}(D)} \}$. 
By a result of Chenais \cite{Chenais1975} (i.e., the uniform extension property), the solution $\zdn$ of \eqref{eq:weak_vd} on $\Omega^{(n)}$ admits an extension $\ezdn$ in $H^{1}(D)$ such that 
\begin{equation}\label{eq:extension_to_D}
	\|\ezdn\|_{H^{1}(D)} \lesssim \|\zdn\|_{H^{1}(\Omega^{(n)})}.
\end{equation}
	Thus, we only need to find a uniform bound for $\|\zdn\|_{H^{1}(\Omega^{(n)})}$ with respect to $n$.
	To do this, we take $\varphi = \zdn \in H_{\Sigma,0}^1(\Omega^{(n)})$ in \eqref{eq:weak_vd}, and note of the assumption that $\alpha(x)$ is a function that can be extended in $D$ such that (denoted by the same notation) $\alpha(x) \geqslant \alpha_{0} > 0$, for all $x \in D$, to obtain
	$\vertiii{\zdn}_{\Omega^{(n)}}^{2} \lesssim a(\zdn,\zdn) 
			= -a(\udo, \zdn)
			\lesssim \vertiii{\zdn}_{\Omega^{(n)}} \vertiii{\udo}_{\Omega^{(n)}}
			\lesssim \vertiii{\zdn}_{\Omega^{(n)}} \vertiii{\udo}_{D}$,
	where $\vertiii{\cdot}_{\Omega^{(n)}}:=\| \nabla \cdot \|_{L^{2}(\Omega^{(n)})} + \| \cdot \|_{L^{2}(\Gamma^{(n)})}$.
	Note here that we can bound $\|\zdn \|_{L^{2}(\Gamma^{(n)})}$ by $|\zdn |_{H^{1}(\Omega^{(n)})}$. 
	Indeed, from Lemma \ref{lem:estimates}(i)--(ii) and \eqref{eq:extension_to_D}, we have
	$
		\|\zdn \|_{L^{2}(\Gamma^{(n)})}
		\lesssim \|\ezdn \|_{H^{1}(D)} 
		\lesssim \|\zdn \|_{H^{1}(\Omega^{(n)})}
		\lesssim |\zdn |_{H^{1}(\Omega^{(n)})}.
	$
	By these inequalities, together with the fact that $\vertiii{\cdot}_{\Omega}$ is a norm on $H^{1}$ over the (open and bounded) set $\Omega$ equivalent to the natural norm \cite[Appx. A, Prop. 2, p. 15]{Meftahi2017}, we get the uniform estimate $\| \zdn \|_{H^{1}(\Omega^{(n)})} \lesssim \| \udo \|_{H^{1}(D)}$.
	Noting that $\udo$ is a fixed function, we then achieve the boundedness of $\{\|\ezdn \|_{H^{1}(D)} \}$, completing the proof of the lemma.
\end{proof}
%
%
%
%
%--------------------------  PROOFS --------------------------
%--------------------------  PROOFS --------------------------
%--------------------------  PROOFS -------------------------- 
We now proceed to the demonstration of Proposition \ref{prop:continuity_of_the_state_zd}.
\begin{proof}[{{Proof of Proposition \ref{prop:continuity_of_the_state_zd}}}]
	Given the assumptions of the proposition and Lemma \ref{lem:estimates}(ii), there exists an element $\ezd$ in $H^{1}(D)$ and a subsequence $\{\ezdk\}$ of $\{\ezdn\}$ such that the weak convergence $\ezdk \rightharpoonup \ezd$ in $H^{1}(D)$ holds.
	
	First, the statement that $\zd=\ezd|_{\Omega(\vect{\phi})}$ is in $H_{\Sigma,0}^{1}(\Omega(\vect{\phi}))$ follows from the boundedness of the trace operator.
	Because the fixed boundary is Lipschitz, the trace operator $(\cdot )|_{\Sigma}:H^{1}(D) \to L^{2}(\Sigma)$ is compact, so it takes weakly convergent sequences into strongly convergent sequences.
	From this, we can infer that $\lim_{k\to\infty} \ezdk|_{\Sigma} = \ezd|_{\Sigma}$ in $L^{2}(\Sigma)$.
	Now, since $\ezdk|_{\Omega^{(k)}} = \zdk$, then $\zd|_{\Sigma} = \lim_{k\to\infty}\ezdk|_{\Sigma} = \lim_{k\to\infty}\zdk|_{\Sigma} = 0$, and so $\zd \in H_{\Sigma,0}^{1}(\Omega(\vect{\phi}))$.
	
	Next, we will show that $\zd(\vect{\phi})=\ezd|_{\Omega(\vect{\phi})}$ is the solution of \eqref{eq:weak_vd} on $\Omega(\vect{\phi})$.
	To this end, we will prove that the variational equation
	\begin{equation}\label{eq:weak_vd_Omegaphi}
	\begin{aligned}
	\int_{\Omega(\vect{\phi})}{\nabla \zd \cdot \nabla v \, d x}
		+ \int_{\Gamma(\vect{\phi})}{\alpha \zd v\, d s}
		 = -\int_{\Omega(\vect{\phi})}{\nabla \udo \cdot \nabla v \, d x }
			- \int_{\Gamma(\vect{\phi})}{ \alpha \udo v\, d s}, \quad \forall v \in H_{\Sigma,0}^{1}(\Omega(\vect{\phi})),
	\end{aligned}
	\end{equation}
	also holds for all test functions $v \in H_{\Sigma,0}^{1}(D)=\{v\in H^{1}(D) \mid v=0 \ \text{on}\ \Sigma\}$.
	Note that the restriction on $\Omega^{(k)}:=\Omega(\vect{\phi}^{(k)})$ of any element $v$ of $H_{\Sigma,0}^{1}(D)$ is in $H_{\Sigma,0}^{1}(\Omega^{(k)})$, for all $k$, which is exactly the test space of \eqref{eq:weak_vd} on $\Omega^{(k)}$.
	Hence, we have
	\begin{equation}\label{eq:weak_vd_Omegaphi_n}
	\begin{aligned}
	&\int_{\Omega(\vect{\phi}^{(k)})}{\nabla \zdk \cdot \nabla v \, d x}
		+ \int_{\Gamma(\vect{\phi}^{(k)})}{ \alpha \zdk v\, d s}\\
		&\qquad\qquad = - \int_{\Omega(\vect{\phi}^{(k)})}{\nabla \udo \cdot \nabla v \, d x }
			- \int_{\Gamma(\vect{\phi}^{(k)})}{ \alpha \udo v\, d s}, \quad \forall v \in H_{\Sigma,0}^{1}(D). 
	\end{aligned}
	\end{equation}
	For the next step, we will obtain \eqref{eq:weak_vd_Omegaphi} from \eqref{eq:weak_vd_Omegaphi_n} by passing to the limit.
	To do this, we look at the difference between equations \eqref{eq:weak_vd_Omegaphi} and \eqref{eq:weak_vd_Omegaphi_n}, and then we let $k$ tends to infinity.
	In the following computations, property $(\text{P2})$ of $\vect{\phi}$ and $\vect{\phi}_k$ will be use several times.
	
	For the difference between the last two corresponding integrals, we have
	\begin{align*}
		\mathbb{I}_{4}&:=\left| \int_{\Gamma(\vect{\phi}^{(k)})}{ \alpha \udo v\, d s} - \int_{\Gamma(\vect{\phi})}{ \alpha \udo v\, d s} \right|\\
		& \lesssim \|\alpha\|_{L^\infty(D)}\Bigg( \|v\|_{L^{2}(\Gamma(\vect{\phi}^{(k)}))} \|\udo \circ \vect{\phi}_k - \udo \circ \vect{\phi}\|_{L^{2}([0,1])}
		+ \|\udo\|_{L^{2}(\Gamma(\vect{\phi}))} \|v \circ \vect{\phi}_k - v \circ \vect{\phi}\|_{L^{2}([0,1])}\\
		&\quad \qquad \qquad  \qquad
			+ \sup_{[0,1]} |\vect{\phi}_k' - \vect{\phi}'| \left\| \udo \right\|_{L^{2}(\Gamma(\vect{\phi}))} \left\| v \right\|_{L^{2}(\Gamma(\vect{\phi}))} \Bigg),	\quad (v \in H_{\Sigma,0}^{1}(D) \subset H^{1}(D)),	
	\end{align*} 
	due to Lemma \ref{lem:estimates}. 
	Since $v \in H_{\Sigma,0}^{1}(D) \subset H^{1}(D)$, then by \cite[Cor. 1]{Boulkhemairetal2008}, we see that 
	\[
		\left\| v \circ \vect{\phi}_k - v \circ \vect{\phi}\right\|_{L^{2}([0,1])} \to 0
		\quad\text{and}\quad
		\|\udo \circ \vect{\phi}_k - \udo \circ \vect{\phi}\|_{L^{2}([0,1])} \to 0,
		\quad \text{as $t\to0$},
	\]
	for any sequence $\{\vect{\phi}^{(k)}\} \subset \{\vect{\phi}^{(n)}\} \subset \mathcal{U}_{\text{ad}}$ and element $\vect{\phi} \in \mathcal{U}_{\text{ad}}$ such that $\vect{\phi}^{(k)} \to \vect{\phi}$ in the sense of \eqref{eq:phi_sequence_convergence}. 
	By this previously mentioned convergence of $\vect{\phi}_k$ to $\vect{\phi}$ in the $\mathcal{C}^{1}([0,1],\mathbb{R}^{2})$-norm, we end up with the limit $\lim_{k\to\infty} \mathbb{I}_{4} = 0$.

	In the same fashion, applying the inequalities and bounds given in Lemma \ref{lem:estimates}, we obtain
		\begin{align*}
		\mathbb{I}_{3}
		&:=\left| \int_{\Gamma(\vect{\phi}^{(k)})}{ \alpha \zdk v\, d s} - \int_{\Gamma(\vect{\phi})}{ \alpha \zd v\, d s} \right|\\
		& \lesssim \|\alpha\|_{L^\infty(D)}\Bigg( \|v\|_{L^{2}(\Gamma(\vect{\phi}^{(k)}))}  \|\zdk \circ \vect{\phi}_k - \zdk \circ \vect{\phi}\|_{L^{2}([0,1])} 
			+ \|v\|_{L^{2}(\Gamma(\vect{\phi}^{(k)}))}   \|\zdk - \zd \|_{L^{2}(\Gamma(\vect{\phi}))}\\
		&\quad \quad \qquad\qquad +\|\zd\|_{L^{2}(\Gamma(\vect{\phi}))} \|v \circ \vect{\phi}_k - v \circ \vect{\phi}\|_{L^{2}([0,1])}
			+ \sup_{[0,1]} |\vect{\phi}_k' - \vect{\phi}'| \left\| \zd \right\|_{L^{2}(\Gamma(\vect{\phi}))} \left\| v \right\|_{L^{2}(\Gamma(\vect{\phi}))} \Bigg)\\ 
		& \lesssim \|\alpha\|_{L^\infty(D)}\Bigg( \|v\|_{H^{1}(D)}  \|\zdk \circ \vect{\phi}_k - \zdk \circ \vect{\phi}\|_{L^{2}([0,1])}
			+ \|v\|_{H^{1}(D)} \|\ezdk - \ezd \|_{H^q(D)} \\
		&\quad \quad \qquad \qquad
			 + \|\ezd\|_{H^{1}(D)} \|v \circ \vect{\phi}_k - v \circ \vect{\phi}\|_{L^{2}([0,1])}
			 + \sup_{[0,1]} |\vect{\phi}_k' - \vect{\phi}'| \left\| \ezd \right\|_{H^{1}(D)} \left\| v \right\|_{H^{1}(D)} \Bigg).
	\end{align*} 
	Because $\zd = \ezd|_{\Omega},\ \zdk = \ezdk|_{\Omega^{(k)}} \in H^{1}(D)$, the first and the third summands in the right side of the inequality above both disappear due to Lemma \ref{lem:estimates}(ii) combined with the application of \cite[Cor. 1]{Boulkhemairetal2008}.
	Similarly, the second summand also goes to zero as $k \to \infty$ because of Lemma \ref{lem:estimates}(ii) and \eqref{eq:injection}.
	Lastly, because of the convergence given in \eqref{eq:phi_sequence_convergence}, the fourth summand in the last inequality above also tends to zero as $k \to \infty$.
	Hence, we also have the limit $\lim_{k\to \infty} \mathbb{I}_{3}=0$.
	
	Now, for the remaining differences on domain integrals, we have
	\begin{align*}
		\mathbb{I}_{1}
			&= \int_{D}{\chi_{\Omega}( \nabla \ezdk - \nabla \ezd) \cdot \nabla v\, d x}
				+ \int_{D}{(\chi_{\Omega^{(k)}} - \chi_{\Omega})\nabla \ezdk \cdot \nabla v \, d x}, \\
		\mathbb{I}_{3} 
			&= \int_{D}{(\chi_{\Omega^{(k)}} - \chi_{\Omega})\nabla \udo \cdot \nabla v \, d x},
	\end{align*}
	and the desired limits $\lim_{k\to \infty} \mathbb{I}_{1} = \lim_{k\to \infty} \mathbb{I}_{2} =0 $ are achieved by applying the $H^{1}(D)$-weak convergence $\ezdk \rightharpoonup \ezd$, the convergence of characteristic functions (see, e.g., \cite[Prop. 2.2.28, p. 45]{HenrotPierre2018} and \cite{Pironneau1984})
	$\chi_{\Omega^{(k)}} \to \chi_{\Omega}$ in $L^\infty(D)$-weak$^\ast$ (see \eqref{eq:weak_star_convergence_of_characteristic_functions}),
	and the fact that the sequence $\{\|\ezdk\|_{H^{1}(D)}\}$ is bounded.
	This proves that $\zd(\vect{\phi})=\ezd|_{\Omega(\vect{\phi})}$ is the solution of \eqref{eq:weak_vd} on $\Omega(\vect{\phi})$.
\end{proof}
%
%
%
%%%%%%%%%%%%%%%%%%%%%%%%%%%%%%%%%%%%%%%%%%%%%%%%%%%%%%%%%%%%%%
%%% LOWER SEMI-CONTINUITY OF THE COST FUNCTION J
%
%
%
With Proposition \ref{prop:continuity_of_the_state_zd} established, we are now set to prove the second part of Theorem \ref{prop:existence_of_optimal_solution} by proving Proposition \ref{prop:lower_semi_continuity_of_JN}. In the proof, however, we need stronger assumptions to establish specific estimates appearing in the argumentation. While the proof of Proposition \ref{prop:continuity_of_the_state_zd} only requires the weak convergence of $\ezdn$ in $H^{1}(D)$, the proof of the next result needs the weak convergence $\ezdn \to \ezd$ to hold in the $H^{2}(D)$ sense.
For this reason, the properties of the admissible sets of functions parametrizing the free boundary $\Gamma$ given in Definition \ref{def:assumptions} need to be modified accordingly. As a consequence, after imposing the necessary assumptions and definitions, the results presented in Proposition \ref{prop:continuity_of_the_state_zd} and also in Lemma \ref{lem:estimates} -- particularly the convergences and estimates -- can be shown to also hold in the $H^{2}$ sense. We omit the exact re-statements for these results and their corresponding proofs for economy of space.
To be precise with the most important part, however, we simply underline here that we assume $\vect{\phi}^{(n)} \to \vect{\phi}$ in the $\mathcal{C}^{2}([0,1])$-topology and $\ezdn \to \ezd$ weakly in $H^{2}(D)$.
\begin{proposition}
	\label{prop:lower_semi_continuity_of_JN}
	The shape functional $J_{N}(\Omega) = \frac{1}{2} \intS{(\dn{\ud(\Omega)} - g)^{2}}$, where $\ud(\Omega)$ solves \eqref{eq:state_ud} in $\Omega$, is lower semi-continuous on $\mathcal{G}$ in the topology induced by \eqref{eq:convergence} where the convergences defined by \eqref{eq:phi_sequence_convergence} and \eqref{eq:Omega_sequence_convergence} are assume to hold in the $\mathcal{C}^{2}([0,1])$-topology and in $H^{2}(D)$-weakly, respectively.
\end{proposition}
Before proving the proposition, we once again recall that for any admissible domain $\Omega$ of class $\mathcal{C}^{1,1}$, we have $\nu\in C^{0,1}(N^{\varepsilon}) \subset W^{1,\infty}(N^{\varepsilon}) \subset H^{1}(N^{\varepsilon})$, where $N^{\varepsilon}$ is a small neighborhood of $\partial\Omega$ (see, e.g., \cite[Sec. 7.8]{DelfourZolesio2011}). 
It follows that $\nu$ can be extended to a Lipschitz continuous function $\tilde{\nu}$ in $\overline{\Omega}$, and even to the larger set $\overline{D}$. (The normal vector $\nu$ is even smoother for $\mathcal{C}^{2,1}$ domains).
\begin{proof}[Proof of Proposition \ref{prop:lower_semi_continuity_of_JN}]
	Let $\{(\Omega^{(n)}, \zdn)\}$ be a sequence in $\mathcal{G}$, $\Omega^{(n)} := \Omega(\vect{\phi}^{(n)})$, and assume that $(\Omega^{(n)}, \zdn) \to (\Omega, \zd)$ as $n \to \infty$,
	where $\Omega:=\Omega(\vect{\phi})$ and $(\Omega, \zd):=(\Omega, \zd(\Omega)) \in \mathcal{G}$.
	We denote the $H^{2}(D)$ extension of $\zd$ and $\zdn$ in $D$ by $\ezd$ and $\ezdn$, respectively, and let $\nu^{(n)}$ be the (outward) unit normal vector to $\Omega^{(n)}$.
	For simplicity, in the computations below, we drop the constant $1/2$.
	So, we have the following sequence of inequalities
	\begin{align*}
		h_{N}
		:= \left| \sqrt{J_{N}(\Omega^{(n)})} - \sqrt{J_{N}(\Omega)} \right| 
		&= \left|  \left\| \dn{\ud(\Omega^{(n)})}  - g \right\|_{L^{2}(\Sigma)} - \left\| \dn{\ud(\Omega)} - g  \right\|_{L^{2}(\Sigma)} \right|\\ 
		&\leqslant \left\| \dn{\ud(\Omega^{(n)})} - \dn{\ud(\Omega)} \right\|_{L^{2}(\Sigma)} \\ 
		&\leqslant \left\| \nabla \udn \cdot (\nu^{(n)} - \nu) \right\|_{L^{2}(\Sigma)} + \left\| \nabla (\udn - \ud) \cdot \nu \right\|_{L^{2}(\Sigma)}\\
		&\leqslant \left\| \nabla (\udn - \ud) \cdot \nu \right\|_{L^{2}(\Sigma)},
	\end{align*} 
	where the last inequality follows from the fact that the outward unit normal $\nu^{(n)} $ and $\nu$ coincide on the fixed boundary $\Sigma$.
	Further estimation gives us
	\[
	h_{N} \leqslant \left\| \nabla (\udn - \ud) \cdot \nu \right\|_{L^{2}(\Sigma)}
		\lesssim \left\| \eudn - \eud \right\|_{H^{q}(D)}, \quad q \in (3/2,2).
	\]
	In above computations, we have used Assumption \eqref{eq:assumption_on_fixed_boundary}, the property of the trace operator for $H^{2}$ functions (or the trace inequality), and the definition of $H^{q}(D)$ norm.
	We conclude by applying the compactness of the injection of $H^{2}(D)$ into $H^{q}(D)$ for $q \in (3/2,2)$.
	%%%%% ADAMS-FOURNIER, SOBOLEV SPACES, 2003
\end{proof}
%%-------------------------------------------------------------------------------------------------------------------------------------------------------------------------------
Before we end the section, and for the sake of completeness, we formally provide the proof of Theorem \ref{prop:existence_of_optimal_solution} using Proposition \ref{prop:continuity_of_the_state_zd} and Proposition \ref{prop:lower_semi_continuity_of_JN}.
\begin{proof}[Proof of Theorem \ref{prop:existence_of_optimal_solution}]
	\sloppy Let $\{ (\Omega^{(n)},\zdn)\} := \{ (\Omega(\vect{\phi}^{(n)}),\zd(\Omega(\vect{\phi}^{(n)}))\} $ be a minimizing sequence of the shape functional $J_{N}$;
	that is, we let $(\Omega^{(n)},\zdn)$ be such that 
		$
			\lim_{n\to\infty} J_{N}(\Omega^{(n)},\zdn) = \inf\{J_{N}(\hat{\Omega},\zd) \mid (\hat{\Omega},\zd)\in\mathcal{G}\}.
		$
	From a similar argumentation in proving Proposition \ref{prop:continuity_of_the_state_zd}, it can be shown that there exists a subsequence $\{ (\Omega^{(k)},\zdk) \}$ of $\{ (\Omega^{(n)},\zdn) \}$
	and an element $\Omega:=\Omega(\vect{\phi}) \in \mathcal{O}_\text{ad}$\footnote{Here, it must be noted that the definition of the set $\mathcal{O}_\text{ad}$ given in Definition \ref{def:assumptions} is provided with the necessary additional regularity conditions.} such that $\Omega^{(k)} \to \Omega$ (i.e., $\vect{\phi}^{(k)} \to \vect{\phi}$ uniformly in the $\mathcal{C}^{2}$ topology), $\ezdk \rightharpoonup \ezd$ in $H^{2}(D)$, and the function $\ezd|_{\Omega}$ is the unique weak solution to \eqref{eq:weak_vd} in $\Omega$.
	By these results, together with Proposition \ref{prop:lower_semi_continuity_of_JN}, we conclude that -- by \cite[Thm. 2.10]{HaslingerMakinen2003} --
	$
		J_{N}(\Omega,\ezd|_{\Omega})
		= \lim_{k \to \infty} J_{N}(\Omega^{(k)},\zdk)
		= \inf\{J_{N}(\hat{\Omega},\zd):(\hat{\Omega},\zd)\in\mathcal{G}\}.
	$ 
\end{proof} 
}}
%
%
%
%-------------------------------------- TRACKING DIRICHLET DATA APPROACH  -------------------------------------- 
%-------------------------------------- TRACKING DIRICHLET DATA APPROACH  -------------------------------------- 
%-------------------------------------- TRACKING DIRICHLET DATA APPROACH  --------------------------------------  
\subsection{Tracking the Dirichlet data in least-squares approach}\label{subsec:tracking_the_Dirichlet-data}
In this subsection, we present the narrative of tracking Dirichlet data and provide analogous results from previous sections for this case.
%-------------------------------------- THEORETICAL RESULTS  --------------------------------------
%-------------------------------------- THEORETICAL RESULTS  --------------------------------------
%-------------------------------------- THEORETICAL RESULTS  --------------------------------------
\subsubsection{{{Shape optimization formulation}}}
We start by stating that the original problem can also be posed in the following format.
\begin{problem}\label{problem:main_problem}
	Given the Neumann data $g$ on $\Sigma$ and the measured Dirichlet data
	\[
		f:=  u \qquad\text{on $\Sigma$},
	\]
	where $u$ solves the system of PDEs
	\begin{equation}
        \label{equa:state}
        		  - \Delta u	 = 0 	\quad\text{in\ $\Omega$},\qquad
        			      \dn{u} = g	\quad\text{on\ $\Sigma$},\qquad
        	\dn{u} + \alpha u = 0	\quad\text{on\ $ \Gamma $},
        \end{equation}
	determine the shape of the unknown portion of the boundary $\Gamma$.
\end{problem}
In above formulation, we emphasize that $f \in H^{5/2}(\Sigma)$ is seen to be an admissible boundary measurement corresponding to the input flux $g \in  H^{3/2}(\Sigma)$.
In other words, $f$ belongs to the image of the Neumann-to-Dirichlet map $\Upsilon_{\Sigma}: g \in H^{3/2}(\Sigma) \mapsto f=u \in H^{5/2}(\Sigma)$, where $u$ solves \eqref{equa:state}. 
Then, accordingly, one can consider the following minimization problem which consists of tracking the Dirichlet data on the accessible boundary in $L^{2}$ sense:
\begin{equation}\label{equa:shop}
	J_{{D}}(\Omega) := J_{{D}}(D\setminus\overline{\omega}) = \frac12 \intS{\left(\un- f \right)^{2}} \to\ \inf,
\end{equation}
where $\un$ solves the following well-posed systems of PDEs
%%%
\begin{equation}
\label{equa:state_un}
	- \Delta\un 		= 0		\quad\text{in $\Omega$},\qquad
	\dn{\un} 	 		= g		\quad\text{on $\Sigma$},\qquad
	\dn{\un} + \alpha\un= 0	\quad\text{on $\Gamma$},
\end{equation}
whose variational formulation reads as follows:
	\begin{equation}\label{equa:state_weak_form_un}
		\text{Find ${\un} \in H^{1}(\Omega)$ such that  $a(\un,\psi) = \intS{g \psi}$, for all $\psi \in H^{1}(\Omega)$},
	\end{equation}
where $a$ is the bilinear form given in \eqref{eq:bilinear_form}.	
For $g \in H^{-1/2}(\Sigma)$ and Lipschitz $\Omega$, the existence of unique weak solution $\un \in H^{1}(\Omega)$ to \eqref{equa:state_weak_form_un} follows from Lax-Milgram lemma.

Problem \eqref{equa:shop} is equivalent to \eqref{eq:gip} provided we have a perfect match of boundary data on the unknown boundary, meaning $u = f$ and $\dn{u} = g$ on $\Sigma$.
Furthermore, the shape optimization problem \eqref{equa:shop} has a solution for $\mathcal{C}^{1,1}$ regular domains, as shown in Theorem \ref{theorem:main_result}. 
This claim is rigorously justified in the next subsection under a topology induced by a different definition of convergence of domains. 
This approach allows us to prove the existence of an optimal solution to \eqref{equa:shop} in both two and three spatial dimensions.
%--------------------------  EXISTENCE OF OPTIMAL SOLUTION  --------------------------
%--------------------------  EXISTENCE OF OPTIMAL SOLUTION  --------------------------
%--------------------------  EXISTENCE OF OPTIMAL SOLUTION  --------------------------
%--------------------------  EXISTENCE OF OPTIMAL SOLUTION  --------------------------
%--------------------------  EXISTENCE OF OPTIMAL SOLUTION  --------------------------
\subsubsection{{{Existence of a shape solution}}}
\label{subsec:existence_tracking_the_Dirichlet_data}
Here, we address the question of the existence of an optimal solution to \eqref{equa:shop}. 
To achieve this, it suffices to assume that the entire boundary of any admissible domain (to be elaborated later) is $\mathcal{C}^{1,1}$ regular. 
Unless otherwise specified, this regularity will be imposed on any (non-empty, open, and bounded) set considered in this section. 
On some occasions, the aforementioned regularity is explicitly stated for clarity.

To proceed, we rewrite \eqref{equa:shop} as follows:
\begin{equation}
\label{equa:new_shape_opti}
	\min_{\Omega} J_{{D}}(\Omega) := \min_{\Omega} J_{{D}}(\Omega,\un(\Omega)) = \min_{\Omega} \left\{ \frac{1}{2} \int_{\Omega} (\un(\Omega) - f)^{2} \right\},
\end{equation}
where $\un(\Omega)$ is subject to \eqref{equa:state_un}. 
Since \eqref{equa:state_un} has a unique solution, we can define the map $\Omega \ \mapsto \ \un:=\un(\Omega)$, and denote its graph by
\begin{align*}
	\mathcal{G}&=\{(\Omega,\un): \text{ $\Omega \in \mathcal{O}_\text{ad}$ and $\un$ solves \eqref{equa:state_un}}\},
\end{align*}
where $\mathcal{O}_\text{ad}$ is defined further below in \eqref{equa:admissible_set_of_domains}.

Our main result of this subsection is the following.
\begin{theorem}\label{theorem:main_result}
	The minimization problem \eqref{equa:new_shape_opti} has at least one solution.
\end{theorem}
Again, to demonstrate the validity of the aforementioned assertion, we first need to endow the set $\mathcal{G}$ with a topology for which $\mathcal{G}$ is compact and $J_{{D}}$ is lower semi-continuous.
Due to the presence of the Neumann boundary condition on the exterior boundary $\Sigma$, we cannot apply the same technique used in subsection \ref{subsection:existence_tracking_Neumann_data} since it relies on Lemma \ref{lem:estimates}. 
Therefore, we employ a different approach to prove the statement. Instead of introducing a topology induced by the convergence \eqref{eq:convergence}, we introduce a topology on $\mathcal{G}$ induced by the Hausdorff convergence $\Omega^{(n)} \stackrel{H}{\longrightarrow} \Omega$. 
In this way, we can even prove the existence of the optimal solution to \eqref{equa:shop} in arbitrary dimensions ($d\in\{2,3\}$).

Considering the point discussed above, we will now briefly review the definitions of the Hausdorff distance, Hausdorff convergence, and the $\varepsilon$-cone property. 
For elaboration on these concepts, readers are directed to \cite[Ch. 3]{Pironneau1984}.
\begin{definition}
[{\cite[Def. 2.2.7, p. 30]{HenrotPierre2018}}]
\label{def:Hausdorff_distance}
	Let ${{\omega}}_{1}$ and ${{\omega}}_{2}$ be two (compact) subsets of $\mathbb{R}^{d}$, $d \geqslant 2$.
	The Hausdorff distance $d_{H}({{\omega}}_{1},{{\omega}}_{2})$ between ${{\omega}}_{1}$ and ${{\omega}}_{2}$ is defined as follows $d_{H}({{\omega}}_{1},{{\omega}}_{2}) = \max\{\rho({{\omega}}_{1},{{\omega}}_{2}),\rho({{\omega}}_{2},{{\omega}}_{1})\}$
	where $\rho({{\omega}}_{1},{{\omega}}_{2}) = \sup_{s\in{{\omega}}_{1}} d(x, {{\omega}}_{2})$ and $d(x,{{\omega}}_{2}) = \inf_{y \in {{\omega}}_{2}} |x-y|$.
	Note that $d_{H}$ defines a topology on the closed bounded sets of $\mathbb{R}^{d}$.
\end{definition}
\begin{definition}[{\cite[Def. 2.2.8, p. 30]{HenrotPierre2018}}]
\label{def:Hausdorff_convergence}
	Let $\{{{\omega}}^{(n)} \}$ and ${{\omega}}$ be open sets included in $D \subset \mathbb{R}^{d}$, $d \geqslant 2$.
	We say that the sequence ${{\omega}}^{(n)}$  converges in the sense of Hausdorff to ${{\omega}}$ if $d_{H}(D\setminus{{\omega}}^{(n)}, D\setminus{{\omega}})  \longrightarrow 0$ as $n \longrightarrow \infty$.
	We will denote this convergence by ${{\omega}}^{(n)} \stackrel{H}{\longrightarrow} {{\omega}}$ (or simply by ${{\omega}}^{(n)} \longrightarrow {{\omega}}$ when there is no confusion).
\end{definition}
\begin{definition}[{\cite[Def. 2.4.1, p. 54]{HenrotPierre2018}}]
\label{definition:cone_property}
Let $\xi$ be a unitary vector in $\mathbb{R}^{d}$, $d \geqslant 2$, $\varepsilon > 0$ be a real number, and $y \in \mathbb{R}^{d}$.
A cone $C$ with vertex $y$, direction $\xi$, and dimension $\varepsilon$ is the set defined by
\[
	C(y,\xi,\varepsilon) = \{ x\in \mathbb{R}^{d} \mid \langle x-y, \xi\rangle_{\mathbb{R}_{d}} \geqslant \cos(\varepsilon) \|x-y\|_{\mathbb{R}_{d}} \ \ \text{and} \ \ 0 < \|x-y\|_{\mathbb{R}_{d}} < \varepsilon \},
\]
where $\langle \cdot,\cdot\rangle_{\mathbb{R}_{d}}$ is the Euclidean scalar product of $\mathbb{R}^{d}$ and $\|\cdot\|_{\mathbb{R}^{d}}$ is the associated euclidean norm.

An open bounded set $\Omega \subset \mathbb{R}^{d}$ satisfies the $\varepsilon$-cone property, if for $x \in \partial \Omega$, there exists a unitary vector $\xi_{x} \in \mathbb{R}^{d}$ such that for all $y \in \overline{\Omega} \cap B_{\varepsilon}(x)$, we have $C(y,\xi,\varepsilon) \subset \Omega$, where $B_{\varepsilon}(x)$ denotes the open ball with center $x$ and radius $\varepsilon$.
\end{definition}
In light of the definitions provided above, we hereby assert the ensuing proposition, pivotal in substantiating the proof of Theorem \ref{theorem:main_result}.
\begin{proposition}[{\cite[Thm. 2.4.7, p. 56]{HenrotPierre2018}}]
	An open bounded set $\Omega \subset \mathbb{R}^{d}$ has the $\varepsilon$-cone property if and only if it has a Lipschitz boundary.
\end{proposition}
The set of admissible domains $\mathcal{O}_{\text{ad}}$ is defined as follows:
\begin{equation}\label{equa:admissible_set_of_domains}
\mathcal{O}_{\text{ad}} = \{\Omega = D \setminus \overline{\omega} \mid D \in \mathcal{C}^{1,1}, \ \omega \in \Wad, \ \text{$\Omega$ is an open, connected, bounded set}  \}.
\end{equation}
Here, $\Wad$ is essentially the set given previously in \eqref{eq:set_Wad}, with the distinction that we relax the regularity requirement on the domains. 
Specifically, we only assume that $D$, $\omega$, and consequently $\Omega$, are of class $\mathcal{C}^{1,1}$.
It is essential to emphasize that the elements within $\Wad$ exhibit the $\varepsilon$-cone property. 
Furthermore, we recall from Remark \ref{rem:convergence_of_domains_and_boundaries} that given a sequence $\Omega^{(n)}$ of open sets in $\mathcal{O}\text{ad}$, there exists an open set $\Omega \in \mathcal{O}_\text{ad}$ and a subsequence $\Omega^{(k)}$ that converges to $\Omega$ in the Hausdorff sense. 
Moreover, both $\Omega^{(k)}$ and $\partial \Omega^{(k)}$ converge, respectively, in the Hausdorff sense to $\Omega$ and $\partial \Omega$. 
Additionally, these convergences extend in the sense of characteristic functions and in the sense of compacts as well \cite[Thm 2.4.10, p. 59]{HenrotPierre2018}.

In the proof of Proposition \ref{prop:continuity_of_the_state_zd} given in subsection \ref{subsection:existence_tracking_Neumann_data}, we have already mentioned in passing that for any sequence of measurable sets $\Omega^{(n)}$, the corresponding sequence of characteristic functions $\chi_{\Omega^{(n)}}$ is weakly-$\ast$ relatively compact in $L^{\infty}(\mathbb{R}^{d})$.
This means that we can find an element $\chi \in L^{\infty}(\mathbb{R}^{d})$ and a subsequence $\{\Omega^{(k)}\}_{k \geqslant 0} \subset \{\Omega^{(n)}\}_{n \geqslant 0}$ such that
\begin{equation}\label{eq:weak_star_convergence_of_characteristic_functions}
	\text{for all $\psi \in L^{1}(\mathbb{R}^{d})$}, 
	\quad \lim_{k\to \infty} \int_{\mathbb{R}^{d}} \chi_{\Omega^{(k)}} \psi \, d{x}
	=  \int_{\mathbb{R}^{d}} \chi_{\Omega} \psi \, d{x}.
\end{equation}
\begin{remark}[{\cite[p. 27]{HenrotPierre2018}}]
\label{rem:limit_is_not_a_characteristic_function}
We point out here that the limit $\chi$ is not, in general, a characteristic function.
It only takes values between $0$ and $1$ \cite[Prop. 2.2.28, p. 45]{HenrotPierre2018}.
Nevertheless, the limit will be a characteristic function if the convergence is ``strong'' in the sense that it takes place in $L_{loc}^{p}$ for some $p \in [1, \infty)$. 
Indeed, it is then possible to extract a subsequence that converges almost everywhere.
Hence, in the limit, $\chi$ takes only the values $0$, $1$ and it coincides with the characteristic function of the set where it takes the value $1$.
\end{remark}
From the previous paragraph and remark, we observe that the weak-$\ast$ limit is a characteristic function only when the convergence is strong, as precisely stated in the following proposition.
\begin{proposition}[{\cite[Prop. 2.2.1, p. 27]{HenrotPierre2018}}]\label{prop:local_convergence_of_characteristic_functions_in_Lp}
If $\{\Omega^{(n)}\}_{n\geqslant 0}$ and $\Omega$ are measurable sets in $\mathbb{R}^{d}$ such that $\chi_{\Omega^{(n)}}$
weakly-$\ast$ converges in $L^{\infty}(\mathbb{R}^{d}$) in the sense of \eqref{eq:weak_star_convergence_of_characteristic_functions} to $\chi_{\Omega}$, then $\chi_{\Omega^{(n)}} \longrightarrow \chi_{\Omega}$  in $L_{loc}^{p}(\mathbb{R}^{d})$ for any $p<+\infty$ and almost everywhere.
\end{proposition}

As previously established (refer to the statement following Assumption \eqref{eq:assumption_on_fixed_boundary} in subsection \ref{subsection:existence_tracking_Neumann_data}), the set $\Wad$ possesses a significant property regarding its elements. 
Specifically, this property relates to the existence of a uniform extension operator, as precisely stated in the following lemma ({cf. \cite[Eq. (3.83), p. 129]{HenrotPierre2018}}).
\begin{lemma}
	\label{lemma:extension_operator}
	There exists a constant $c>0$ such that for all $\Omega \in {{\mathcal{O}_\text{ad}}}$, there exists a bounded linear extension operator $\mathcal{E}_{\Omega} : H^{m}(\Omega) \longrightarrow H^{m}(D)$ such that $\max_{m=0,1}\{\|\mathcal{E}_{\Omega}\|_{\mathscr{B}^{m}(\Omega)}\} \leqslant c$, where $\|\mathcal{E}_{\Omega}\|_{\mathscr{B}^{m}(\Omega)} = \sup_{v \in H^{m}(\Omega)\setminus\{0\}} \left\{\|\mathcal{E}_{\Omega}v\|_{H^{m}(D)}/\|v\|_{H^{m}(\Omega)}\right\}$ and $\mathscr{B}^{m}:= \mathscr{L}(H^{m}(\Omega),H^{m}(D))$.
\end{lemma}
Note that by the choice of the set $\mathcal{O}_\text{ad}$, every admissible domain enjoys the $\varepsilon$-cone property which is a sufficient condition for the result in Lemma \ref{lemma:extension_operator} to hold.
Now, with the previous lemma at our disposal, we can easily prove the following proposition which is the analog result of Proposition \ref{prop:continuity_of_the_state_zd}.
\begin{proposition}\label{proposition:convergence_to_solution}
	Let the following assumptions be satisfied:
	\begin{description}
		\item[$(\text{A1})$]  $\{\Omega^{(n)}\} \subset {{\mathcal{O}_\text{ad}}}$ is a sequence that converges to $\Omega \in {{\mathcal{O}_\text{ad}}}$ in the Hausdorff sense.
		\item[$(\text{A2})$]  For each $n\in\mathbb{N}$, $\Omega^{(n)} \in \mathcal{O}_\text{ad}$, and $\un^{(n)} \in H^{1}(\Omega^{(n)})$ solves \eqref{equa:state_weak_form_un}.
		\item[$(\text{A3})$]  There is some constant $c^{\star} > 0$ such that, for all $\Omega^{(n)} \in {{\mathcal{O}_\text{ad}}}$, $n \in \mathbb{N}$, the extension operator $\mathcal{E}_{\Omega^{(n)}} : H^{m}(\Omega^{(n)}) \longrightarrow H^{m}(D)$ satisfy, for all $n \in \mathbb{N}$, the inequality condition $\max_{m=0,1}\{\|\mathcal{E}_{\Omega^{(n)}}\|_{\mathscr{B}^{m}(\Omega^{(n)})}\} \leqslant c^{\star}$ (and the same holds for the limit shape $\Omega \in \mathcal{O}_\text{ad}$).
	\end{description}
	Then, the sequence of extensions $\eun^{(n)} := \mathcal{E}_{\Omega^{(n)}}\un^{(n)} \in H^{1}(D)$ converges (up to a subsequence) to a function $\eun$ in $H^{1}(D)$-weak and in $L^{2}(D)$-strong such that $\eun \big|_{\Omega} = \un$ solves \eqref{equa:state_weak_form_un} in $\Omega$.
	Moreover, $\chi_{\Omega^{(n)}} \nabla \eun^{(n)}$ converges strongly in $L^{2}(D)^{d}$ to $\chi_{\Omega}\nabla\eun$.
	In addition, if the extension operators $\{\mathcal{E}_{\Omega^{(n)}}\}$ satisfy the compatibility condition 
	\begin{equation}\label{eq:combatibility_condition_for_the_convergence_of_extension_operators}
		 \mathcal{E}_{\Omega^{(n)}} (\chi_{\Omega^{(n)}} \un) \longrightarrow \eun\quad \text{strongly in $H^{1}(D)$},
	\end{equation}
	then the convergence of $\eun^{(n)}$ to $\eun$ also holds strongly in $H^{1}(D)$.
\end{proposition}
\begin{proof} 
	Let the given assumptions (A1), (A2), and (A3) be satisfied.
    We have
    \begin{equation}\label{eq:varform_per_n}
        a(\un^{(n)},\psi) = \intS{g \psi}, \quad \text{for all $\psi \in H^{1}(\Omega^{(n)})$}.
    \end{equation}
    Taking $\psi = \un^{(n)} \in H^{1}(\Omega^{(n)})$ and using the equivalence between the norm $\vertiii{\cdot}_{\Omega^{(n)}}$ and the usual $H^{1}(\Omega^{(n)})$-Sobolev norm, we obtain the inequality $\|\un^{(n)}\|_{H^{1}(\Omega^{(n)})} \lesssim \|g\|_{L^{2}(\Sigma)}$.
    Using the assumption on the choice of extension operators $\mathcal{E}_{\Omega^{(n)}}$, $n \in \mathbb{N}$, combined with Lemma \ref{lemma:extension_operator}, we get the estimate
    \begin{equation}\label{eq:uniform_bound_for_extension}
        \|\eun^{(n)}\|_{H^{1}(D)} \leqslant \|\mathcal{E}_{\Omega^{(n)}}\|_{\mathscr{B}^{1}(\Omega^{(n)})} \|\un^{(n)}\|_{H^{1}(\Omega^{(n)})} \leqslant c^{\star} \|g\|_{L^{2}(\Sigma)}.
    \end{equation}
    Clearly, from \eqref{eq:uniform_bound_for_extension}, we see that the sequence $\eun^{(n)}$ is bounded in $H^{1}(D)$.
    By the Rellich-Kondrachov and Banach-Alaoglu theorems, we may extract a subsequence $\{\eun^{(k)}\} \subset \{\eun^{(n)}\}$ such that we have weak convergence $\eun^{(k)} \rightharpoonup \eun$ in $H^{1}(D)$ and strong convergence $\eun^{(k)} \rightarrow \eun$ in $L^{2}(D)$, for some element $\eun \in H^{1}(D)$.

	We next show that the limit point $\eun \in H^{1}(D)$ actually solves \eqref{equa:state_weak_form_un} in $\Omega$ (i.e., $\eun \big|_{\Omega} = \un$ where $\un$ solves \eqref{equa:state_weak_form_un}) by passing through the limit and using the pointwise almost everywhere convergence of the characteristic function $\chi_{\Omega^{(n)}}$ to $\chi_{\Omega}$ (i.e., we use the fact that there exists $\chi \in L^{\infty}(D)$ such that $\chi_{\Omega^{(n)}} \stackrel{\ast}{\rightharpoonup} \chi$ in $L^{\infty}(D)$ with $\chi_{\Omega} \leqslant \chi \leqslant 1$).
	In the rest of the proof, we use the fact that $\nu$ can be extended to a Lipschitz continuous function, again denoted by ${\nu}$, on $\overline{\Omega}$, and even to the larger set $\overline{D}$.
	
	Let us first note that for every $n$, each test function $\psi \in H^{1}(\Omega^{(n)})$ admits an extension in $H^{1}(\mathbb{R}^{d})$ and, specifically, in $H^{1}(D)$ -- still we denote by $\psi$ -- by Stein's extension theorem \cite[Thm. 5.24, p. 154]{AdamsFournier2003} and by a result of Chenais \cite{Chenais1975}.
	Therefore, we can also pose the variational problem \eqref{eq:varform_per_n} with the test space $H^{1}(D)$ by considering the following variational equation
	\begin{equation}\label{eq:variational_equation_over_domain_D}
		\mathcal{A}^{(n)}:=\intD{ \chi_{\Omega^{(n)}} \nabla \eun^{(n)} \cdot \nabla \psi} + \intD{ \chi_{\Omega^{(n)}} \operatorname{div}(\alpha \eun^{(n)} \psi \nu)} = \intS{g \psi},
%%		\quad \text{for all $\psi \in {{H^{1}(D)}}$},%
		\quad \forall\psi \in {{H^{1}(D)}},%
	\end{equation}
	where $\nu$ is the unit normal vector to the boundary $\Gamma$, pointing outward from $\Gamma$.
	From Proposition \ref{prop:local_convergence_of_characteristic_functions_in_Lp}, we know that $\chi_{\Omega^{(n)}}$ almost everywhere converges to $\chi_{\Omega}$ in $L^{1}(D)$.
	As a consequence, we get 
	\begin{equation}\label{equa:L2D_convergence}
	\text{$\chi_{\Omega^{(n)}} \nabla{\psi} \longrightarrow \chi_{\Omega}\nabla{\psi}$ \qquad strongly in $L^{2}(D)$.}
	\end{equation}
	Next, let us show that $\eun \big|_{\Omega} = \un$ actually solves \eqref{equa:state_weak_form_un} by proving that
	\[
		\mathcal{A}^{(\infty)} := \intD{ \chi_{\Omega} \nabla \eun \cdot \nabla \psi} + \intD{ \chi_{\Omega} \operatorname{div}(\alpha \eun \psi \nu)} = \intS{g \psi},
%%		\quad \text{for all $\psi \in {{H^{1}(D)}}$}.
		\quad \forall\psi \in {{H^{1}(D)}}.
	\]
	Using \eqref{equa:L2D_convergence}, the weak convergence $\eun^{(n)} \rightharpoonup \eun$ in $H^{1}(D)$, and the weak-$^{\ast}$ convergence $\chi_{\Omega^{(n)}} \stackrel{\ast}{\rightharpoonup} \chi_{\Omega}$ in $L^{\infty}(D)$, we see that $\mathcal{A}^{(n)} \longrightarrow \mathcal{A}^{(\infty)}$.
	Therefore, we have $\mathcal{A}^{(\infty)} = \intS{g \psi}$, for all $\psi \in {{H^{1}(D)}}$,
	or equivalently,
	\[
		\intO{ \nabla \eun \cdot \nabla \psi} + \intO{ \operatorname{div}(\alpha \eun  \psi \nu)} = \intS{g \psi},
		\quad \text{for all $\psi \in {{H^{1}(D)}}$}.
	\]	
	It is not hard to see that this is also valid for all $\psi \in H^{1}(\Omega)$, thanks to the extension property of $\Omega \in \mathcal{O}_\text{ad}$.
	Thus, with the uniqueness of the limit, we conclude that $\eun\big|_{\Omega} = \un$ -- recovering the variational equation in \eqref{equa:state_weak_form_un}.

	We finish the proof by verifying the last two claims in the proposition.
	We emphasize that we can also obtain a uniform estimate for $\|\chi_{\Omega^{(n)}} \nabla \eun^{(n)} \|_{L^{2}(D)^{d}}^{2}$ which can be deduced from \eqref{eq:variational_equation_over_domain_D}.
	Note that by taking $\psi = \eun^{(n)} \in H^{1}(D)$ in the previous variational equation, and since $\chi_{\Omega^{(n)}}^{2} = \chi_{\Omega^{(n)}}$, then we also have the inequality condition
	\[
		\|\chi_{\Omega^{(n)}} \nabla \eun^{(n)}\|_{L^{2}(D)^{d}}^{2}
			\leqslant \intD{ \chi_{\Omega^{(n)}} |\nabla \eun^{(n)}|^{2} }
			+ \intD{ \chi_{\Omega^{(n)}} \operatorname{div}(\alpha (\eun^{(n)})^{2} \nu)} 
			\lesssim \|g\|_{L^{2}(\Sigma) }^{2},
	\]
	where the rightmost inequality is due to \eqref{eq:uniform_bound_for_extension}.
	This gives us the estimate $\|\chi_{\Omega^{(n)}} \nabla \eun^{(n)}\|_{L^{2}(D)^{d}} \lesssim \|g\|_{L^{2}(\Sigma) }$.

	Now, if the extension operators $\{\mathcal{E}_{\Omega^{(n)}}\}$ satisfy the compatibility condition \eqref{eq:combatibility_condition_for_the_convergence_of_extension_operators}, then
	\[
		\lim_{n \to \infty} \mathcal{B}^{(n)} :=
		\lim_{n \to \infty} \intD{ \chi_{\Omega^{(n)}} \nabla \eun^{(n)} \cdot \nabla \eun^{(n)} } = \intD{\chi_{\Omega} \nabla \eun \cdot \nabla \eun }
		=: \mathcal{B}^{(\infty)}
	\]
	because
	\begin{equation}\label{eq:estimate_computation}
	\begin{aligned}
	(\text{LHS}) := \left| \mathcal{B}^{(n)} -  \mathcal{B}^{(\infty)}\right|
%%%	&:=\left| \intD{ \chi_{\Omega^{(n)}} \nabla \eun^{(n)} \cdot \nabla \eun^{(n)}} - \intD{ \chi_{\Omega} \nabla \eun \cdot \nabla \eun} \right|\\
%%%		& \leqslant \left| \intD{ \left( \chi_{\Omega^{(n)}} - \chi_{\Omega} \right) \nabla \eun^{(n)} \cdot \nabla \eun^{(n)}} \right|
%%%			+ \left| \intD{ \chi_{\Omega} ( \nabla \eun^{(n)} \cdot \nabla \eun^{(n)}  - \nabla \eun \cdot \nabla \eun) } \right|\\
		&\ \lesssim \left| \intD{ \left( \chi_{\Omega^{(n)}} - \chi_{\Omega} \right) \nabla \eun^{(n)} \cdot \nabla \eun^{(n)}} \right|\\
		&\qquad	+ \left| \intD{ ( | \nabla \eun^{(n)}|^{2}  - |\nabla \eun|^{2}) } \right|,			
	\end{aligned}
	\end{equation}
	and that we have the convergences $\chi_{\Omega^{(n)}} \stackrel{\ast}{\rightharpoonup} \chi_{\Omega}$ in $L^{\infty}(D)$ and $\eun^{(n)} \to \eun$ strongly in $H^{1}(D)$.
	With respect to the second integral in \eqref{eq:estimate_computation} and the latter convergence, let us note that it holds that
	\[
		\| \mathcal{E}_{\Omega^{(n)}}(\un^{(n)}) -  \mathcal{E}_{\Omega^{(n)}} (\chi_{\Omega^{(n)}} \eun) \|_{H^{1}(D)}
			\lesssim \| \un^{(n)} - \chi_{\Omega^{(n)}} \eun \|_{H^{1}(\Omega^{(n)})}.
	\]
	Evidently, the right-hand side vanishes as $n$ tends to infinity.
	So, due to \eqref{eq:combatibility_condition_for_the_convergence_of_extension_operators}, it follows that $\un^{(n)} \to \un$ strongly in $H^{1}(D)$.
	Hence, we also deduce the convergence $\chi_{\Omega^{(n)}} \nabla \eun^{(n)} \cdot \nabla \eun^{(n)} \to \chi_{\Omega} \nabla \eun \cdot \nabla \eun$ in $L^{1}(D)$.
	Because $\Omega^{(n)}, \Omega \in {{\mathcal{O}_{\text{ad}}}}$, $\Omega^{(n)}$ and $\Omega$ are both (measurable) subsets of $D$, and we have the convergence of $\Omega^{(n)} \to \Omega$ not only in the sense of Hausdorff, but also in the case of characteristic functions (and also in the sense of compacts), then we can extract subsequences $\{\Omega^{(k)}\}$ and $\{\chi_{\Omega^{(k)}} \nabla \eun^{(k)} \cdot \nabla \eun^{(k)}\}$ such that
	\[
		\lim_{k \to \infty} \int_{\Omega^{(k)}}{ \chi_{\Omega^{(k)}} \nabla \eun^{(k)} \cdot \nabla \eun^{(k)} } \, dx = \intO{\chi_{\Omega} \nabla \eun \cdot \nabla \eun }.
	\]
	Indeed, this is immediate from \eqref{eq:estimate_computation} as we have\footnote{We can also deduce from the previous computations that
	$
		\|\chi_{\Omega^{(n)}} \nabla \eun^{(n)} \|_{L^{2}(D)^{d}}^{2} \to \|\chi_{\Omega} \nabla \eun\|_{L^{2}(D)^{d}}^{2}
	$ 
	which implies the strong convergence $\chi_{\Omega^{(n)}} \nabla \eun^{(n)} \to \chi_{\Omega}\nabla\eun$ in $L^{2}(D)^{d}$.
	Additionally, by a similar argument, we can also deduce the strong convergence $\chi_{\Omega^{(n)}} \eun^{(n)} \to \chi_{\Omega} \eun$ in $L^{2}(D)$.}
	\[
	\left| \int_{\Omega^{(k)}}{ \chi_{\Omega^{(k)}} \nabla \eun^{(k)} \cdot \nabla \eun^{(k)} } \, dx - \intO{\chi_{\Omega} \nabla \eun \cdot \nabla \eun } \right| \leqslant (\text{LHS}).
	\]
	This finishes the proof of the proposition.
\end{proof}
To close out this subsection, we provide the proof of Theorem \ref{theorem:main_result}.
\begin{proof}[Proof of Theorem \ref{theorem:main_result}]
	Observe that the infimum of $J_{D}(\Omega)$ is finite.
	Hence, we can find a minimizing sequence $\{\Omega^{(n)}\} \subset {{\mathcal{O}_\text{ad}}}$ which is bounded such that $\lim_{n\to\infty} J_{D}(\Omega^{(n)}) = \inf_{\Omega \in {{\mathcal{O}_\text{ad}}}} J_{D}(\Omega)$.
%%%	Note that the sequence $\{\Omega^{(n)}\} \subset {{\mathcal{O}_\text{ad}}}$ is bounded.
	By Remark \ref{rem:convergence_of_domains_and_boundaries} (cf. \cite[Thm. 2.4.10, p. 59]{HenrotPierre2018}), there exists ${{\Omega}} \in {{\mathcal{O}_\text{ad}}}$, and a subsequence $\{\Omega^{(k)}\} \subset \{\Omega^{(n)}\}$ such that $\Omega^{(k)}$ converges to ${{\Omega}}$ in the sense of Hausdorff (Definition \ref{def:Hausdorff_convergence}).
	Then, with the premise of Proposition \ref{proposition:convergence_to_solution}, we know that the sequence of extensions $\eun^{(n)} := \mathcal{E}_{\Omega^{(n)}}\un^{(n)} \in H^{1}(D)$ (of functions $\un^{(n)} \in H^{1}(\Omega^{(n)})$ which solves \eqref{equa:state_weak_form_un} on each of its respective domain) -- taking a further subsequence if necessary -- converges to (the unique limit) ${{\eun}} \in H^{1}(D)$ where ${{\eun}} \big|_{{{\Omega}}} = \un^{}$ solves \eqref{equa:state_weak_form_un} in ${{\Omega}}$.
	Now, to conclude, it is left to show that the shape functional $J_{D}(\Omega)$ is lower-semicontinuous; that is, we have $J_{D}({{\Omega}}) \leqslant \lim_{k \to \infty} J_{D}(\Omega^{(k)}) = \inf_{{{\hat{\Omega}}} \in {{\mathcal{O}_\text{ad}}}} J_{D}({{\hat{\Omega}}}) \leqslant J_{D}({{\hat{\Omega}}})$.
	From Proposition \ref{proposition:convergence_to_solution}, we know that the map $\Omega \mapsto \un(\Omega)$ is continuous.
%%%	Therefore, the map $\Omega \mapsto \intS{(\un(\Omega) - f)^{2}}$ is also continuous, in particular, it is lower-semicontinuous.
	Therefore, the map $\Omega \mapsto J_{D}(\Omega)$ is also continuous, in particular, it is lower-semicontinuous.	
	This completes the proof of Theorem \ref{theorem:main_result}.
\end{proof}
Having addressed the existence of optimal shape solutions for equations \eqref{eq:shop} and \eqref{equa:shop}, we are now prepared to discuss the numerical solution of these optimization problems. 
A common approach for solving such problems numerically involves employing a gradient-based descent scheme. 
To enable this, we need the shape gradient of the cost functionals, which can be readily computed using shape calculus \cite{DelfourZolesio2011,HenrotPierre2018,MuratSimon1976,Simon1980,SokolowskiZolesio1992}. 
Recently, in \cite{AfraitesRabago2022}, the expression for the shape gradient of both $J_{N}$ and $J_{D}$ was established using a chain rule approach. 
Thus, in the subsequent part of this section, our focus will shift to the issue of ill-posedness of optimization problems \eqref{eq:shop} and \eqref{equa:shop}. 
Following this discussion, we will present a numerical algorithm for solving the minimization problems, which involves utilizing multiple measurements. 
Finally, we will provide some numerical examples.
%--------------------------  SHAPE DERIVATIVES --------------------------
%--------------------------  SHAPE DERIVATIVES --------------------------
%--------------------------  SHAPE DERIVATIVES --------------------------
\section{Second-order analyses and stability issue}\label{sec:second-order_analyses_and_stability_issue}
In \cite{AfraitesRabago2022}, it was observed that $J_{N}$ exhibits insensitivity to perturbations, which is likely due to the ill-posed nature of the optimization problem. 
The primary objective of this subsection is to explore this issue by analyzing the shape Hessian of $J_{N}$ at a critical shape. 
This analysis aims to underscore the ill-posedness inherent in the shape problem \eqref{eq:shop}. 
Furthermore, to address the numerical difficulties arising from this ill-posedness, particularly in identifying concave regions of the unknown interior boundary, we adopt an approach based on multiple data measurements along the accessible boundary. 
This method builds upon the concept of utilizing multiple boundary measurements, as discussed in \cite{AlvesMartinsRoberty2009,GiacominiPantzaTrabelsi2017}, and has recently been proposed in \cite{Fang2022}. 
Instead of resorting to a second-order method, which can be computationally demanding and intricate, especially in higher-dimensional scenarios, we opt for multiple boundary measurements. 
This approach is straightforward to implement and only necessitates knowledge of the first-order derivative of the cost function.

In the following, we briefly present some preliminary concepts related to shape derivatives from shape calculus.

\subsection{Some elements of shape calculus}
Throughout the paper, vectorial functions and spaces are written in bold faces. 
Let us define $D_{\delta}$ as an open set with a $\mathcal{C}^{\infty}$ boundary, such that $\{ x \in D \mid \text{$d(x,\partial D) > \delta/2$}\} \subset D_{{\delta}} \subset \{ x\in D \mid \text{$d(x,\partial D) > \delta/3$}\}$. 
We let $\vect{\mathcal{C}}^{2,1}(\mathbb{R}^{d})$, $d \in \{2,3\}$, be a smooth vector field with compact support in $\overline{D}_{\delta}$ and define $\sfTheta$ as the collection of all such admissible deformation fields. 
We represent the normal component of $\VV$ as $\Vn=\langle V,\nn\rangle$, where $\nn$ is the outward unit normal to $\Omega$. 

Let $t_{0}$ be a fixed (made sufficiently small when necessary) positive number.
We define a perturbation {{$\Omega_{t}:=\Omega_{t}(\VV)$ due to the $t$-independent deformation field $\VV$}}, $t \in \mathcal{I} := [0,t_{0})$, of $\Omega$, $T_{0}(\Omega) = \Omega$, by the diffeomorphic map
\[
	T_{t} : t\in\mathcal{I}\mapsto \operatorname{id} + t \VV \in \vect{\mathcal{C}}^{2,1}(\mathbb{R}^{d}), \qquad \VV \in \sfTheta.
\]
Observe that $(d/dt)T_t \big|_{t=0} = \VV$ vanishes on $\Sigma$ and on some small tubular neighborhood ${D \setminus\overline{D}_{\delta}}$ of $\Sigma$ since $\operatorname{supp}(\VV) \subset \overline{D}_{\delta}$.
Throughout the paper, an expression with subscript `$t$' means it is defined on the perturbed domain $\Omega_{t}$ (e.g., $\udt$ satisfies \eqref{eq:state_ud} with $\Omega$ replaced by $\Omega_{t}=T_{t}(\Omega)$).
We emphasize that here, and throughout the study, $\VV$ is understood to be an \textit{autonomous} (admissible) deformation field.
Moreover, it is always understood hereinafter that $\Omega = D \setminus \overline{\omega}$ where $\omega \in \Wad$ and $\Wad$ is given by \eqref{eq:set_Wad}.

We say that the function $u(\Omega)$ has a \textit{shape} derivative $u'= u'(\Omega)[\VV] $ at $0$ (that is, with respect to $\Omega$) in the direction of the vector field $\VV$ if the limit $u' = \lim_{t \searrow 0} \frac{1}{t}[u(\Omega_t) - u(\Omega)]$ exists.
%%% SHAPE DERIVATIVE OF COST FUNCTIONALS
Meanwhile, a shape functional $\sfj : \Omega \to \mathbb{R}$ has a directional Eulerian derivative at $\Omega$ in the direction $\VV$ if the limit $\lim_{t \searrow0} \frac{1}{t}[\sfj(\Omega_t) - \sfj(\Omega)] =: {d} \sfj(\Omega)[\VV]$ exists (cf. \cite[Eq. (3.6), p. 172]{DelfourZolesio2011}). 
If the map $\VV \mapsto {d} \sfj(\Omega)[\VV]$ is linear and continuous, then $\sfj$ is \textit{shape differentiable} at $\Omega$, and the map is referred to as the \textit{shape gradient} of $\sfj$.

Similarly, the \textit{second-order} Eulerian derivative of $\sfj$ at $\Omega$ along the two vector fields $\VV$ and $\WW$ is given by
\[
	\lim_{s \searrow0} \frac{{d}\sfj(\Omega_s(\WW))[\VV] - {d}\sfj(\Omega)[\VV]}{s} 
	=: d^{2}\sfj(\Omega)[\VV,\WW],
\]
if the limit exists {{(cf. \cite[Chap. 9, Sec. 6, Def. 6.1, p. 506]{DelfourZolesio2011})}}. %%% \cite[Def. 2.3]{DelfourZolesio1991a}.
In addition, $\sfj$ is said to be \textit{twice shape differentiable} if, for all $\VV$ and $\WW$, $d^{2}\sfj(\Omega)[\VV,\WW]$ exists, and is bilinear and continuous with respect to $\VV, \WW$.
In this case, we call the expression the \textit{shape Hessian} of $\sfj$.
%----------------------------------------------------------------------------------------------------------------------------------
%	SHAPE DERIVATIVES OF THE STATE VARIABLES
%----------------------------------------------------------------------------------------------------------------------------------
\subsection{Shape derivative of $\ud$} \label{subsection:shape_derivative_of_the_Dirichlet_state}
To carry out a second-order analysis, we recall the shape derivative of the state variable $\ud$.
For this purpose, let us assume the following:
\begin{description}%%%\label{eq:assumption_on_Robin_coefficient} 
	\item[$(\text{{A}})$]  the impedance function\footnote{Recall that $\alpha$, generally, is assumed to be a fixed non-negative Lipschitz function in $\mathbb{R}^{d}$, $d \in \{2,3\}$, such that $\alpha \geqslant \alpha_{0} > 0$, where $\alpha_{0}$ is a known constant.} (i.e., the Robin coefficient) $\alpha \geqslant \alpha_{0} > 0$ on $\Gamma$ has an $H^{2}$ extension (still denoted by $\alpha$ for simplicity) in some neighborhood of $\Gamma$ and is constant in the normal direction, i.e., it is such that $\dn{\alpha} = 0$.
\end{description}
The reason we impose assumption $(\text{{A}})$ is to simplify the discussion throughout the paper.
In general, without assumption $(\text{{A}})$, an additional boundary term involving the normal derivative of $\alpha$ on $\Gamma$ must be considered in $\Upsilon(v)$.
Hereinafter, Assumption (A) will be assumed without further notice.
%
%
%
%%% SHAPE DERIVATIVE OF STATES
\begin{lemma}[\cite{AfraitesRabago2022}]\label{lem:shape_derivative_of_the_states}
	Let $\Omega \in \mathcal{C}^{2,1}$ be an admissible domain (i.e., $\Omega = D\setminus \overline{\omega}$ where $\omega \in \Wad$), $\VV \in \sfTheta$, and $\ud \in H^{1}(\Omega)$ be the solution to \eqref{eq:state_ud}.
	Then, $\ud \in H^{3}(\Omega)$ and is shape differentiable with respect to $\Omega$ in the direction of $\VV$. 
	Its shape derivative $\udp \in H^{1}(\Omega)$ uniquely solves the boundary value problem
	\begin{equation}
	\label{eq:shape_derivative_of_the state_ud}
		-\Delta \udp 			=0 \quad \text{in $\Omega$},\qquad
		\udp 					=0 \quad \text{on $\Sigma$},\qquad
		\dn{\udp} + \alpha \udp	=\Upsilon(\ud)[\Vn] \quad \text{on $\Gamma$},
	\end{equation}
	where 
	\begin{equation}\label{eq:Upsilon_definition}
		\Upsilon(v)[\Vn] = {\operatorname{div}}_{\tau} (\Vn\nabla_{\tau} {v}) 
						- \alpha (\dn{{v}}+\kappa {v})\Vn,
	\end{equation}
	for $v \in H^{3}(\Omega)$, and $\kappa =  {\operatorname{div}}_{\tau} \nn$ is the mean curvature of $\Gamma$.\footnote{Here $\nabla_{\tau}$ denotes the tangential gradient operator while ${\operatorname{div}}_{\tau}$ denotes the tangential divergence on $\Gamma$; see, e.g., \cite{DelfourZolesio2011}.}
\end{lemma}
\begin{remark}
Introducing the linear form ${l}_{{D}}(\psi) = \intG{\Upsilon({u}_{{D}})[\Vn] \psi}$, where $\Upsilon$ is given by \eqref{eq:Upsilon_definition},
we can state the variational formulation of \eqref{eq:shape_derivative_of_the state_ud} as follows
	\begin{equation}\label{eq:shape_derivative_of_the_state_weak_form_ud}
		\text{Find ${\udp} \in H_{\Sigma,0}^{1}(\Omega)$ such that $a(\udp,\psi) = l_{{D}}(\psi)$, for all $\psi \in H_{\Sigma,0}^{1}(\Omega)$}.
	\end{equation}
With $\Upsilon(\ud)[\Vn] \in H^{-1/2}(\Gamma)$, \eqref{eq:shape_derivative_of_the_state_weak_form_ud} can be shown to admit a unique weak solution in $H^{1}(\Omega)$ via the Lax-Milgram lemma.
\end{remark}
%%%
%----------------------------------------------------------------------------------------------------------------------------------
%	SHAPE DERIVATIVES OF THE COST FUNCTIONS
%----------------------------------------------------------------------------------------------------------------------------------
\subsection{Shape gradient of $J_{N}$} 
We recall the shape gradient of the shape function $J_{N}$ computed in \cite{AfraitesRabago2022}.
%%% ---------------------------------------------------------------------------
%	SHAPE DERIVATIVE VIA EULERIAN DERIVATIVE
%%% ---------------------------------------------------------------------------
\begin{proposition}[\cite{AfraitesRabago2022}]
	\label{prop:shape_gradients}
	Let $\Omega \in \mathcal{C}^{2,1}$ be an admissible domain, $\VV \in \sfTheta$, and $\ud \in H^{1}(\Omega)$ be the solution to \eqref{eq:state_ud}.
	The map $t \mapsto J_{N}(\Omega_t)$, is $\mathcal{C}^1$ in a neighborhood of $0$, and its shape derivative at $0$ is given by $dJ_{N}(\Omega)[\VV] = \intG{G_{N} \nu \cdot \VV}$, where the shape gradient $G_{N}$ is given by
	\begin{equation}\label{eq:shape_gradient_gn}
		G_{N} = \nabla_{\tau} \ud \cdot \nabla_{\tau} \pd + \alpha \left( \dn{\ud} + \kappa \ud \right)\pd.
	\end{equation}
	Here, the adjoint variable $\pd \in H^{1}(\Omega)$ solves the following system of PDEs
	\begin{equation}\label{eq:adjoint_system_pd}
		-\Delta \pd 		=0 \ \text{in $\Omega$},\qquad\quad
		\pd 				=\dn{\ud} - g \ \text{on $\Sigma$},\qquad\quad
		\dn{\pd} + \alpha \pd	=0 \ \text{on $\Gamma$}.
	\end{equation}		 
\end{proposition}
The variational formulation of \eqref{eq:adjoint_system_pd} can be stated as follows:
	\begin{equation}\label{eq:adjoint_weak_form_ud}
		\text{Find ${\pd} \in H^{1}(\Omega)$, $\pd = \dn{\ud} - g$ on $\Sigma$, such that $a(\pd,\psi) = 0$, $\forall \psi \in H_{\Sigma,0}^{1}(\Omega)$}.
	\end{equation}
The regularities $\Omega \in \mathcal{C}^{1,1}$ and $f \in H^{3/2}(\Sigma)$ are sufficient for \eqref{eq:adjoint_system_pd} to have a unique weak solution in $H^{1}(\Omega)$ since $\dn{\ud} - g \in H^{1/2}(\Sigma)$.
The existence of unique weak solution to \eqref{eq:adjoint_weak_form_ud} then follows from the Lax-Milgram lemma.
%
%
%
%--------------------------  SHAPE HESSIAN --------------------------
\subsection{Shape Hessian of $J_{N}$ at a critical shape}
\label{subsec:shape_Hessian}
Our goal here is to prove Proposition \ref{prop:shape_Hessian} which provides the structure of the shape Hessian at a critical shape.
In the next lemma, we give a result on the necessary optimality condition of our control problem that will be used in Proposition \ref{prop:shape_Hessian}.  
\begin{lemma}\label{lem:necessary_optimality_condition_tracking_Neumann}
Let ${\omega}^{\ast}$, or equivalently $\Omega^{\ast} = D\setminus \overline{\omega}^{\ast}$, be the solution of Problem \ref{eq:main_problem}. 
That is, the domain $\Omega^{\ast}$ is such that $\ud=\ud(\Omega^{\ast})$ satisfies \eqref{eq:gip}; i.e., it holds that $\dn{\ud}=g$ on $\Sigma$ where $\ud$ solves \eqref{eq:state_ud}.
Then, the adjoint state $\pd$ satisfying \eqref{eq:adjoint_system_pd} vanishes in $\Omega^{\ast}$, and with $G_{N}$ given by \eqref{eq:shape_gradient_gn}, there holds the necessary optimality condition $G_{N}=0$ on $\Gamma^{\ast}$.\footnote{This also gives us the \textit{shape Euler equation} \cite[p. 260]{DelfourZolesio2011} or Euler equation $dJ_{N}(\Omega^{\ast})[\VV]=0$.}
\end{lemma}
\begin{proof}
We assume that there exists an admissible shape $\Omega^{\ast} = D\setminus \overline{\omega}^{\ast}$ such that it realizes the absolute minimum of the criterion $J_{N}$.
That is, $J_{N}(\Omega^{\ast})=0$, or equivalently, $\dn{\ud(\Omega^{\ast})} = g$ on $\Sigma$.
This is satisfied by the solution of the inverse problem under consideration (i.e., Problem \ref{eq:main_problem}), and the solution of the adjoint problem 
\begin{equation}\label{eq:adjoint_system_pd_optimal}
		-\Delta \pd 		=0 \ \text{in $\Omega^{\ast}$},\qquad\quad
		\pd 				=0\ \text{on $\Sigma$},\qquad\quad
		\dn{\pd} + \alpha \pd	=0 \ \text{on $\Gamma^{\ast}$}.
	\end{equation}	
Because $\alpha > 0$, then clearly, $\pd\equiv0$ in $\overline{\Omega}^{\ast}$, {{and we get $G_{N}=0$ on $\Gamma^{\ast}$.}}
\end{proof}
We next give the characterization of the shape Hessian at the critical shape $\Omega^{\ast}$.
\begin{proposition}\label{prop:shape_Hessian}
	Let $\Omega \in \mathcal{C}^{2,1}$ be an admissible domain, $\VV, \WW \in \sfTheta$, and $\ud \in H^{3}(\Omega)$ and $\udp \in H^{1}(\Omega)$ be the respective solutions of \eqref{eq:state_ud} and \eqref{eq:shape_derivative_of_the state_ud}.
	The shape Hessian of $J_{N}$ at the solution $\Omega^{\ast}$ of \eqref{eq:gip} has the following structure:
	\begin{equation}\label{eq:shape_Hessian_at_stationary_domain}
		d^{2}J_{N}(\Omega^{\ast})[\VV,\WW] = -\intGast{\Upsilon(\ud)[\Vn] \pdp[\WW]},
	\end{equation}
where the adjoint variable $\pdp=\pdp[\WW]$ satisfies the system of PDEs 	
	\begin{equation}\label{eq:adjoint_system_pdp}
		-\Delta \pdp 		=0 \ \text{in $\Omega^{\ast}$},\qquad\quad
		\pdp 				= \dn{\udp}[\WW]  \ \text{on $\Sigma$},\qquad\quad
		\dn{\pdp} + \alpha \pdp	=  0 \ \text{on $\Gamma^{\ast}$}.
	\end{equation}	
\end{proposition}
\begin{proof}
Let the assumptions of the proposition hold.
Using Hadamard's formula, the second-order derivative of $J_{N}$ is obtained as 
\begin{equation}\label{eq:shape_Hessian_temporary_representation}
	d^{2}J_{N}(\Omega^{\ast})[\VV,\WW]
	= \intS{\left(\dn{\udp}[\VV]\dn{\udp}[\WW]+(\dn{\ud} - g)\dn{\udpp}[\VV,\WW] \right)},
\end{equation}
where $\udpp[\VV,\WW]$\footnote{The existence of $\udpp$ is implicitly assumed here to ensure the well-defined nature of the representation in \eqref{eq:shape_Hessian_temporary_representation}. However, this expression will not appear in the final expression of the shape Hessian, as demonstrated in \eqref{eq:shape_Hessian_at_stationary_domain}.} is the second-order derivative of the state $\ud$ at $0$ in the direction of the vector fields $\VV$ and $\WW$.
Using the necessary optimality condition in Lemma \ref{lem:necessary_optimality_condition_tracking_Neumann}, we obtain  
\begin{equation}\label{eq:Hessian}
	d^{2}J_{N}(\Omega^{\ast})[\VV,\WW]=\intS{\dn{\udp}[\VV]\dn{\udp}[\WW]}. 
\end{equation}
To transform the integral over $\Sigma$ into an integral over $\Gamma^{\ast}$, we utilize the shape derivative of \eqref{eq:adjoint_system_pd} given by the adjoint system \eqref{eq:adjoint_system_pdp}. 
Thanks to the properties of the adjoint state at the optimum shape, as given in Lemma \ref{lem:necessary_optimality_condition_tracking_Neumann}, $\pdp[\WW]$ shares the same characteristics as $\udp$, leading to the formulation in \eqref{eq:adjoint_system_pdp}. 
We then apply integration by parts to \eqref{eq:shape_derivative_of_the state_ud} and \eqref{eq:adjoint_system_pdp} with multipliers $\udp$ and $\pdp$, respectively, to obtain 
\[
	\intS{\left(\dn{\udp}[\VV]\dn{\udp}[\WW]\right)} = -\intGast{ \dn{\udp}[\VV]\pdp[\WW]}+\intGast{\dn{\pdp[\WW]}\udp[\VV]}.
\]	
Using the boundary conditions of \eqref{eq:shape_derivative_of_the state_ud} and \eqref{eq:adjoint_system_pdp} on $\Gamma^{\ast}$, we get 
\begin{equation*}
\begin{split}
\intS{\left(\dn{\udp}[\VV]\dn{\udp}[\WW]\right)}
%%%	&=-\intGast{ \left( \dn{\udp[\VV]}\pdp[\WW]+\alpha \pdp[\WW]\udp[\VV] \right)}\\
	&=- \intGast{\Upsilon(\ud)[\Vn] \pdp[\WW]}.
\end{split}
\end{equation*}
This proves the proposition.
\end{proof}
Let us assume the existence of an admissible inclusion $\Omega^{\ast}$ such that $J_{N}(\Omega^{\ast})=0$. 
This condition is evidently satisfied by the solution of Problem \ref{eq:main_problem}, or equivalently, the solution of the overdetermined boundary value problem \eqref{eq:gip}. 
Consequently, Euler's equation $dJ_{N}(\Omega^{\ast})[\VV]=0$ holds, and we have demonstrated in \eqref{eq:Hessian} that $d^{2}J_{N}(\Omega^{\ast})[\VV,\VV]=\intS{\left(\dn{\udp}[\VV]\right)^{2}}$. 
It is noteworthy that if $V_n\neq 0$ on $\Sigma$, then $d^{2}J_{N}(\Omega^{\ast})[\VV,\VV] > 0$. 
However, this positivity of $d^{2}J_{N}(\Omega^{\ast})$ does not imply the well-posedness of the minimization problem. 
Indeed, Proposition \ref{prop:compactness_result}, provided in the next subsection, elucidates the instability of the shape optimization problem under consideration.
%
%--------------------------  STABILITY ANALYSIS --------------------------
\subsection{Instability analysis of the critical shape of $J_{N}$}
\label{subsec:instability_analysis}
{{Now that we have obtained the expression for the shape Hessian at the optimal shape solution $\Omega^{\ast}$, we can evaluate whether the shape optimization problem \eqref{eq:shop} is well-posed. 
The shape Hessian derived, as outlined in Proposition \ref{prop:shape_Hessian}, offers insight into this issue.
The motivation behind this investigation stems from the fundamental interest in understanding the behavior of any reconstruction algorithm in the case under examination. 

Let us briefly review the current understanding of the stability of optimal shapes. 
The shape calculus developed in preceding sections remains valid within the $\mathcal{C}^2$ topology for deformation, with relevant findings documented in \cite{AfraitesDambrineEpplerKateb2007,Dambrine2002,DambrineSokolowskiZochowski2003,EpplerHarbrechtSchneider2007}. 
As noted in prior investigation (see, e.g., \cite{AfraitesDambrineEpplerKateb2007}), it is generally unreasonable to expect a shape Hessian to exhibit coercivity in this norm. 
However, coercivity might hold in a weaker norm: this situation is known as the \textit{two-norm discrepancy} problem, see \cite{Dambrine2002,Eppler2000a,EpplerHarbrechtSchneider2007}.
For instance, given additional conditions of continuity, both the stability results of  \cite{Dambrine2002,DambrineSokolowskiZochowski2003}  and the convergence result of \cite{EpplerHarbrechtSchneider2007} necessitate that the shape Hessian at the critical shape be coercive in such a weaker norm.
As emphasized in \cite{AfraitesDambrineEpplerKateb2007} and elucidated in \cite{EpplerHarbrecht2005} for a closely related shape inverse problem, encountering such a favorable scenario in this context is unlikely.
Consequently, the objective of this subsection is to carry out a precise analysis of the positivity of the shape Hessian. 
For a more comprehensive exploration of the topic, particularly regarding the stability of critical shapes, we recommend consulting \cite{DambrinePierre2000,Dambrine2002,EpplerHarbrecht2005}.

In this section, we establish conclusively that the problem under examination is ill-posed by proving the compactness of the shape Hessian at a critical shape. To achieve this, we employ a methodology akin to that utilized in \cite{Afraites2022} and \cite{AfraitesMasnaouiNachaoui2022}, building upon their methodologies.
Our argumentation hinges on the regularities observed in the admissible domains, deformation fields, the solution to the state problem \eqref{eq:state_ud}, and notably, the continuity of the mean curvature for boundaries of class $\mathcal{C}^{2,1}$. Subsequently, we employ a local regularity result to demonstrate the compactness of the Riesz operator corresponding to the shape Hessian at the exact solution of Problem \ref{eq:main_problem}.
Alternatively, one could pursue the desired result by applying potential layers, as exemplified in \cite{AfraitesDambrineEpplerKateb2007}, \cite{AfraitesDambrineKateb2008}, and \cite{EpplerHarbrecht2005}.
\begin{proposition}%%%[Compactness at the critical shape]
\label{prop:compactness_result}
If $\Omega^{\ast}$ is the critical shape of $J_{N}$, then the Riesz operator associated to the quadratic shape Hessian $d^{2}J_{N}(\Omega^{\ast}): \vect{H}^{1/2}(\Gamma^{\ast}) \longrightarrow \vect{H}^{-1/2}(\Gamma^{\ast})$ {{given by \eqref{eq:shape_Hessian_at_stationary_domain}}} is compact.
\end{proposition}
The previous result underlines the lack of stability of the shape optimization problem \eqref{eq:shop}.
It suggests, roughly, that in the vicinity of the critical shape $\omega^{\ast}$, or equivalently, $\Omega^{\ast}$ (i.e., for $t$ very small), the cost functional $J_{N}$ behaves as its second-order approximation and one cannot expect an estimate of the kind $t \leqslant C \sqrt{J_N(\Omega_{t})}$ with a constant $C$ uniform in $\VV$.
In other words, the proposition points out that the shape gradient does not have a uniform sensitivity with respect to the deformation directions; that is, $J_{N}$ is degenerate for wildly oscillating perturbations.

To prove Proposition \ref{prop:compactness_result}, we prepare the following small lemma.
%%%
\begin{lemma}
\label{lem:aux_result}
	Let $\Omega \in \mathcal{C}^{2,1}$ be an admissible domain, $\VV \in \sfTheta$, and $\ud \in H^{1}(\Omega)$ be the solution to \eqref{eq:state_ud}.
	Then, the map $\VV \longmapsto \Upsilon(\ud)[\Vn]$ is a continuous map from $\vect{H}^{1/2}(\Gamma)$ to $H^{1/2}(\Gamma)$.
\end{lemma}
\begin{proof}
	Let $\Omega \in \mathcal{C}^{2,1}$ be an admissible domain, $\VV \in \sfTheta$, and $\ud \in H^{1}(\Omega)$ be the solution to \eqref{eq:state_ud}.
	Then, $\ud$ is $H^{3}(\Omega)$ regular, and we have the following regularities: $\nn \in \mathcal{C}^{1,1}(N^{\varepsilon})$ and $\kappa \in C^{0,1}(N^{\varepsilon}) \subset W^{1,\infty}(N^{\varepsilon}) \subset H^{1}(N^{\varepsilon})$, where $N^{\varepsilon}$ is a neighborhood of $\partial\Omega$; see, e.g., \cite[Sec. 7.8]{DelfourZolesio2011}.
	Recall that $\Upsilon(\ud)[\Vn] = {\operatorname{div}}_{\tau} (\Vn\nabla_{\tau} {\ud}) - \alpha (\dn{{\ud}}+\kappa {\ud})\Vn$.
	Let us first look at the map $\VV \mapsto \kappa {\ud} \Vn$.
	By McShane-Whitney extension theorem, there is some function $\tilde{\kappa} \in \mathcal{C}^{0,1}(\overline{\Omega})$ such that $\tilde{\kappa}|_{\Gamma} = \kappa$.
%%%	Also, from \cite[Thm. 3.37, p. 102]{McLean2000}, there is a bounded linear extension operator $\tilde{E}:H^{1/2}(\Gamma) \to H^{1}(\Omega)$.
	Hence, in view of \cite[Thm. 1.4.1.1, p. 21]{Grisvard1985} and by trace theorem, we see that the operator $\VV \mapsto \Upsilon(\ud)[\Vn]$ is continuous from $\vect{H}^{1/2}(\Gamma)$ to $H^{1/2}(\Gamma)$ since the trace of $\tilde{\kappa} \ud \Vn$ is a composition of bounded operators.
	Now, on the other hand, we note that ${\operatorname{div}}_{\tau} (\Vn\nabla_{\tau} {\ud}) = {\operatorname{div}}_{\tau} (\nabla_{\tau} {\ud})\Vn + \nabla_{\tau} {\ud} \cdot \nabla_{\tau} {\Vn}$.
	By the same reasoning, using the regularities of $\VV$, $\ud$, and of $\nu$ mentioned earlier, the map $\VV \mapsto {\operatorname{div}}_{\tau} (\Vn\nabla_{\tau} {\ud})$ is also continuous from $\vect{H}^{1/2}(\Gamma)$ to $H^{1/2}(\Gamma)$.
\end{proof}
\begin{proof}[Proof of Proposition \ref{prop:compactness_result}]
The proof is based on the observation that the shape Hessian can be expressed as a composition of some linear continuous operators and a compact one.
The compactness being a consequence of the compact embedding between two Sobolev spaces.

Letting $\WW = \VV$, we recall the formula of the shape Hessian given in Proposition \ref{prop:shape_Hessian}, and denote by $\langle \cdot, \cdot \rangle$ the product of dualities $\vect{H}^{1/2}(\Gamma^{\ast}) \times \vect{H}^{-1/2}(\Gamma^{\ast})$, so that we can write $d^{2}J_{N}(\Omega^{\ast})$ as
\[
 d^{2}J_{N}(\Omega^{\ast})[\VV,\VV]=-\Big\langle \Upsilon(\ud)[\Vn],\pdp[\VV]\Big\rangle.
\]
%
%%% where $\VV \in \sfTheta$ is a perturbation field such that $\|\VV\|_{\vect{\mathcal{C}}^{2,1}(D)} < \min\{1,\delta/3\}$.
%
%
{{Here, $\pdp = \pdp[\VV]$ (which is assumed to exists)}} solves the system of PDEs
	\begin{equation*}
		-\Delta \pdp 		=0 \ \text{in $\Omega^{\ast}$},\qquad\quad
		\pdp 				= \dn{\udp}[\VV]  \ \text{on $\Sigma$},\qquad\quad
		\dn{\pdp} + \alpha \pdp	=  0 \ \text{on $\Gamma^{\ast}$}.
	\end{equation*}	
We introduce the mappings
\begin{equation*}
\begin{array}{rclrcl}
 \mathcal{L}: \vect{H}^{1/2}(\Gamma^{\ast}) &  \longrightarrow  &  H^{1/2}(\Gamma^{\ast})  ~~~~~~~~&\mathcal{K} : \vect{H}^{1/2}(\Gamma^{\ast}) &\longrightarrow &  H^{-1/2}(\Gamma^{\ast})  \\
   \VV& \longmapsto & \Upsilon(\ud)[\Vn], &  ~~~~~~~~\VV& \longmapsto  & \pdp[\VV].
\end{array}
\end{equation*}
By these maps we can express the shape Hessian as follows
\begin{equation*}
	d^{2}J_{N}(\Omega^{\ast})[\VV,\VV] = - \Big\langle  \mathcal{L}(\VV), \mathcal{K}(\VV)\Big\rangle.
\end{equation*}
The operator $ \mathcal{L}$ is clearly linear and continuous by Lemma \ref{lem:aux_result}, but the operator $\mathcal{K}$ is compact.\footnote{The composition of a compact linear operator and a bounded linear operator yields a compact linear operator \cite[Lem. 8.3-2, p. 422]{Kreyszig1989}}
Indeed, according to the characterization of $\pdp[\VV]$, we can decompose $\mathcal{K}$ into $\mathcal{K}=\mathcal{K}_{2} \circ \mathcal{K}_{1}$ with  
\begin{equation*}
\begin{array}{rclrcl}
 \mathcal{K}_{1}: \vect{H}^{1/2}(\Gamma^{\ast}) &  \longrightarrow  &  H^{1/2}(\Sigma) ~~~~~~~~~~ &\mathcal{K}_{2} :H^{1/2}(\Sigma) &\longrightarrow &  H^{-1/2}(\Gamma^{\ast})  \\
   \VV& \longmapsto & \dn{\udp[\VV]}, &  ~~~~~~~~~~\Psi& \longmapsto  & \Phi,
\end{array}
\end{equation*}
and $\Phi$ is the trace on $\Gamma^{\ast}$ of the solution $\Phi \in H^{1}(\Omega^{\ast})$ of the following problem
	\begin{equation*}\label{eq:regularite_local}
		-\Delta \Phi 		=0 \ \text{in $\Omega^{\ast}$},\qquad\quad
		\Phi  				= \Psi  \ \text{on $\Sigma$},\qquad\quad
		\dn{\Phi } + \alpha \Phi 	=  0 \ \text{on $\Gamma^{\ast}$}.
	\end{equation*}	
The map $\mathcal{K}_{1}$ is linear and continuous, but $\mathcal{K}_{2}$ is compact. 
The latter claim follows from the $H^{1}(\Omega^{\ast})$ regularity (globally) of $\Phi$, which is locally $H^{3}(D_{\delta}\setminus\overline{\omega}^{\ast})$ regular, the trace theorem, and the compact embedding $\vect{H}^{\frac{1}{2}}(\Gamma^{\ast}) \hookrightarrow H^{-\frac{1}{2}}(\Gamma^{\ast})$. 
This, in turn, concludes the desired compactness result.
\end{proof}
%
%
%
%--------------------------  SECOND-ORDER ANALYSIS  --------------------------
%--------------------------  SECOND-ORDER ANALYSIS  --------------------------
%--------------------------  SECOND-ORDER ANALYSIS  --------------------------
%--------------------------  SECOND-ORDER ANALYSIS  --------------------------
%--------------------------  SECOND-ORDER ANALYSIS  --------------------------
%
%
%
\subsection{Shape Hessian of $J_{D}$ and stability issue concerning \eqref{equa:shop}}
As in the investigation outlined in the previous subsection, the analysis for the stability issue concerning \eqref{equa:shop} requires knowledge of the shape Hessian. 
In this regard, the shape derivative of the state variable $\un$, provided below, is needed to compute the expression for $d^{2}J_{{D}}$ using the chain rule approach.
%%% SHAPE DERIVATIVE OF STATES
\begin{lemma}[\cite{AfraitesRabago2022}]\label{lem:shape_derivative_of_the_states}
	Let $\Omega \in \mathcal{C}^{2,1}$ be an admissible domain (i.e., $\Omega = D\setminus \overline{\omega}$ where $\omega \in \Wad$), $\VV \in \sfTheta$, and $\un \in H^{1}(\Omega)$ be the solution to \eqref{equa:state_un}.
	Then, $\un \in H^{3}(\Omega)$ is shape differentiable with respect to $\Omega$ in the direction of $\VV$. 
	Its shape derivative $\unp \in H^{1}(\Omega)$ uniquely solves the boundary value problem
	\begin{equation}
	\label{equa:shape_derivative_of_the state_un}
		-\Delta \unp 			=0 \quad \text{in $\Omega$},\qquad
		{\dn{\un'}} 				=0 \quad \text{on $\Sigma$},\qquad
		\dn{\unp} + \alpha \unp	=\Upsilon(\un)[\Vn] \quad \text{on $\Gamma$},
	\end{equation}
	where $\Upsilon(v)[\Vn] = {\operatorname{div}}_{\tau} (\Vn\nabla_{\tau} {v}) 
						- \alpha (\dn{{v}}+\kappa {v})\Vn$, for $v \in H^{3}(\Omega)$, and $\kappa =  {\operatorname{div}}_{\tau} \nn$ is the mean curvature of $\Gamma$.
\end{lemma}
By defining the linear form $l_{N}(\psi) = \intG{\Upsilon(\un)[\Vn] \psi}$, the variational formulation of \eqref{equa:shape_derivative_of_the state_un} can be stated as follows:
	\begin{equation}\label{equa:shape_derivative_of_the_state_weak_form_un}
		\text{Find ${\unp} \in H^{1}(\Omega)$ such that  $a(\unp,\psi) = l_{N}(\psi)$, for all $\psi \in H^{1}(\Omega)$},
	\end{equation}
where $a$ is of course the bilinear form defined in \eqref{eq:bilinear_form}.	

With $\Upsilon(\un)[\Vn] \in H^{-1/2}(\Gamma)$, it can be verified that \eqref{equa:shape_derivative_of_the_state_weak_form_un} admits a unique weak solution in $H^{1}(\Omega)$ via the Lax-Milgram lemma.
The $\mathcal{C}^{2,1}$ regularity assumption on the admissible domains is enough to justify the existence of $\unp$.

Meanwhile, the shape functional $J_{{D}}$ can be readily computed as $dJ_{{D}}(\Omega)[\VV] = \int_{\Gamma} G_{{D}} \nu \cdot \VV$, where the shape gradient $G_{{D}}$ is defined as (see \cite{AfraitesRabago2022})
	\begin{equation} \label{equa:shape_gradient_gd} 
		G_{{D}} = - \nabla_{\tau}\un\cdot \nabla_{\tau} \pn - \alpha \left( \dn{\un} + \kappa\un\right)\pn.
	\end{equation}
Here, $\pn:\Omega \to \mathbb{R}$ is the adjoint variable that uniquely satisfies the following system of PDEs:
\begin{equation}\label{equa:adjoint_system_pn}
    -\Delta \pn = 0 \quad \text{in $\Omega$}, \qquad
    \dn{\pn} = \un - f \quad \text{on $\Sigma$}, \qquad
    \dn{\pn} + \alpha \pn = 0 \quad \text{on $\Gamma$}.
\end{equation}
The weak formulation of \eqref{equa:adjoint_system_pn} can be stated as follows:
\begin{equation}\label{equa:adjoint_weak_form_un}
    \text{Find $\pn \in H^{1}(\Omega)$ such that $a(\pn,\psi) = \int_{\Sigma} (\un - f)\psi$, for all $\psi \in H^{1}(\Omega)$}.
\end{equation}

Due to the sufficient regularities on the domain and the data, the existence of a unique weak solution $\pn \in H^{1}(\Omega)$ to \eqref{equa:adjoint_weak_form_un} can then be readily deduced from the Lax-Milgram lemma.

The corresponding necessary optimality condition for the optimization problem \eqref{equa:shop}, analogous to Lemma \ref{lem:necessary_optimality_condition_tracking_Neumann}, can be immediately established and is provided in the next lemma.
\begin{lemma}\label{lem:necessary_optimality_condition_tracking_Dirichlet}
Let $\Omega^{\ast} = D\setminus \overline{\omega}^{\ast}$ be the solution of Problem \ref{problem:main_problem}.
That is, the domain $\Omega^{\ast}$ is such that $\un=\un(\Omega^{\ast})$ satisfies \eqref{eq:gip}; i.e., it holds that $\un=f$ on $\Sigma$ where $\un$ solves \eqref{equa:state_un}.
Then, the adjoint state $\pn$ satisfying \eqref{equa:adjoint_system_pn} vanishes in $\Omega^{\ast}$, and with $G_{D}$ given by \eqref{equa:shape_gradient_gd}, there holds the necessary optimality condition $G_{D}=0$ on $\Gamma^{\ast}$. 
\end{lemma}
Using the above result, the characterization of the shape Hessian at the critical shape $\Omega^{\ast}$ is established in the same way as in the proof of Proposition \ref{prop:shape_Hessian}.
\begin{proposition}
	Let $\Omega \in \mathcal{C}^{2,1}$ be an admissible domain, $\VV, \WW \in \sfTheta$, and $\un \in H^{3}(\Omega)$ and $\unp \in H^{1}(\Omega)$ be the respective solutions of \eqref{equa:state_un} and \eqref{equa:shape_derivative_of_the state_un}.
	The shape Hessian of $J_{{D}}$ at the solution $\Omega^{\ast}$ of \eqref{eq:gip} has the following structure:
	\begin{equation}\label{equa:shape_Hessian_of_JD_at_stationary_domain}
		d^{2}J_{{D}}(\Omega^{\ast})[\VV,\WW] = \intGast{\Upsilon(\un)[\Vn] \pnp[\WW]},
	\end{equation}
where the adjoint variable $\pnp=\pnp[\WW]$ satisfies the system of PDEs 	
	\begin{equation}\label{equa:adjoint_system_pdp}
		-\Delta \pnp 		=0 \ \text{in $\Omega^{\ast}$},\qquad\quad
		\dn{\pnp} 				= \unp[\WW]  \ \text{on $\Sigma$},\qquad\quad
		\dn{\pnp} + \alpha \pnp	=  0 \ \text{on $\Gamma^{\ast}$}.
	\end{equation}	
\end{proposition}
With all the necessary ingredients prepared, one can now obtain a compactness result for the shape Hessian $d^{2}J_{{D}}$, analogous to Proposition \ref{prop:compactness_result}, as follows:
\begin{proposition}
If $\Omega^{\ast}$ is the critical shape of $J_{{D}}$, then the Riesz operator associated to the quadratic shape Hessian $d^{2}J_{{D}}(\Omega^{\ast}): \vect{H}^{1/2}(\Gamma^{\ast}) \longrightarrow \vect{H}^{-1/2}(\Gamma^{\ast})$ given by \eqref{equa:shape_Hessian_of_JD_at_stationary_domain} is compact. 
\end{proposition}
The argument to verify the above claim, as one could expect, closely resembles the proof provided for Proposition \ref{prop:compactness_result}. 
Specifically, the conclusion follows from the regularities of $\Omega$, $\VV$, and the data $g$, in conjunction with the trace theorem and \cite[Thm. 1.4.1.1, p. 21]{Grisvard1985}.
%%%%%%%%%%%%%%%%%%%%%%%%%%%%%%%%%%%%%%%%%%%%%%%%%%%%%%%%%%%%%%%%%%%%%%%%%%%%%
%%%%%%%%%%%%%%%%%%%%%%%%%%%%%%%%%%%%%%%%%%%%%%%%%%%%%%%%%%%%%%%%%%%%%%%%%%%%%
%%%%%%%%%%%%%%%%%%%%%%%%%%%%%%%%%%%%%%%%%%%%%%%%%%%%%%%%%%%%%%%%%%%%%%%%%%%%%
%----------------------------------------------------------------------------------------------------------------------------------
%	NUMERICAL STUDIES
%----------------------------------------------------------------------------------------------------------------------------------
\section{Numerical implementation and experiments}\label{sec:numerics}
\subsection{Employing the Neumann data-tracking least-squares approach}
\label{subsec:numerical_experiments_tracking_the_Neumann_data}

As indicated in subsection \ref{sec:tracking_the_Neumann_data}, we will solve the minimization problem \eqref{eq:shop} numerically using a shape-gradient-based descent method. 
This approach will be implemented via the finite element method (FEM), following the methodology outlined in our previous work \cite{AfraitesRabago2022}. 
A significant departure from \cite{AfraitesRabago2022} is our use of multiple pairs of Cauchy data. 
This choice aims to mitigate the numerical instability associated with identifying internal boundaries with concavities.
We emphasize that as opposed to \cite{RabagoAzegami2018,CaubetDambrineKateb2013}, our numerical scheme will not incorporate any remeshing techniques, such as adaptive mesh refinement. 
Furthermore, we will limit our investigation to cases of exact measurements and exclude consideration of noisy data. 
This decision is intentional, as we aim to examine the ill-posed nature of the problem, with a particular emphasis on cases where exact matching of the boundary data is achieved.
\begin{remark}
In geometric inverse problems, regularization, such as perimeter penalization, is crucial, particularly when dealing with noisy data. 
The presence of (weighted) perimeter functional in the objective function adds compactness properties to minimizing sequences and therefore contributes to the existence of optimal shapes. 
Such regularization methods not only provide stable reconstruction by effectively controlling the curve's length numerically but also aid in demonstrating the existence of minimizers for the shape optimization reformulation of the overdetermined problem. 
In future studies, we will explore the inclusion of a regularizing term as we examine inversion with noisy data, focusing on reconstructing Robin boundaries with concavities in greater detail. 
We note that adding a regularizing term to the shape functional can definitely result in significant changes to the analysis of existence optimal shape solution. 
For further insights, we refer the reader to \cite[Chap. 4.6, pp. 166-169]{HenrotPierre2018} regarding the effect of perimeter constraints on shape optimization problems.
\end{remark}
%--------------------------  NUMERICAL ALGORITHM --------------------------
%--------------------------  NUMERICAL ALGORITHM --------------------------
%--------------------------  NUMERICAL ALGORITHM --------------------------
\subsubsection{Numerical algorithm}
\label{subsec:Numerical_Algorithm}
For clarity and completeness, we provide essential details of our numerical method.

%
%-------------------------------------------
% CHOICE OF DESCENT DIRECTION
%-------------------------------------------
\textit{Choice of descent direction.} 
In computing the descent direction, we will make use of the Riesz representation \cite{Neuberger1997} of the shape gradient $G_{N}$ by solving the variational equation 
\begin{equation}\label{eq:smoothing}
	\intO{( \nabla \VV : \nabla \vect{\varphi} + \VV \cdot \vect{\varphi} ) } 
		= - \intG{G_{N}\nn \cdot \vect{\varphi}}, \qquad \text{for all $\vect{\varphi} \in H_{\Sigma,0}^1(\Omega)^{d}$}.
\end{equation}
The reason for this choice is straightforward: the $L^{2}(\Gamma)$ regularity of the shape gradient $G_{N}$ alone is not sufficient for a stable approximation of the exact unknown boundary. 
Additionally, since $G_{N}$ is only supported in $\Gamma$, we require a smooth extension of $G_{N}$ across the entire domain $\Omega$ to efficiently address the minimization problem using the finite element method.
This technique, which involves smoothing and preconditioning the extension of $-{G}_{N}\nn$ over the entire domain $\Omega$, has been employed in previous studies, as seen in, for example, \cite{CaubetDambrineKateb2013,RabagoAzegami2018}. 
For a more in-depth investigation of discrete gradient flows in shape optimization, we refer the reader to \cite{Doganetal2007}.

Now our algorithm for computing $k$th domain $\Omega^{k}$ is given as follows:
%---------------------------------------------------------------------------------------------------
%THE MAIN ALGORITHM (FIRST-ORDER METHOD)
%---------------------------------------------------------------------------------------------------
\begin{description}
\setlength{\itemsep}{0.05em}
	\item[1. \it{Initilization}] Choose an initial shape $\Omega^{0}$.  
	\item[2. \it{Iteration}] For $k = 0, 1, 2, \ldots$
		\begin{enumerate}
			\item[2.1] Solve the state and adjoint state systems on the current domain $\Omega^{k}$.
			\item[2.2] Choose $t^{k}>0$, and compute $\VV^{k}$ in $\Omega=\Omega^{k}$ according to \eqref{eq:smoothing}.
			\item[2.3] Update the current domain by setting $\Omega^{k+1} := \{ x + t^{k} \VV^{k}(x) \mid x \in \Omega^{k} \} $. 
		\end{enumerate}
	\item[3. \it{Stop Test}] Repeat \textit{Iteration} until convergence.
\end{description}
The step size $t^{k}$ in Step 2.2 is determined using a backtracking line search procedure, which relies on an Armijo-Goldstein-like condition for the shape optimization method, as described in \cite[p. 281]{RabagoAzegami2020}. 
Additionally, this value is decreased further to prevent reversed triangles within the mesh after the update. 
Meanwhile, convergence is achieved in Step 3 when the algorithm has completed a finite number of iterations.\footnote{Here, the number of iterations is set to be large enough so that the absolute difference between two consecutive cost values (i.e., the magnitude of $|J(\Omega^{k+1})-J(\Omega^{k})|$) eventually becomes smaller than some specified small positive number.}
Certainly, this criterion can be modified and improved upon. 
However, as it stands, it enables us to achieve good results, especially when utilizing multiple boundary measurements.}

%---------------------------------------------------------------------------------------------------
%THE BOUNDARY VARIATION ALGORITHM (SECOND-ORDER METHOD)
%---------------------------------------------------------------------------------------------------
\begin{remark}\label{rem:Newton_method}%%%[Incorporating the shape Hessian information in the numerical procedure]
The idea that integrating Hessian information into a gradient-based iterative scheme enhances convergence is widely acknowledged. 
However, second-order methods entail the drawback of increased computational burden and time, especially when dealing with complex Hessians \cite{NovruziRoche2000,Simon1989}. 
Moreover, the effort required to compute the Hessian often does not justify the reduction in iteration count \cite{AfraitesDambrineEpplerKateb2007}. Consequently, we choose not to utilize a second-order method for numerical optimization. 
The second-order analysis was conducted solely for stability assessment regarding the proposed optimization problem(s). 
Instead, our approach employs a simpler and more direct strategy, using multiple measurements, to enhance inclusion detection and algorithm convergence.
\end{remark}

For the input data, we will consider multiple sets of Cauchy pairs $\{(f^{(i)},g^{(i)})\}_{1\leqslant i \leqslant M}$, where $M$ denotes the maximum number of pairs, on $\Sigma$. 
Here, the prescribed data is the Dirichlet boundary data $f^{(i)}$, while its corresponding additional boundary measurement $g^{(i)}$ on the outer part $\Sigma$ is obtained synthetically. 
In other words, the corresponding Neumann flux $g^{(i)}$ measured on the accessible boundary $\Sigma$ is generated by numerically solving the direct problem corresponding to Problem \ref{eq:main_problem} using the finite element method. 
This is achieved by specifying $f$ and fixing the shape $\Omega^{\ast}$ (or equivalently, $\Sigma = \partial{D}$ and $\Gamma^{\ast} = \partial{\omega}^{\ast}$), then solving \eqref{eq:state} in $\Omega^{\ast}$, and extracting the measurement $g$ by computing $\dn{u}$ on $\Sigma$.
To prevent ``inverse crimes'' -- as discussed by Colton and Kress in \cite[p. 154]{ColtonKress2013} -- we create the synthetic data using a different numerical scheme compared to the inversion process. 
Specifically, we utilize a larger number of discretization points and apply ${P}{2}$ finite element basis functions in the \textsc{FreeFem++} \cite{Hecht2012} code. 
In the inversion procedure, all variational problems are solved using ${P}{1}$ finite elements, and we initially discretize the domain with a uniform mesh size of $h = 0.03$.
Subsequently, we extract $g^{(i)}$ for each $i = 1, 2, \ldots, M$ by computing $\dn{u^{(i)}}$, where $u^{(i)}$ solves \eqref{eq:state_ud} with $f = f^{(i)}$, on $\Sigma$. 
For our numerical experiments, we consider up to four linearly independent Cauchy pairs, with the values for $f^{(i)}$, $i=1,2,3,4$, given as follows: $f^{(1)} = \sin(t)$, $f^{(2)} = \cos(t)$, $f^{(3)} = \sin(2t)$, and $f^{(4)} = \cos(2t)$.

Depending on the value of $M$, we simply replace the shape gradient $G_{N}$ in \eqref{eq:smoothing} with the sum $\sum_{i=1}^{M} G_{N}^{(i)}$. 
Each $G_{N}^{(i)}$ corresponds to the shape gradient computed using the input data $f^{(i)}$, along with the cost function $J_{N}^{(i)} = \frac12 \int_{\Sigma} (\partial_{\nu} u^{(i)} - g^{(i)})^{2} \, ds$ for $i = 1,2,\ldots,M$.

For the exact geometry of the unknown boundary $\Gamma$ in the forward problem, we consider the following shapes:
\begin{itemize}
	\item a kite-shape boundary with parametrization 
	\[
		\Gamma_{\text{K}} = \left\{\begin{pmatrix} 0.195+0.4(\cos{t}+0.65\cos{2t})\\ 0.55\sin{t}\end{pmatrix}, \ t \in [0, 2\pi) \right\},
	\]
	\item a ribbon-shape boundary with parametrization 
	\[
		\Gamma_{\text{R}} = \left\{\begin{pmatrix} 0.64 \cos{t} \\ 0.48 \sin{t} (1.8 + \cos{(2t)}) \end{pmatrix}, \ t \in [0, 2\pi) \right\},
	\]
	\item a peanut-shape boundary with parametrization 
	\[
		\Gamma_{\text{F}} 
		= \left\{\begin{pmatrix} -0.25 + \displaystyle\frac{0.6+0.54\cos{t}+0.06\sin{2t}}{1+0.75\cos{t}}\cos{t} \\[0.75em] 0.05 + \displaystyle\frac{0.6+0.54\cos{t}+0.06\sin{2t}}{1+0.75\cos{t}}\sin{t}\end{pmatrix}, \ t \in [0, 2\pi) \right\},
	\]
	\item and the shape given by the boundary $\Gamma_{\text{L}}$ of the domain $\omega=(-0.55,0.55)^{2}\setminus[0,0.55]^{2}$.	
\end{itemize}
To condense statements and discussions, all experiments and figures will occasionally be referenced using the following notation: $\texttt{Test}(\Gamma^{\ast})$, where $\Gamma^{\ast}$ (the shapes defined above) represents the exact geometry of $\Gamma$. Lastly, in all of our experiments, the Robin coefficient is set to $\alpha = 1$, and the results are obtained with the initial guess $\Gamma^{(0)}$ being the circle of radius $0.3$. %%% See Figure \ref{fig:figure1}.
%%%%%%%%%%%%%%%%%%%%%%%%%%%%%%%%%%%%%%%%%%%%%%%%%%%%%%%%%%%%%%%%%%%%%%%%%%%%%
%%%%%%%%%%%%%%%%%%%%%%%%%%%%%%%%%%%%%%%%%%%%%%%%%%%%%%%%%%%%%%%%%%%%%%%%%%%%%
%------------------------------------------------------------------------------------
% 			NUMERICAL RESULTS AND DISCUSSION
%------------------------------------------------------------------------------------  
%%%\subsubsection{Numerical results and discussion}\label{subsec:discussion}
%
%%%%%%%%%%%%%%%%%%%%%%%%%%%%%%%%%%%%%%%%%%%%%%%%%%%%%%%%%%%%%%%%%%%%%%%%%%%%%
\subsubsection{Tests with single boundary measurement}
Before we present the numerical results for detections under multiple boundary measurements, we first motivate this section by providing numerical findings in the case when we only have a single input data for the inversion process. To this end, we consider three different inputs $f$: (i) $f_{1} = 1$, (ii) $f_{2} = \sin{t}$, $t\in [0,2\pi)$, and (iii) $f_{3} = \cos{t}$, $t\in [0,2\pi)$.

The results of the experiments, which were tested for shapes with exact interior boundaries given by $\Gamma_{\text{K}}$, $\Gamma_{\text{R}}$, $\Gamma_{\text{P}}$, and $\Gamma_{\text{L}}$, are plotted in Figure \ref{fig:figure1}. It is evident from the plotted figures that the detected shapes are far from the exact geometries, except in the case when $f=1$ is the input data. In this case, the algorithm was able to cover at least the convex hull of $\Gamma_{\text{P}}$ and $\Gamma_{\text{L}}$, and produce fair detections for the other two test cases. These results motivate us to consider multiple boundary measurements in the inversion process, which we discuss next.
%%%%%%%%%%%%%%%%%%%%%%%%%%%%%%%%%%%%%%%%%%%%%%%%%%%%%%%%%%%%%%%%%%%%%%%%%%%%%
\subsubsection{Tests with multiple Cauchy data}\label{subsec:multiple_measurements_tracking_Neumann}
For our numerical experiments involving multiple measurements, we maintain consistency by using the same set of geometries for $\Gamma^{\ast}$. 
The results, depicted in Figure \ref{fig:figure2}, illustrate the results when employing two, three, and four linearly independent input data or boundary measurements.

As anticipated, the detections significantly improve with multiple boundary measurements compared to using only a single pair of Cauchy data in the inversion process. 
Particularly, accurate detection of the concave parts of the exact interior boundary is achieved for \texttt{Test}($\Gamma_{\text{K}}$), \texttt{Test}($\Gamma_{\text{R}}$), and \texttt{Test}($\Gamma_{\text{P}}$), with detection generally enhancing as the number of boundary measurements increases.

However, for \texttt{Test}($\Gamma_{\text{L}}$), detection of the concave part appears less effective, likely due to the less smooth shape of $\Gamma_{\text{L}}$ and its distance from the measurement region. 
Nevertheless, employing multiple measurements substantially enhances the detection of unknown interior boundaries with non-convex shapes, partially mitigating the ill-posedness of the shape inverse problem.

Figure \ref{fig:figure3} summarizes the histories of the (normalized) cost value, Sobolev gradients' norm, and Hausdorff distances between the approximate and exact solutions against the number of iterations for \texttt{Test}($\Gamma_{\text{L}}$). 
Overall, as the number of boundary measurements increases, both the accuracy of the solution and the convergence behavior of the algorithm improve, as expected.

In conclusion, our results demonstrate the effectiveness of the least-squares method in tracking Neumann data using multiple boundary measurements compared to using a single pair of boundary data. 
While the proposed strategy provides valuable information for detecting inclusion boundaries near the measurement region, detecting interior boundaries with sharp concavities and located farther from the measurement region remains challenging, primarily due to the inherent nature of the problem.
%
%------------------------------------------------------------------------------------
% 			FIGURE  0: SHAPES
%------------------------------------------------------------------------------------         
\begin{figure}[htb!]
\centering
\resizebox{0.235\linewidth}{!}{\includegraphics{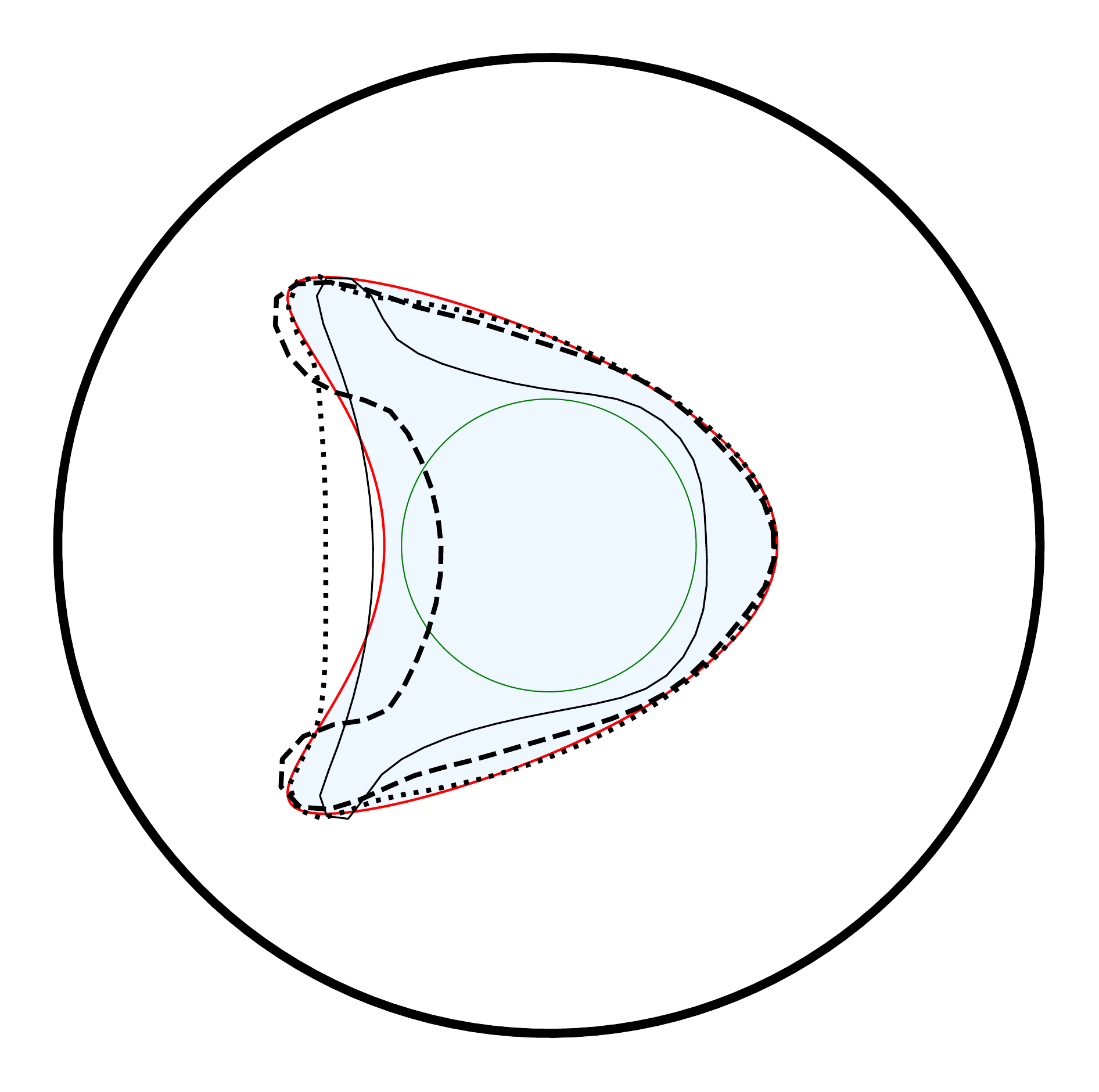}} \
\resizebox{0.235\linewidth}{!}{\includegraphics{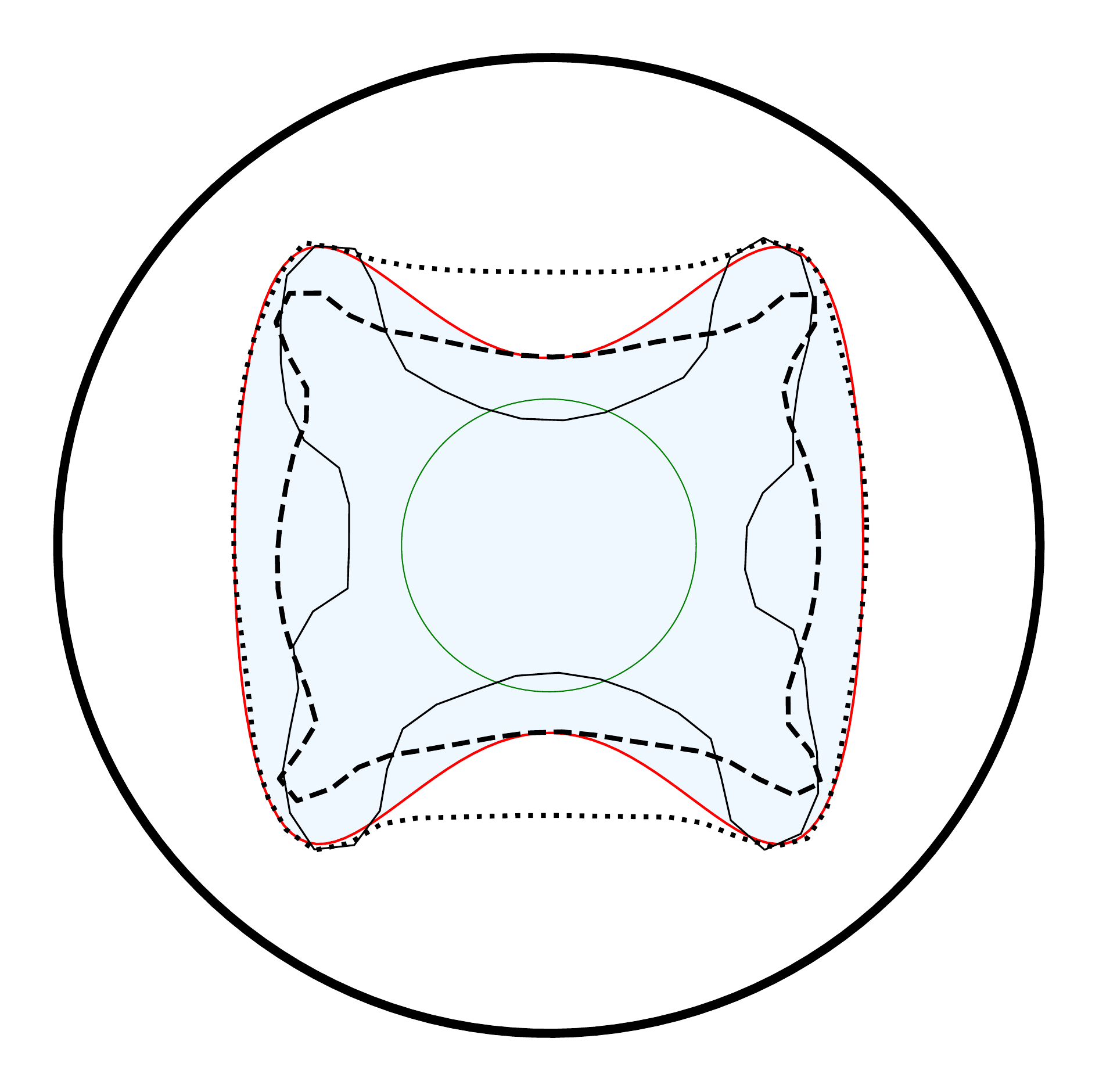}} \
\resizebox{0.235\linewidth}{!}{\includegraphics{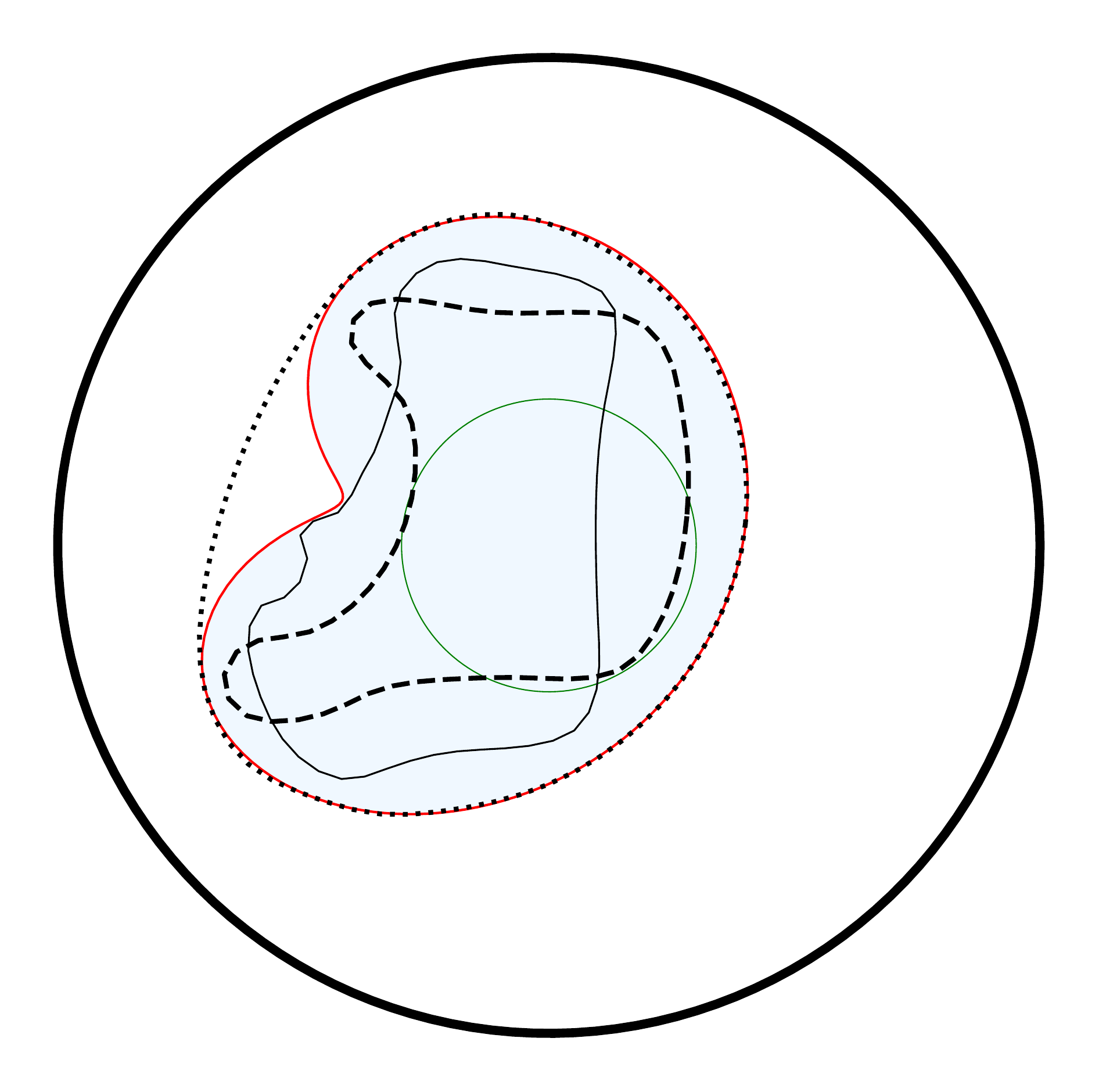}} \
\resizebox{0.235\linewidth}{!}{\includegraphics{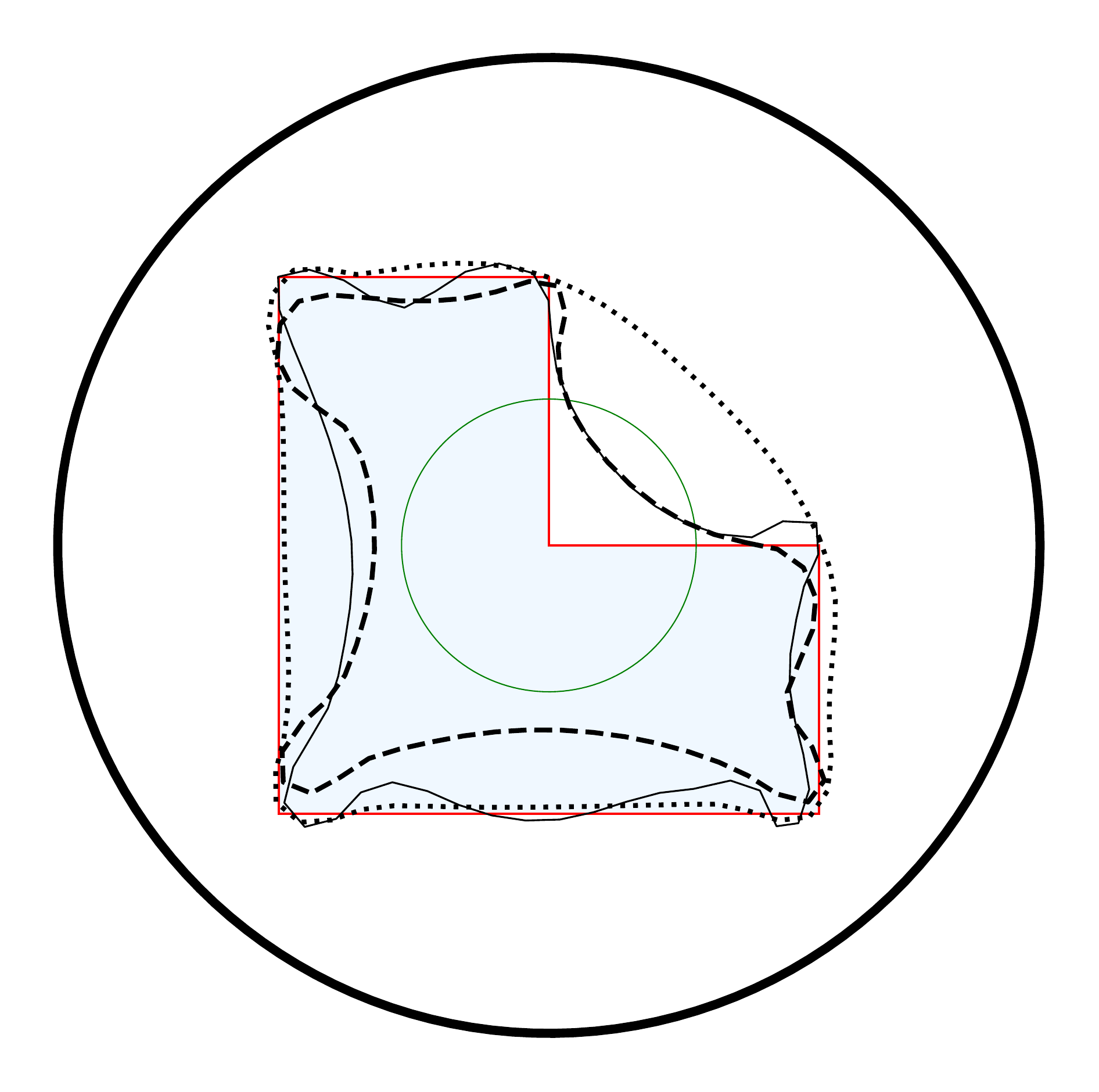}} \\
\resizebox{0.8\linewidth}{!}{\includegraphics{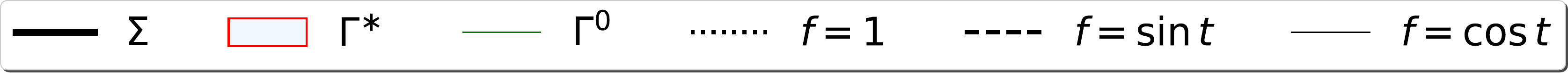}} 
\caption{Results of the detections with a single boundary measurement} 
\label{fig:figure1}
\end{figure}
%
%
%
%
%
%
%------------------------------------------------------------------------------------
% 			FIGURE  2: SHAPES
%------------------------------------------------------------------------------------         
\begin{figure}[htp!]
\centering
\resizebox{0.235\linewidth}{!}{\includegraphics{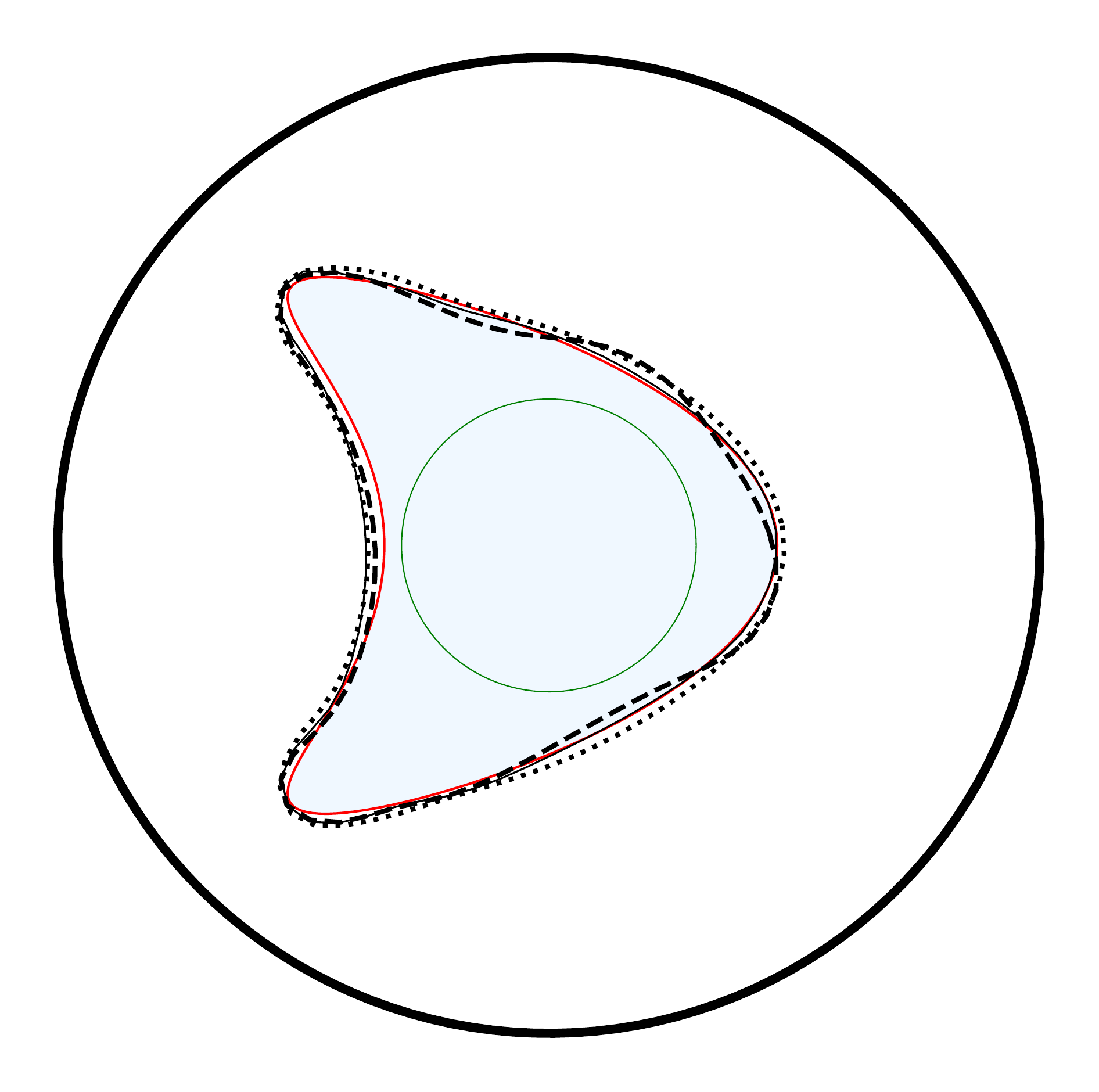}} \
\resizebox{0.235\linewidth}{!}{\includegraphics{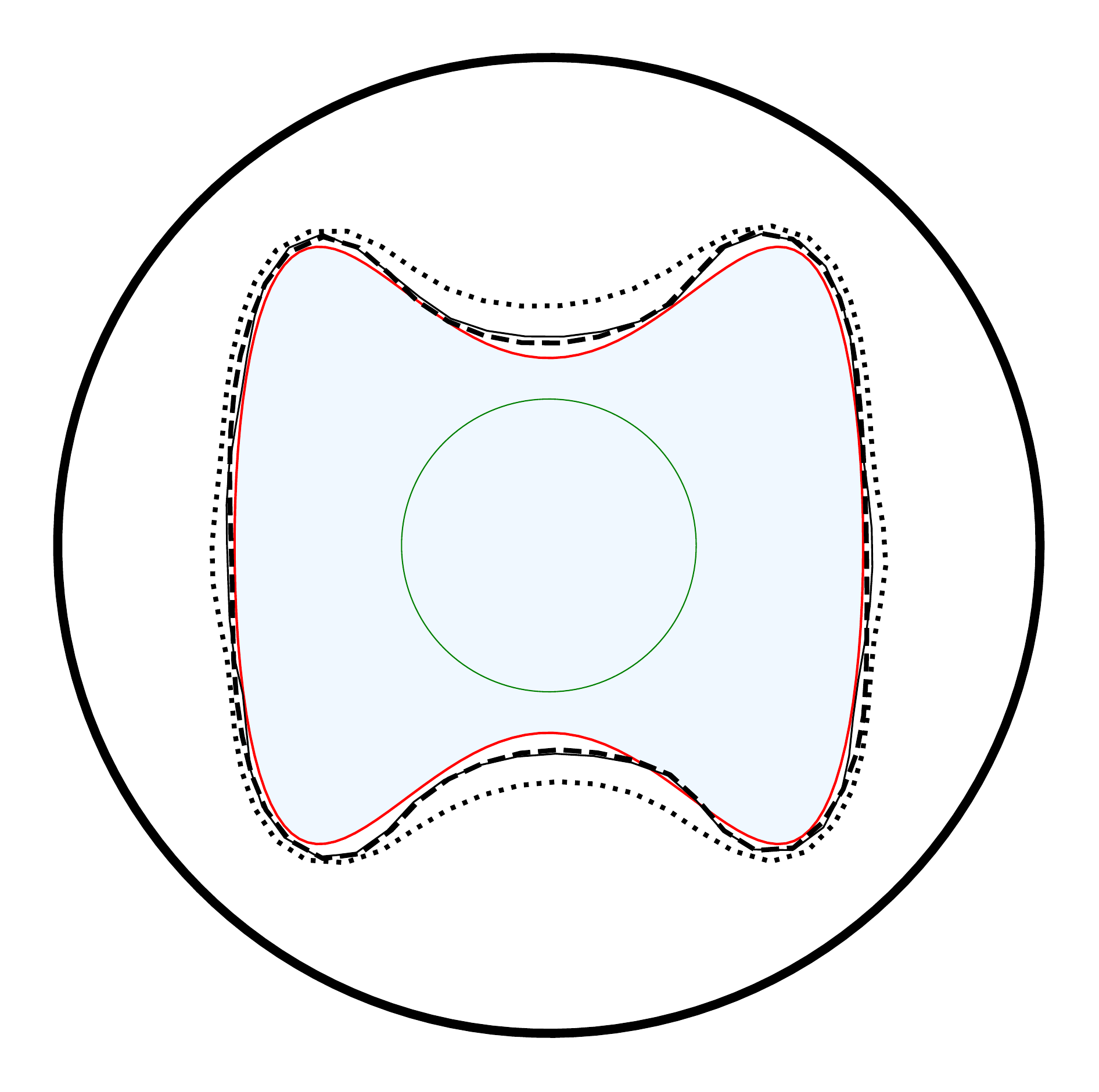}} \
\resizebox{0.235\linewidth}{!}{\includegraphics{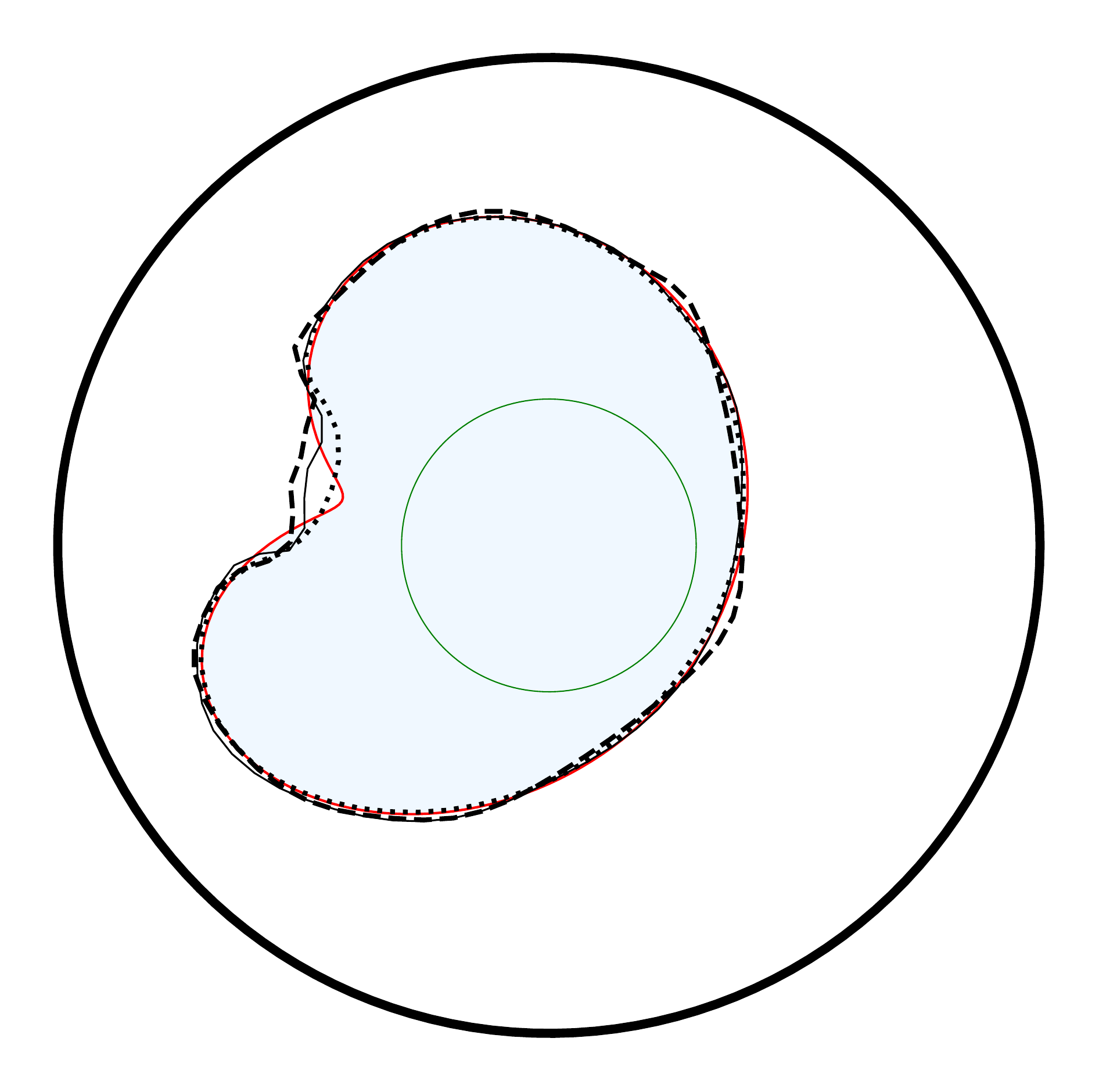}} \
\resizebox{0.235\linewidth}{!}{\includegraphics{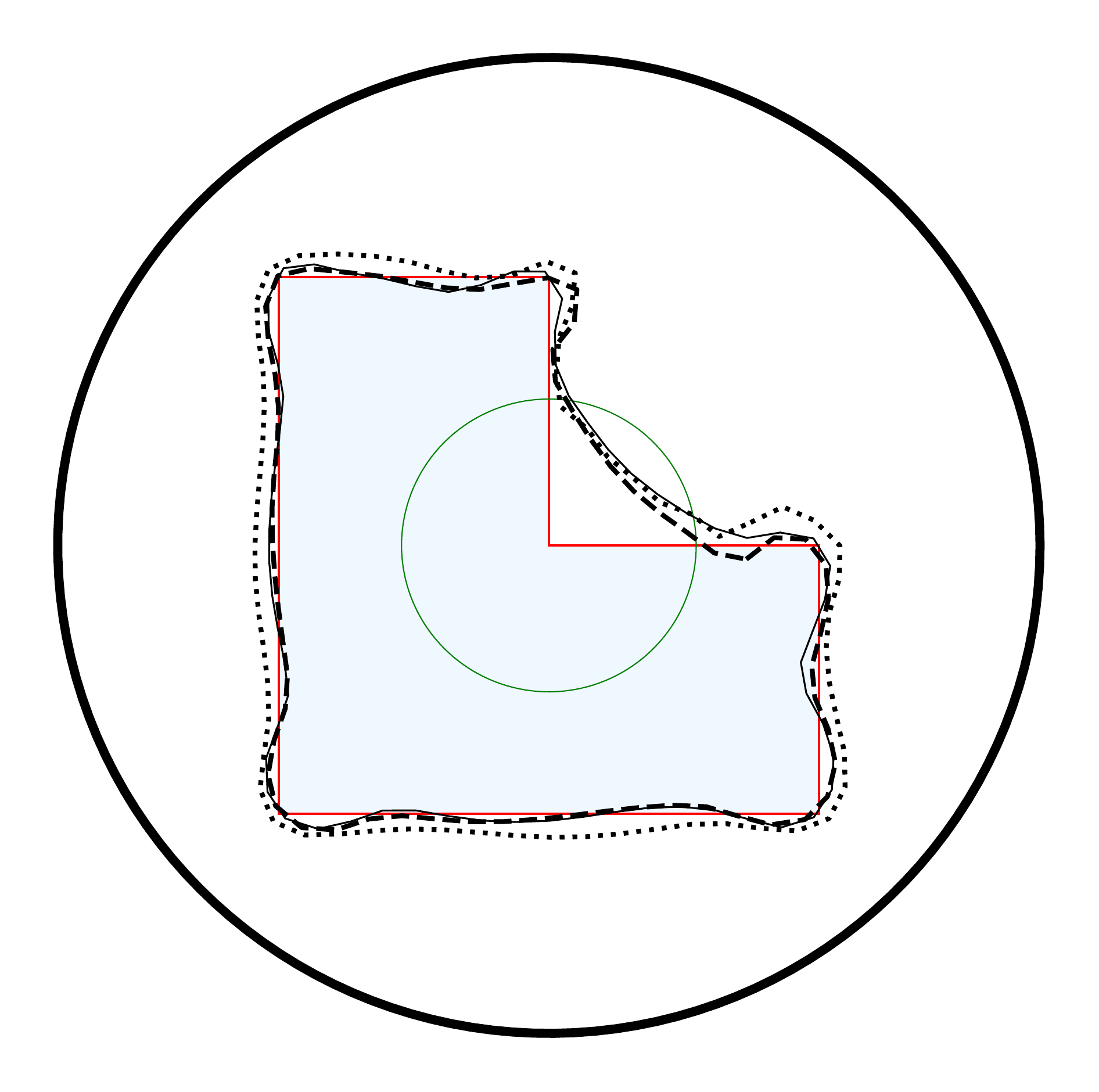}} \\ 
\resizebox{0.8\linewidth}{!}{\includegraphics{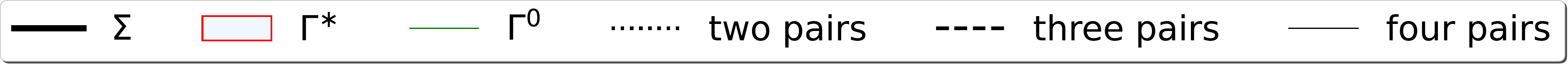}} 
\caption{Results of the detections with multiple boundary measurements}
\label{fig:figure2}
\end{figure}
%
%
%------------------------------------------------------------------------------------
% 			FIGURE  3: HISTORIES OF VALUES
%------------------------------------------------------------------------------------         
\begin{figure}[htp!]
\centering
\resizebox{0.325\linewidth}{!}{\includegraphics{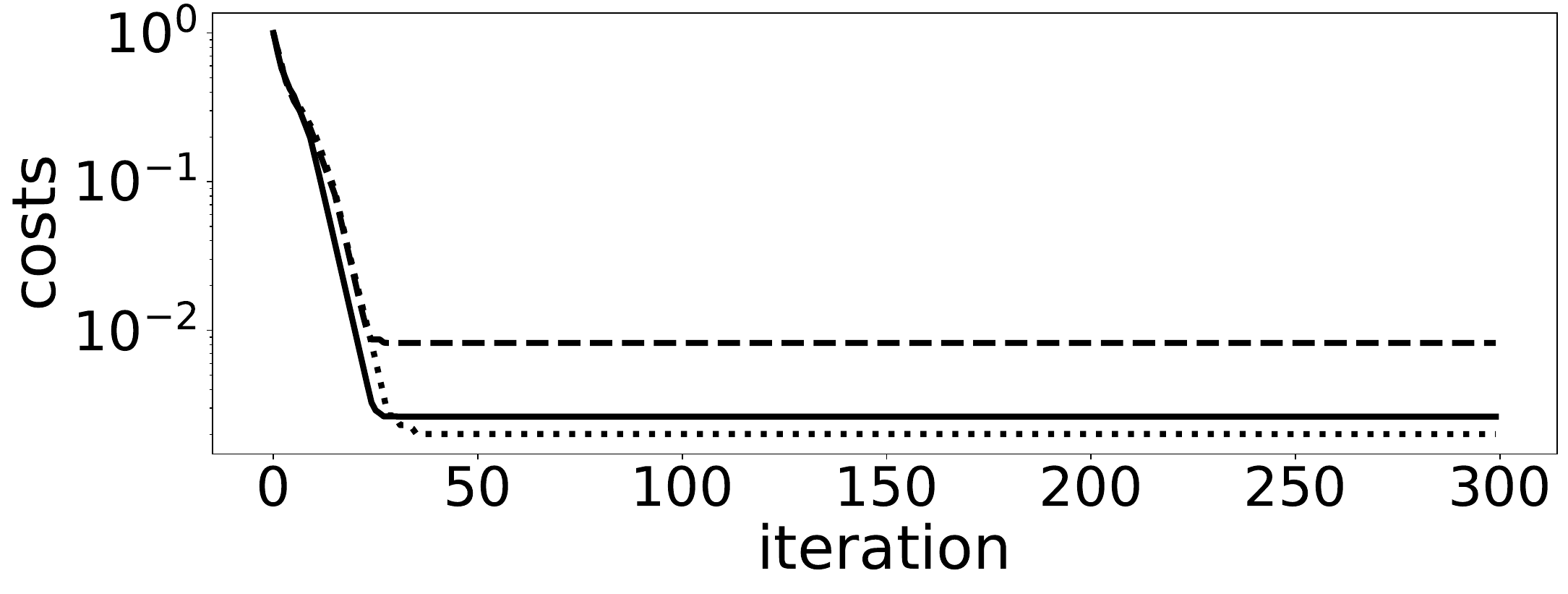}} \hfill
\resizebox{0.325\linewidth}{!}{\includegraphics{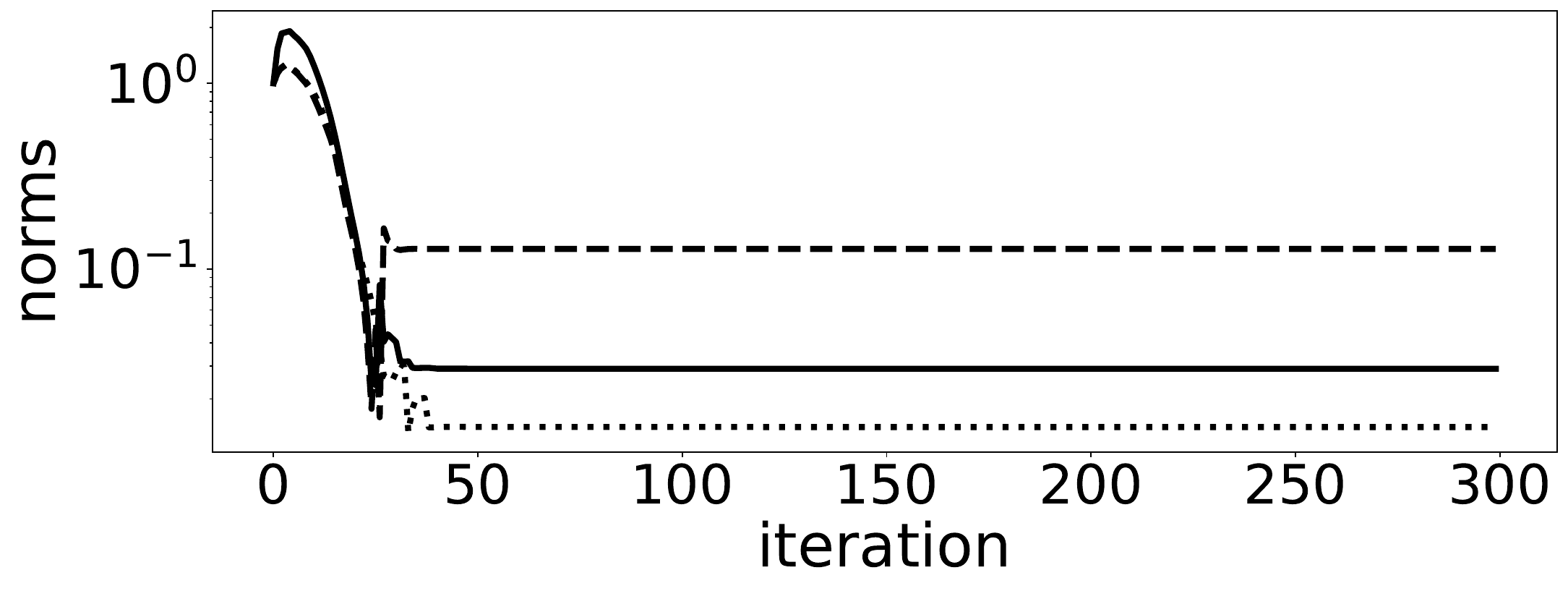}} \hfill
\resizebox{0.325\linewidth}{!}{\includegraphics{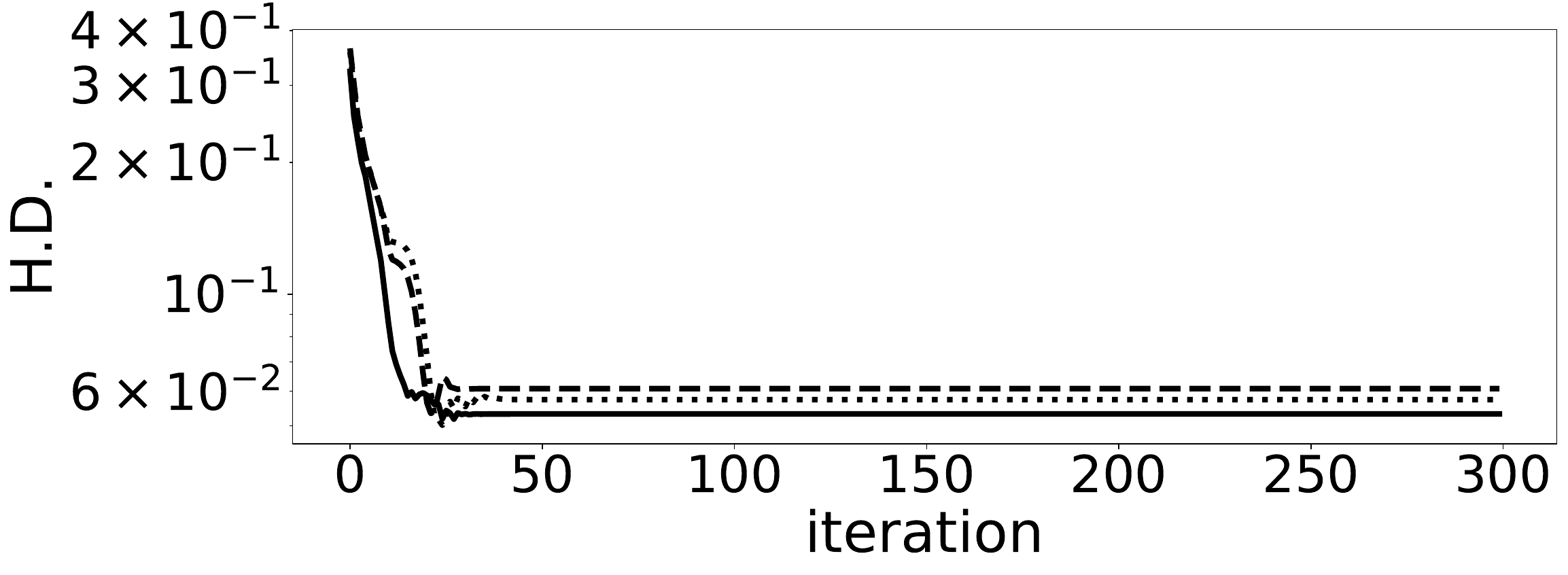}} \\[0.5em]
\resizebox{0.3\linewidth}{!}{\includegraphics{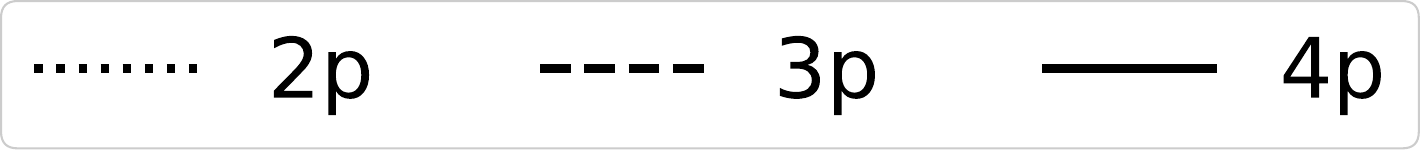}} 
\caption{Histories of (normalized) cost values, Sobolev gradient norms, and Hausdorff distances $d_{H}(\Gamma^{(k)},\Gamma^{\ast})$ for \texttt{Test}($\Gamma_{\text{K}}$)}
\label{fig:figure3}
\end{figure}
\begin{remark}
In Figure \ref{fig:figure2}, it is observed that even with multiple boundary measurements, the numerical results for an $\textsf{L}$-shaped non-convex shape (which violates the regularity assumption) reconstruction seem to be unsatisfactory. 
Possible reasons behind this result may include inaccurate computations of shape gradients due to the non-remeshing technique employed in the reconstruction process, as well as the smoothness loss on the domain during discretization. 
Additionally, mesh qualities may be compromised during shape changes, leading to inaccurate finite element approximations. 
To address this issue, an improved boundary type of shape gradient, as proposed in \cite{GongZhu2021} and \cite{GongLiZhu2022}, may be employed. 
The boundary correction formula therein can be incorporated into numerical algorithms to enhance the accuracy of reconstructions, not only when employing the Neumann-data tracking approach but also when utilizing a boundary-type cost functional more generally.
\end{remark}
%-------------------------------------- NUMERICAL RESULTS  --------------------------------------
%-------------------------------------- NUMERICAL RESULTS  --------------------------------------
%-------------------------------------- NUMERICAL RESULTS  --------------------------------------
%-------------------------------------- NUMERICAL RESULTS  --------------------------------------
%-------------------------------------- NUMERICAL RESULTS  --------------------------------------
\subsection{Employing the Dirichlet data-tracking least-squares approach}
\label{subsec:numerical_experiments_tracking_the_Dirichlet_data}
Using the same algorithm laid out in subsection \ref{subsec:Numerical_Algorithm}, we provide here some numerical experiments for the optimization problem \eqref{equa:shop} with multiple boundary measurements.
For the input data, we consider up to four linearly independent Cauchy pairs for our numerical tests with values for $g^{(i)}$, $i=1,\ldots,4$, given as follows: $g^{(i)} = \sin((i+1)t/2)$, for $i=1,3$, and $g^{(i)} = \cos(it/2)$, for $i =2,4$, for $t \in [0,2\pi)$.
With the above consideration and as before, depending on the value of $M$, we need to replace $G_{{D}}$ in \eqref{eq:smoothing} with $\sum_{i=1}^{M} G_{{D}}^{(i)}$, where, for each $i = 1,2,\ldots,M$, $G_{{D}}^{(i)}$ corresponds to the shape gradient computed with the input data $g^{(i)}$ with the cost function $J_{{D}}^{(i)} = \intS{(\un^{(i)} - f^{(i)})^{2}}$.
%
%
%%%In the forward problem, we consider the same (exact) geometries of the interior boundary $\Gamma$ considered in the previous discussion (i.e., $\Gamma \in \{\Gamma_{\text{K}},\Gamma_{\text{R}},\Gamma_{\text{F}},\Gamma_{\text{L}}\}$).
In the forward problem, we consider the same set of test geometries $\{\Gamma_{\text{K}},\Gamma_{\text{R}},\Gamma_{\text{F}},\Gamma_{\text{L}}\}$ for $\Gamma$.

%------------------------------------------------------------------------------------
% 			NUMERICAL RESULTS AND DISCUSSION
%------------------------------------------------------------------------------------  
%%%%%%%%%%%%%%%%%%%%%%%%%%%%%%%%%%%%%%%%%%%%%%%%%%%%%%%%%%%%%%%%%%%%%%%%%%%%%
\subsubsection{Tests with single boundary measurement}
Again, to prompt the use of more than one set of Cauchy data in the inversion procedure, we first issue some numerical results obtained from a single boundary measurement.
On this purpose, we consider three different inputs for the Neumann flux $g$, namely, (i) $g = 1$, (ii) $g = \sin{t}$, $t\in [0,2\pi)$, and (iii) $g = \cos{t}$, $t\in [0,2\pi)$.

The experimental results carried out for the present problem with exact interior boundaries given by $\Gamma_{\text{L}}$, $\Gamma_{\text{K}}$, $\Gamma_{\text{F}}$, and $\Gamma_{\text{R}}$, are plotted in Figure \ref{fig2:figure4}.
As evident from the plotted figures, the detected shapes are far from the exact geometries and the algorithm was only able to locate the position of the inclusion. 
Moreover, it seems difficult for the method to detect the concave regions of the unknown inclusion. 
Clearly, the detections are far from being acceptable, and hence require attention for improvement.
So, as in subsection \ref{subsec:multiple_measurements_tracking_Neumann}, these motivate us to consider multiple boundary measurements in the inversion process, which we give next in the following subsection.
%
%------------------------------------------------------------------------------------
% 			FIGURE  4: SHAPES
%------------------------------------------------------------------------------------         
\begin{figure}[htp!]
\centering
\resizebox{0.235\linewidth}{!}{\includegraphics{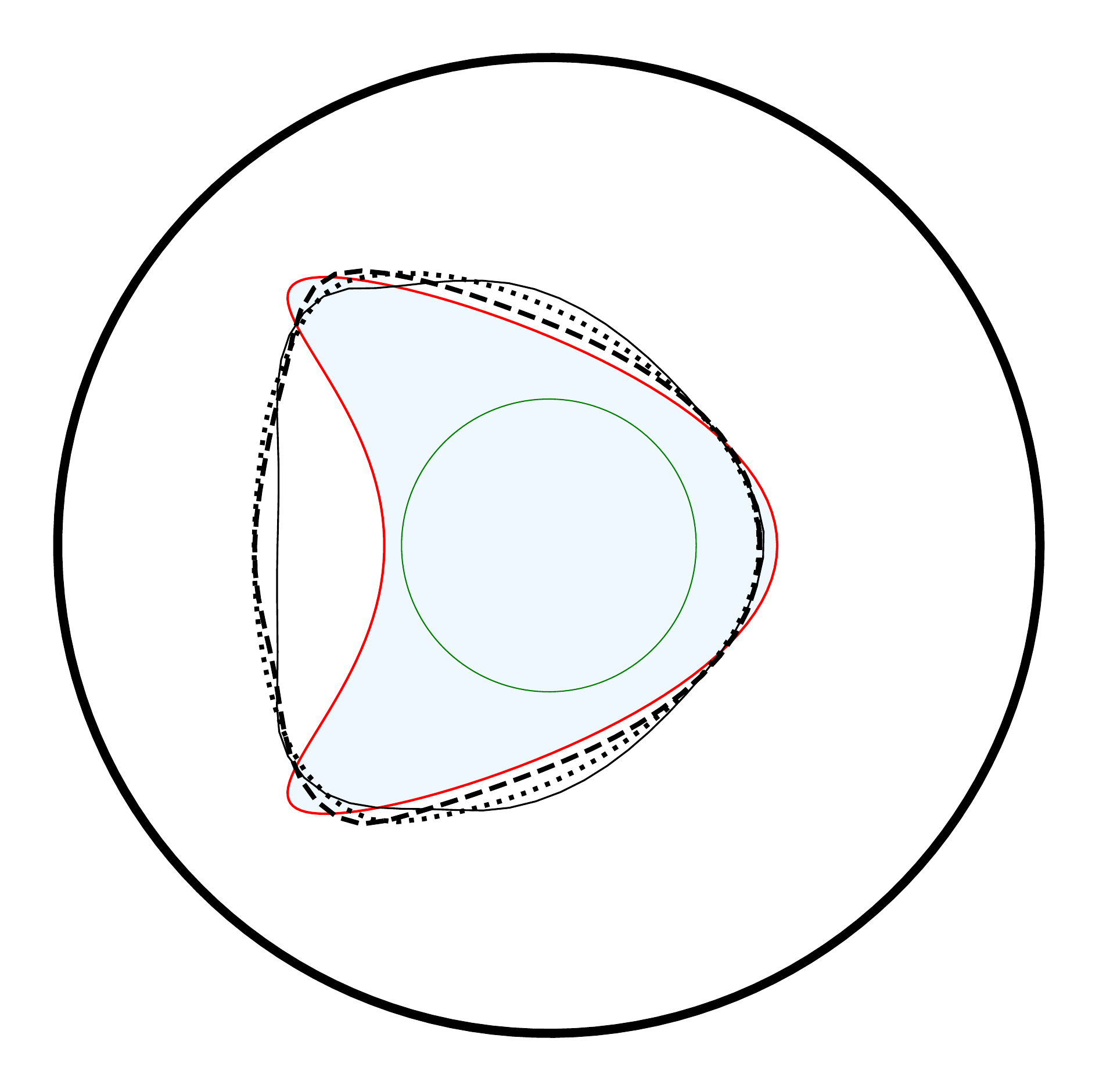}} \
\resizebox{0.235\linewidth}{!}{\includegraphics{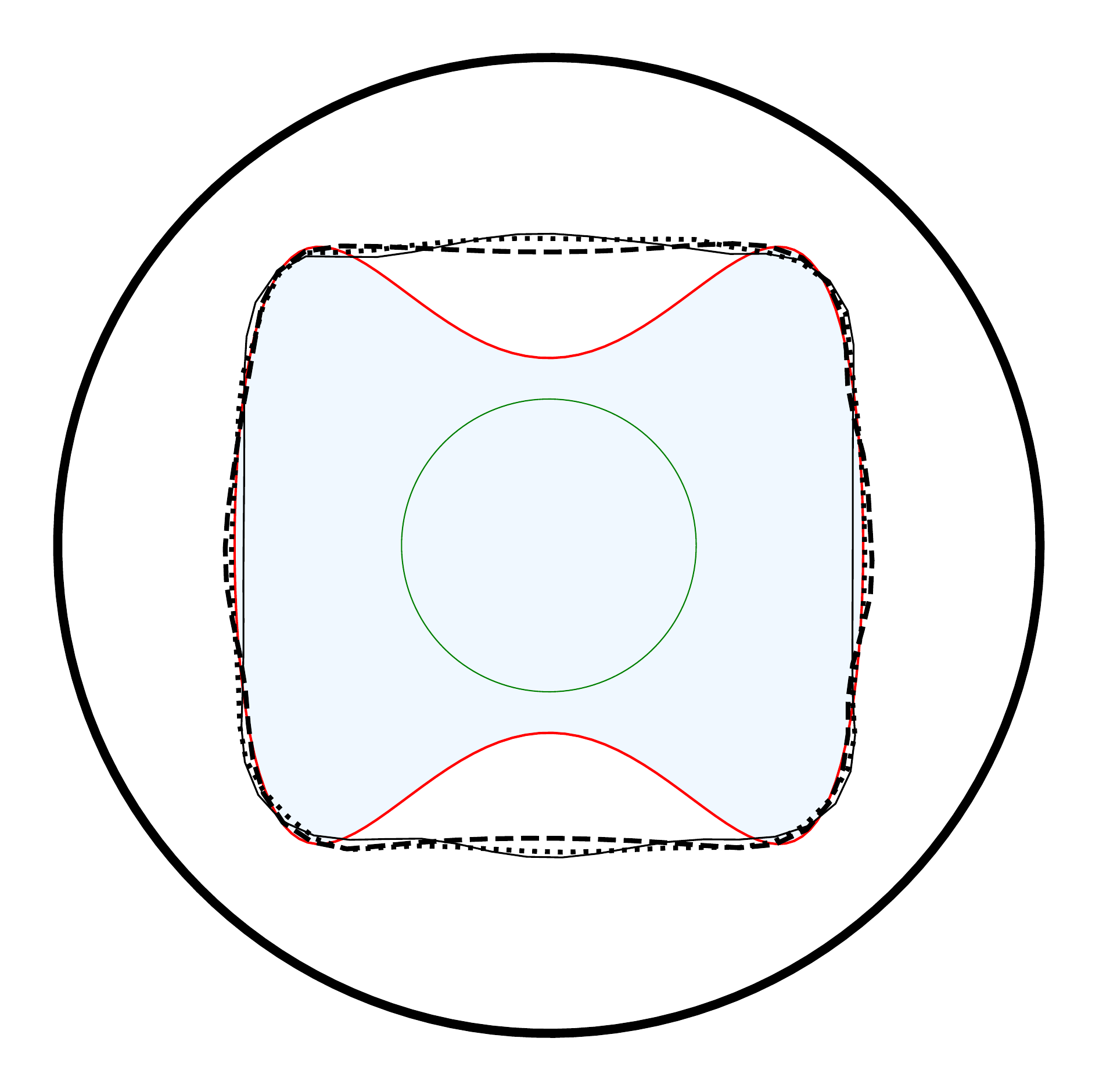}} \
\resizebox{0.235\linewidth}{!}{\includegraphics{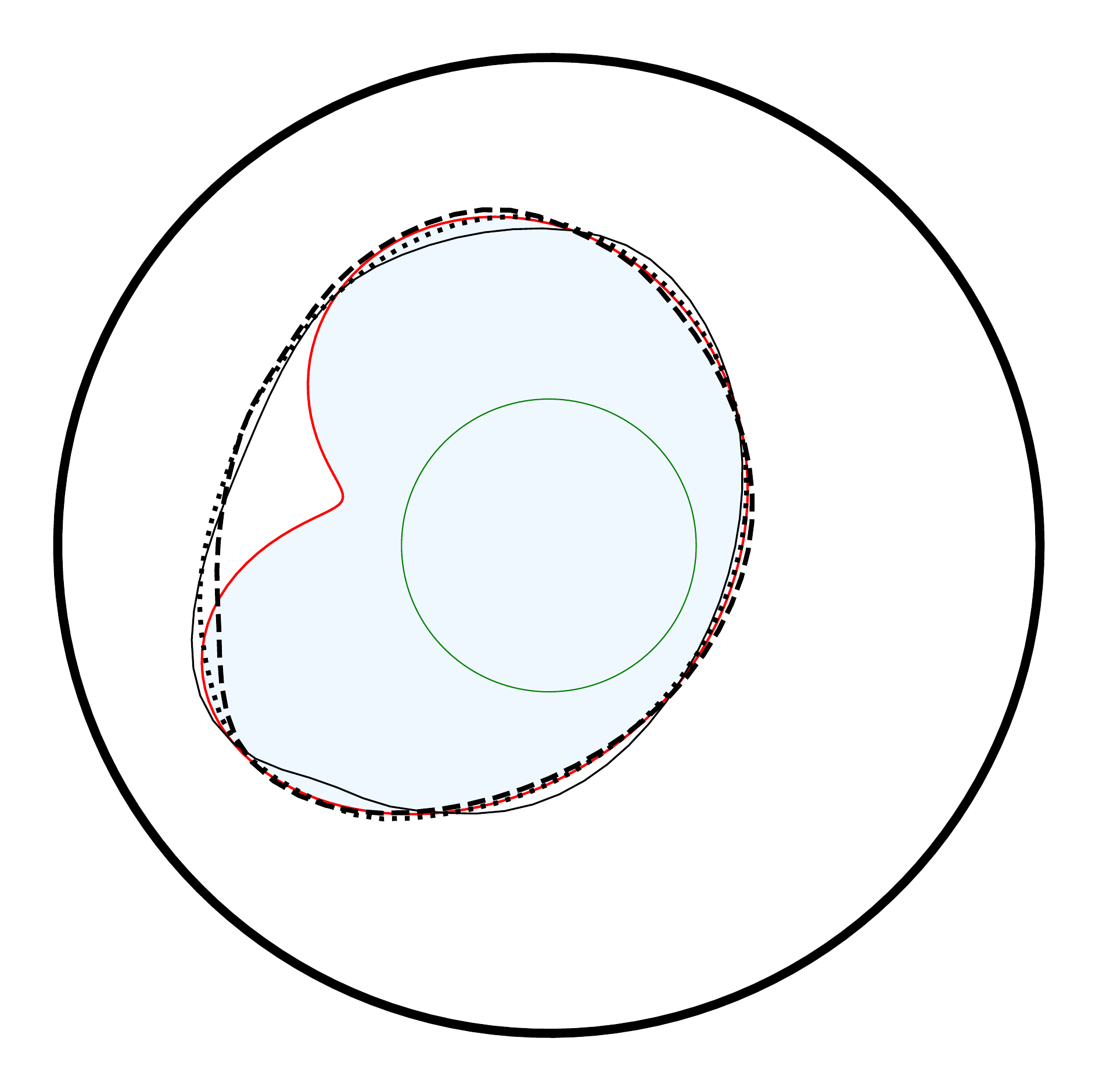}} \
\resizebox{0.235\linewidth}{!}{\includegraphics{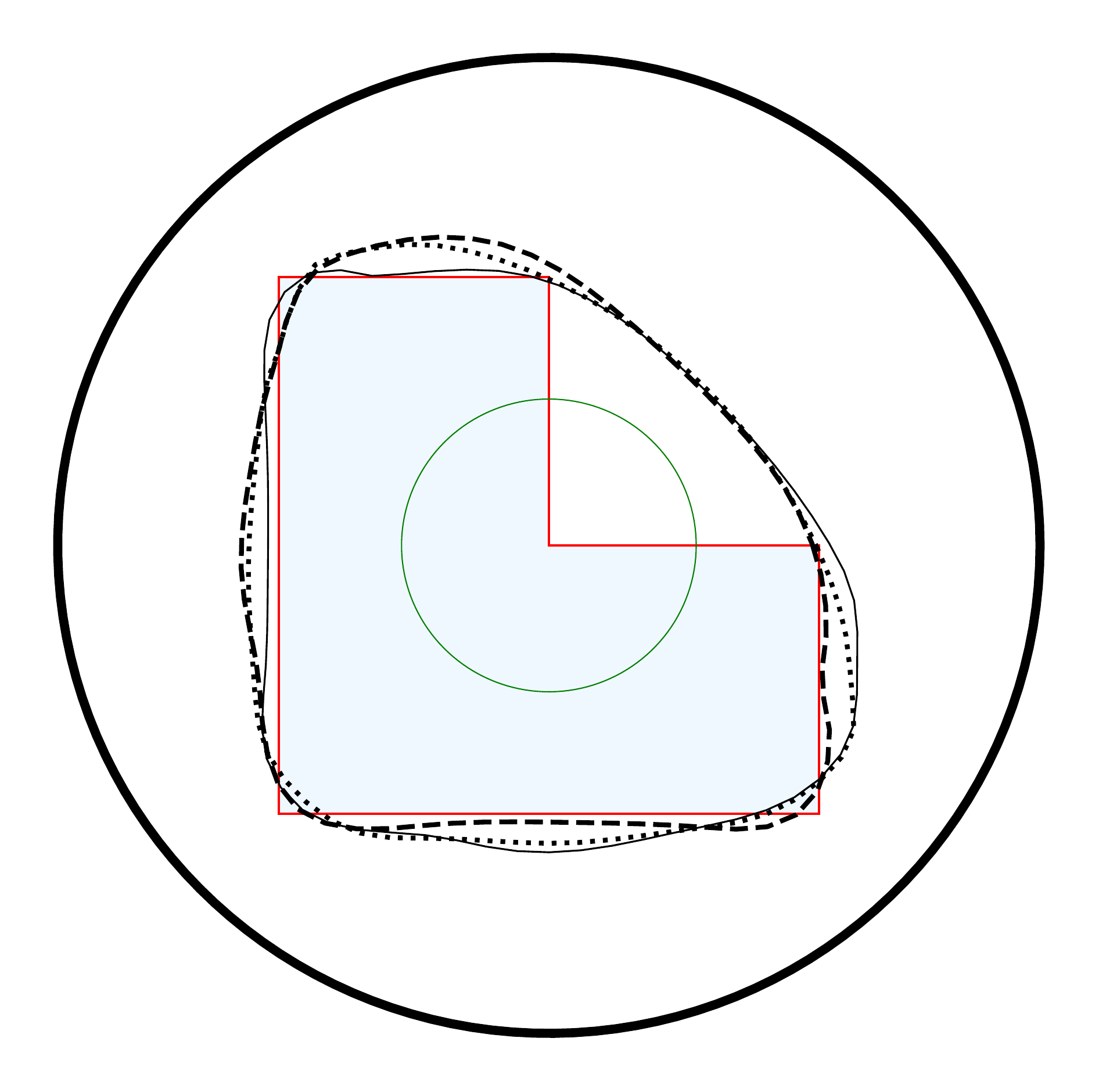}} \\
\resizebox{0.8\linewidth}{!}{\includegraphics{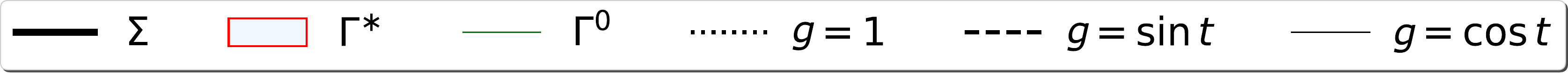}} 
\caption{Results of the detections with a single boundary measurement}
\label{fig2:figure4}
\end{figure}
%
%
%
%
%%%%%%%%%%%%%%%%%%%%%%%%%%%%%%%%%%%%%%%%%%%%%%%%%%%%%%%%%%%%%%%%%%%%%%%%%%%%%
\subsubsection{Tests with multiple Cauchy data}

For our experiments involving multiple measurements, we use the same set of geometries as described in the previous subsection. 
The results of these experiments, where we employ two to four linearly independent input data, are depicted in Figure \ref{fig2:figure5} and Figure \ref{fig2:figure7}, respectively. 
We initialize the process with a circle of radius $0.3$ and $0.6$, as shown in the figures.

Consistent with our prior experiments, employing more than one boundary measurement in the inversion process yields more accurate detections compared to using only a single pair of Cauchy data. 
Notably, the algorithm accurately identifies the concave parts of the exact interior boundary. Particularly for \texttt{Test}($\Gamma_{\text{K}}$), \texttt{Test}($\Gamma_{\text{P}}$), and \texttt{Test}($\Gamma_{\text{R}}$), we achieve highly accurate detections of the exact shapes, with detection improving as the number of boundary measurements increases. 
Additionally, we observe that the reconstruction is more precise when the initial guess is closer to the exact inclusion, as demonstrated in Figure \ref{fig2:figure7}.
However, for \texttt{Test}($\Gamma_{\text{L}}$), even with a close initial guess, the detection of the concave part appears less effective. This limitation could be attributed to the less smooth shape of $\Gamma_{\text{L}}$ compared to the other shapes, along with its distance from the measurement region.

Further insights are provided in Figure \ref{fig2:figure6} and Figure \ref{fig2:figure8}, which summarize the histories of (normalized) cost values, Sobolev gradients' norms, and the Hausdorff distances between the approximate and exact solutions against the number of iterations for \texttt{Test}($\Gamma_{\text{K}}$). 
We observe similar trends as discussed in subsection \ref{subsec:multiple_measurements_tracking_Neumann} for the previous approach. 
In summary, our findings suggest that employing multiple measurements significantly improves the detection of unknown inclusions with non-convex shapes, thereby partially addressing the ill-posedness of the shape inverse problem under consideration.
%
%
%------------------------------------------------------------------------------------
% 			FIGURE  3: SHAPES
%------------------------------------------------------------------------------------         
\begin{figure}[htp!]
\centering
\resizebox{0.235\linewidth}{!}{\includegraphics{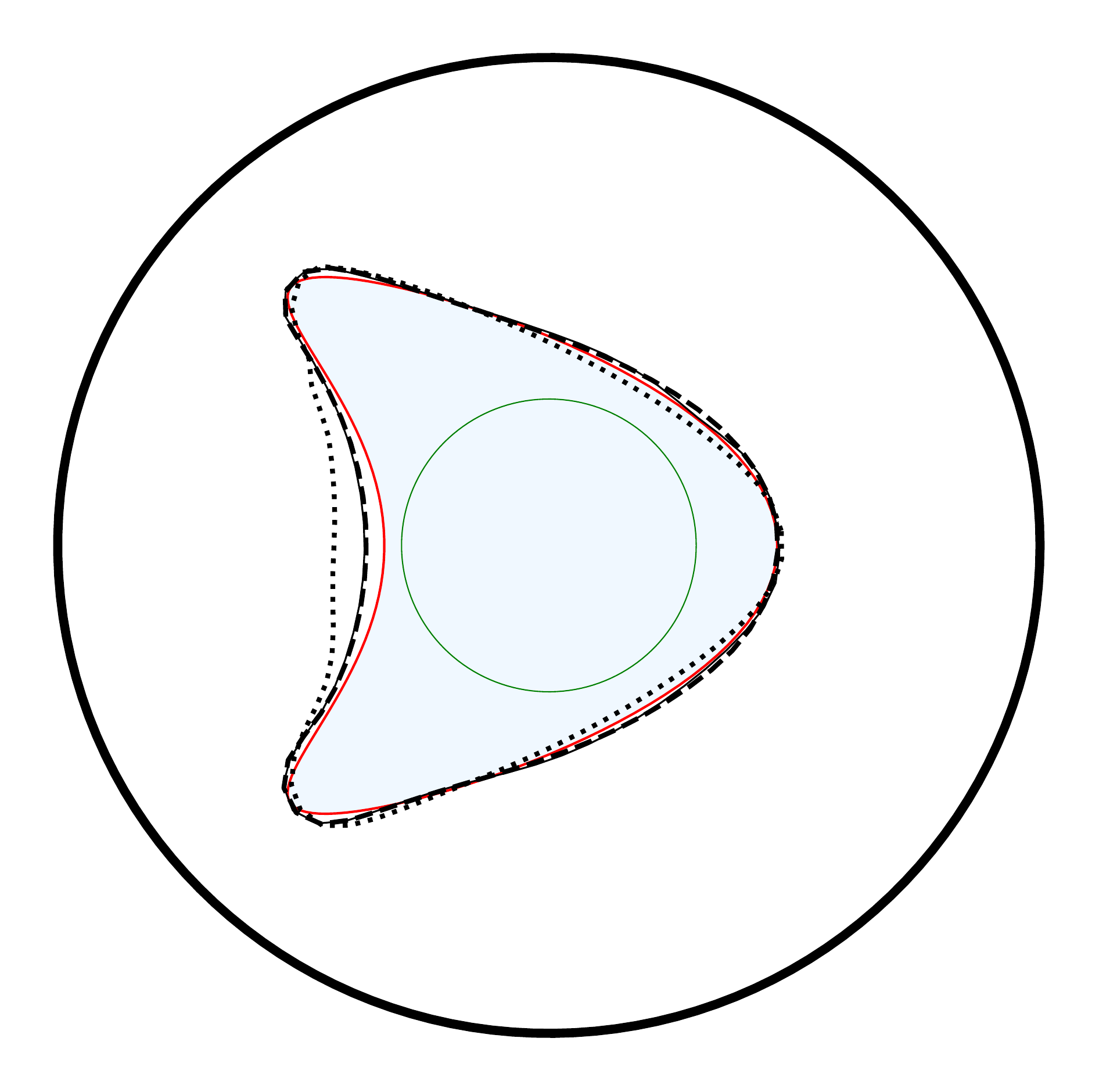}} \
\resizebox{0.235\linewidth}{!}{\includegraphics{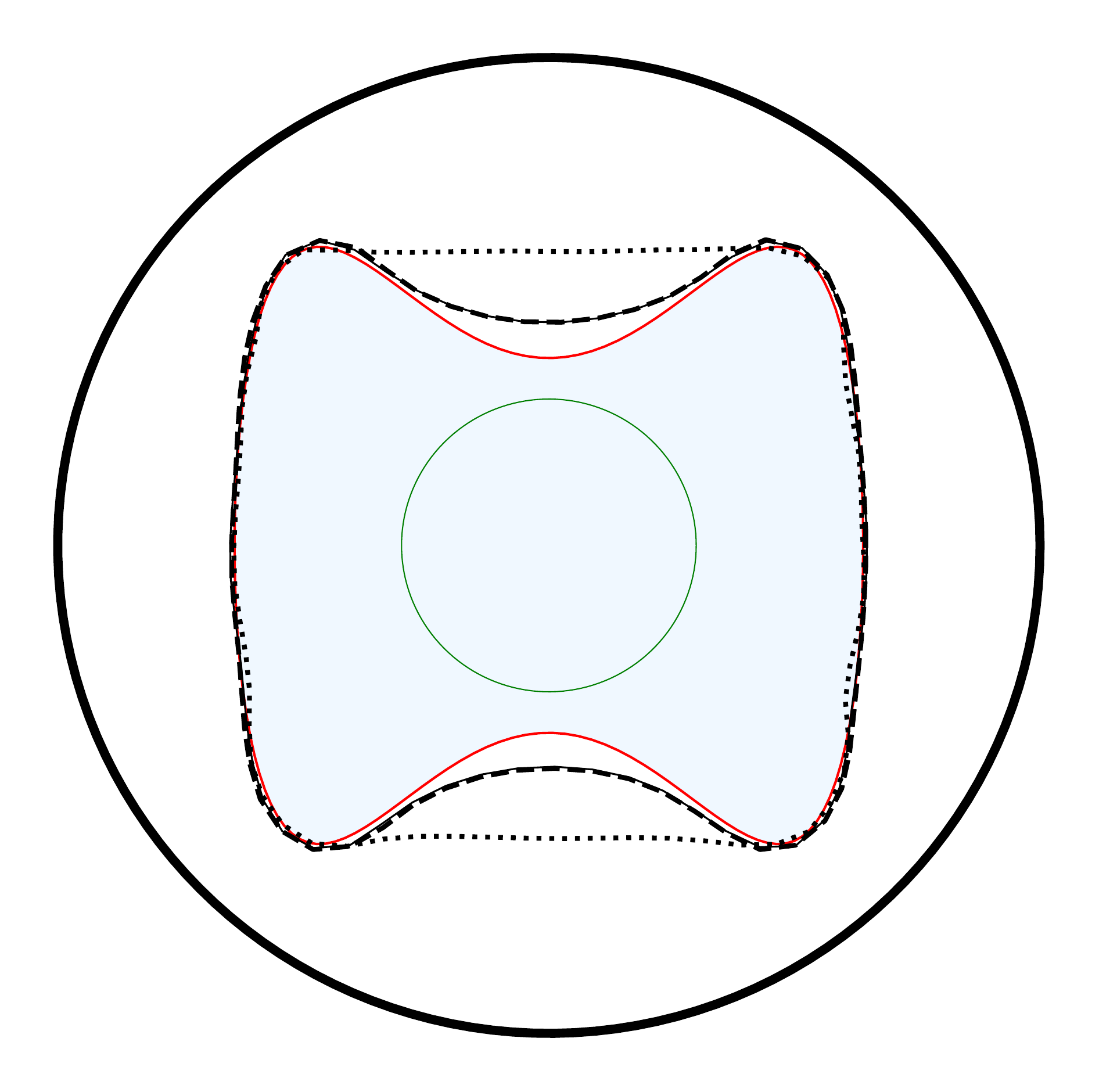}} \
\resizebox{0.235\linewidth}{!}{\includegraphics{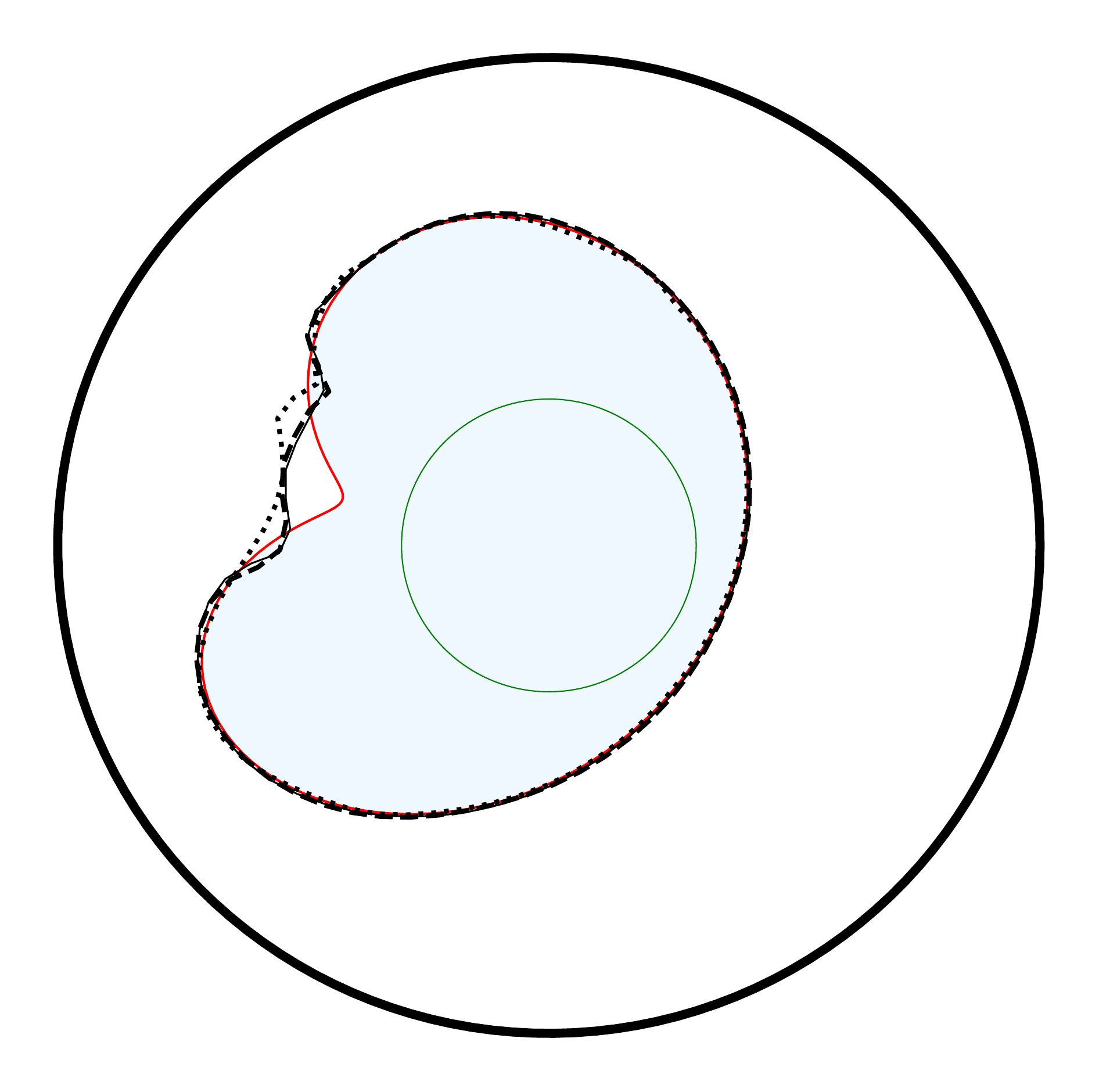}} \
\resizebox{0.235\linewidth}{!}{\includegraphics{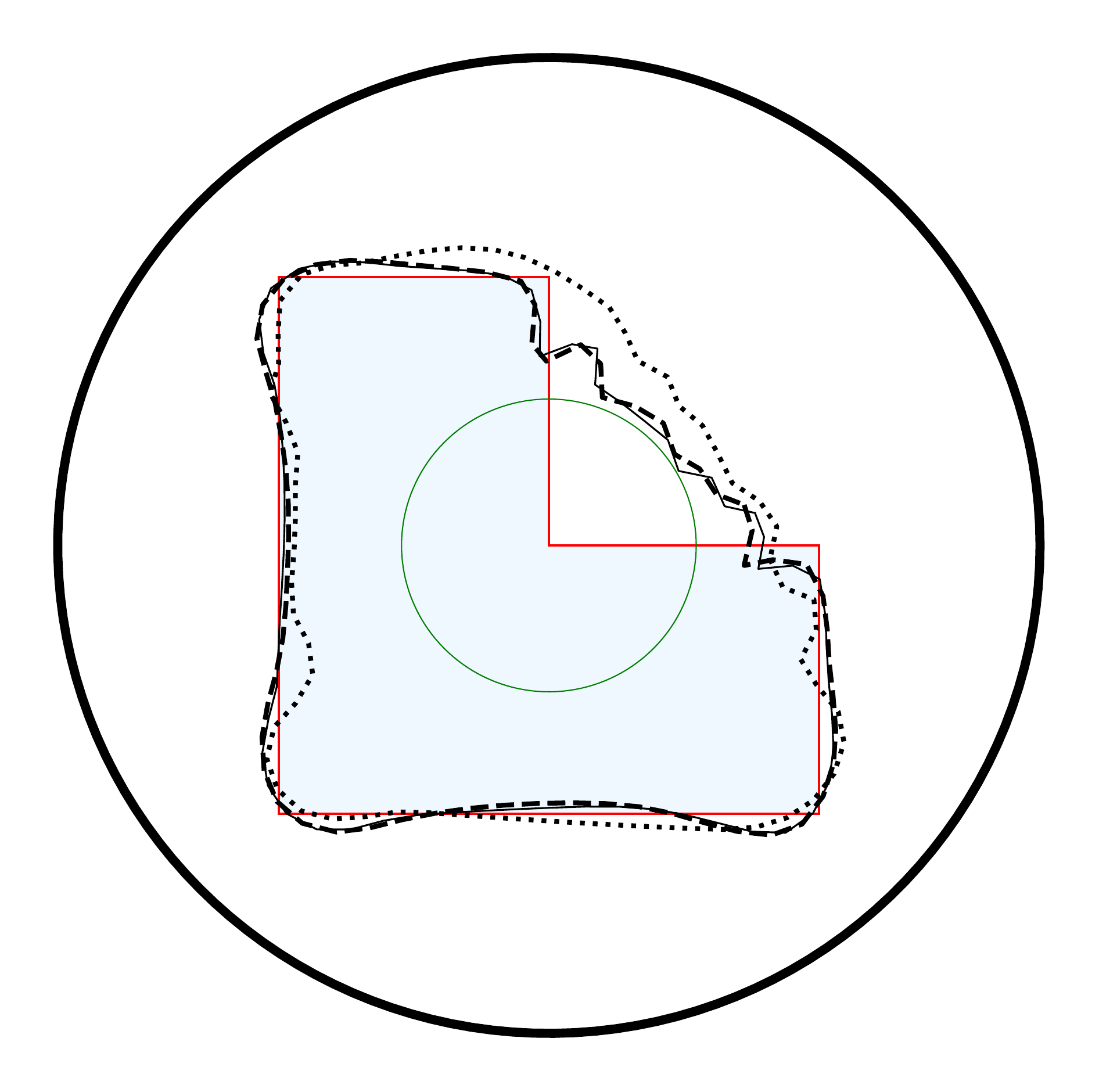}} \\
\resizebox{0.8\linewidth}{!}{\includegraphics{legend2.pdf}} 
\caption{Results of the detections with multiple boundary measurements ($\Gamma^{(0)} = C(\vect{0},0.3)$)}
\label{fig2:figure5}
\end{figure}
%
%
%------------------------------------------------------------------------------------
% 			FIGURE  4: HISTORIES OF VALUES
%------------------------------------------------------------------------------------         
\begin{figure}[htp!]
\centering
\resizebox{0.325\linewidth}{!}{\includegraphics{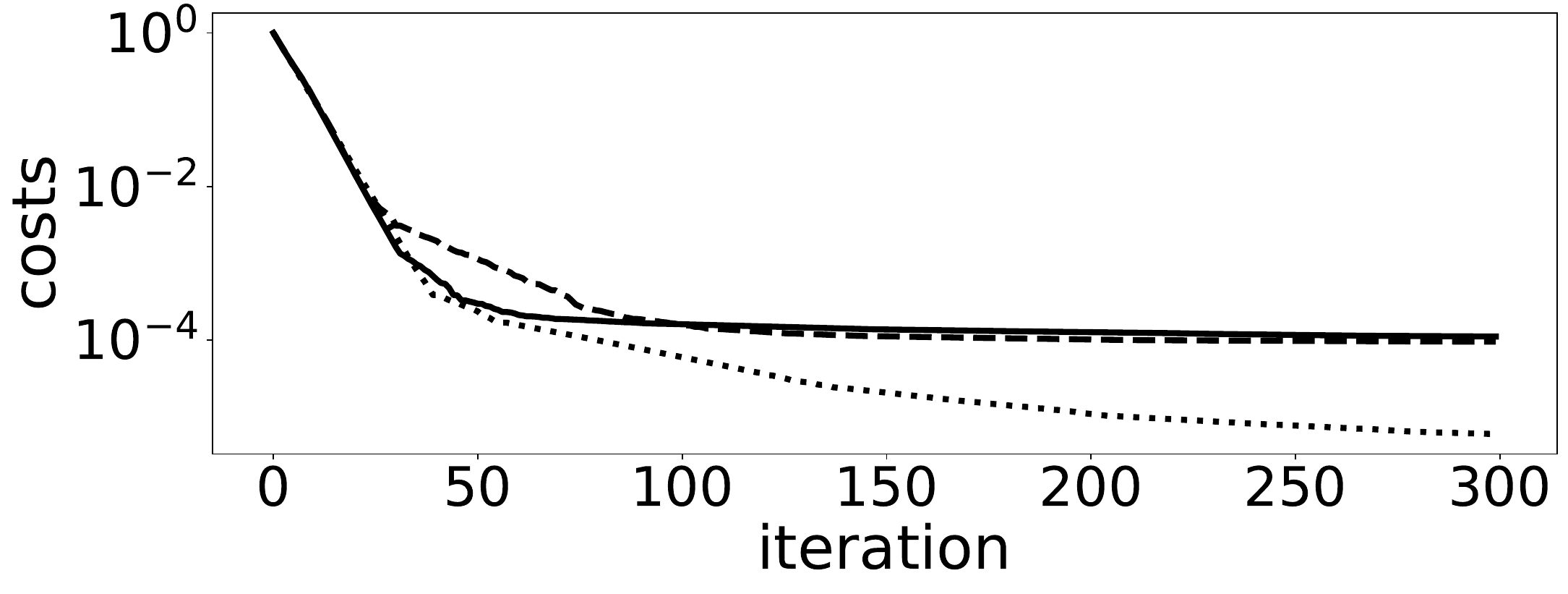}} \hfill
\resizebox{0.325\linewidth}{!}{\includegraphics{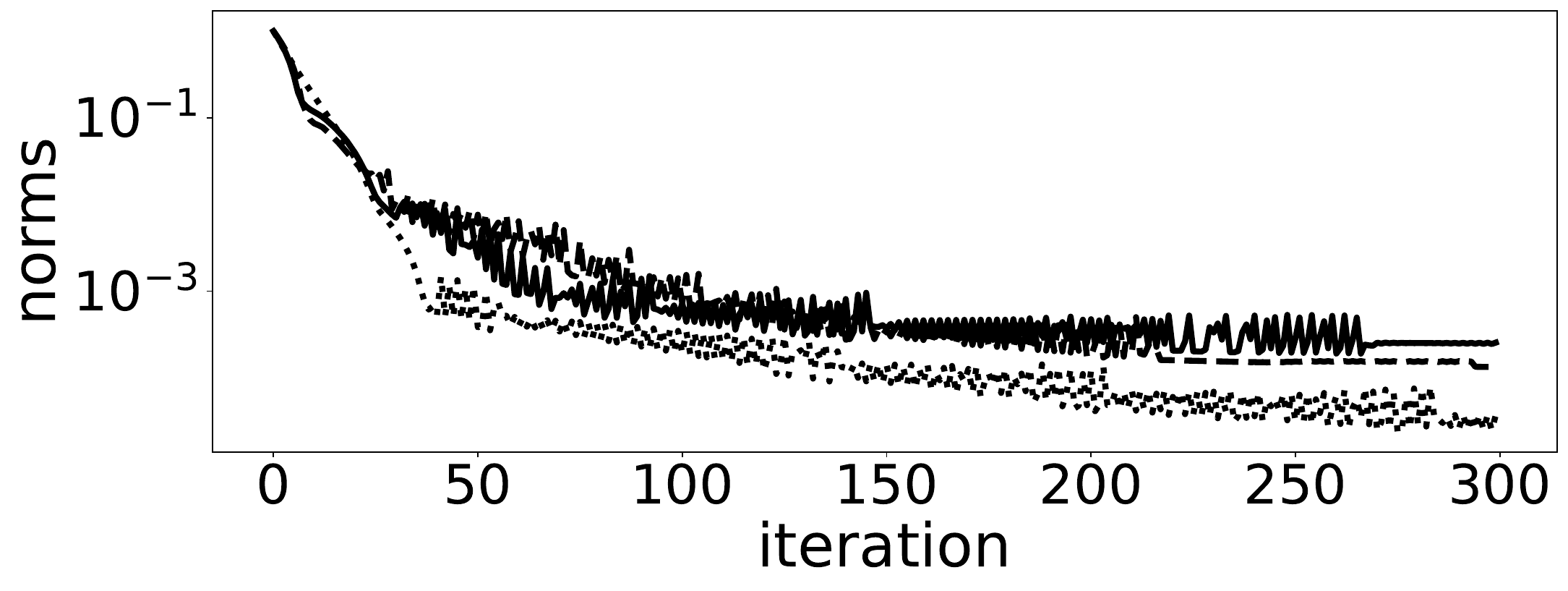}} \hfill
\resizebox{0.325\linewidth}{!}{\includegraphics{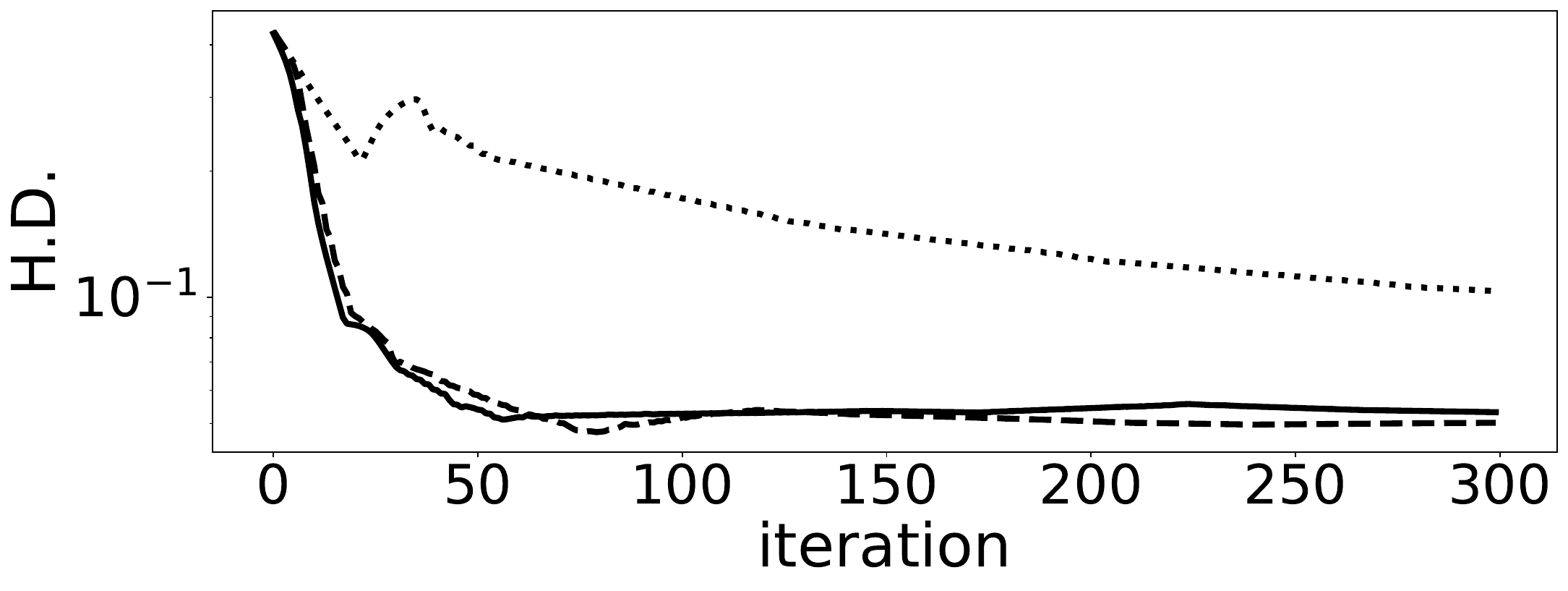}} \\[0.5em]
\resizebox{0.3\linewidth}{!}{\includegraphics{legendlines.pdf}} 
\caption{Histories of (normalized) cost values, Sobolev gradient norms, and Hausdorff distances $d_{H}(\Gamma^{(k)},\Gamma^{\ast})$ for \texttt{Test}($\Gamma_{\text{K}}$)}
\label{fig2:figure6}
\end{figure}%
%
%------------------------------------------------------------------------------------
% 			FIGURE  7: SHAPES
%------------------------------------------------------------------------------------         
\begin{figure}[htp!]
\centering
\resizebox{0.235\linewidth}{!}{\includegraphics{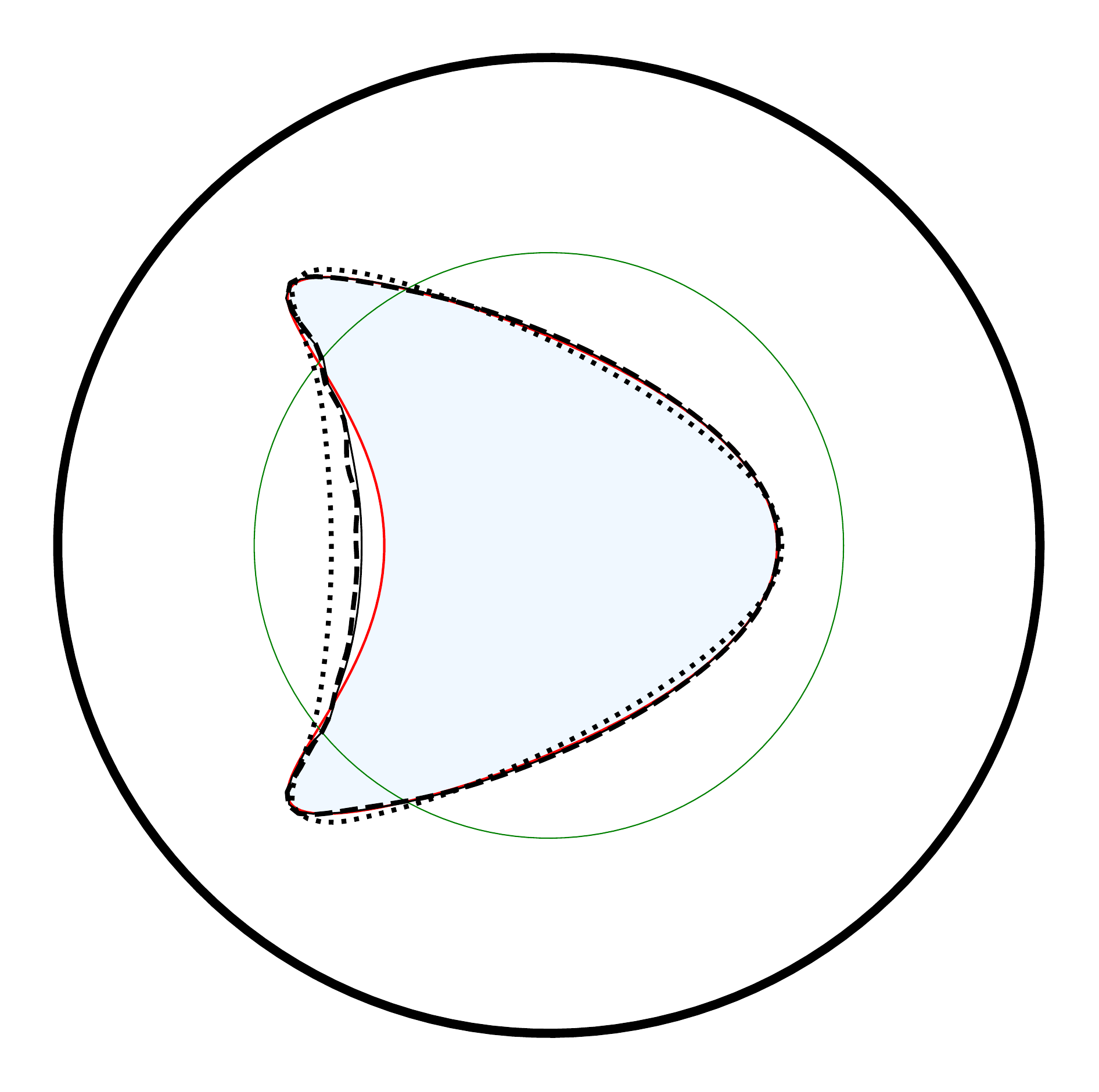}} \
\resizebox{0.235\linewidth}{!}{\includegraphics{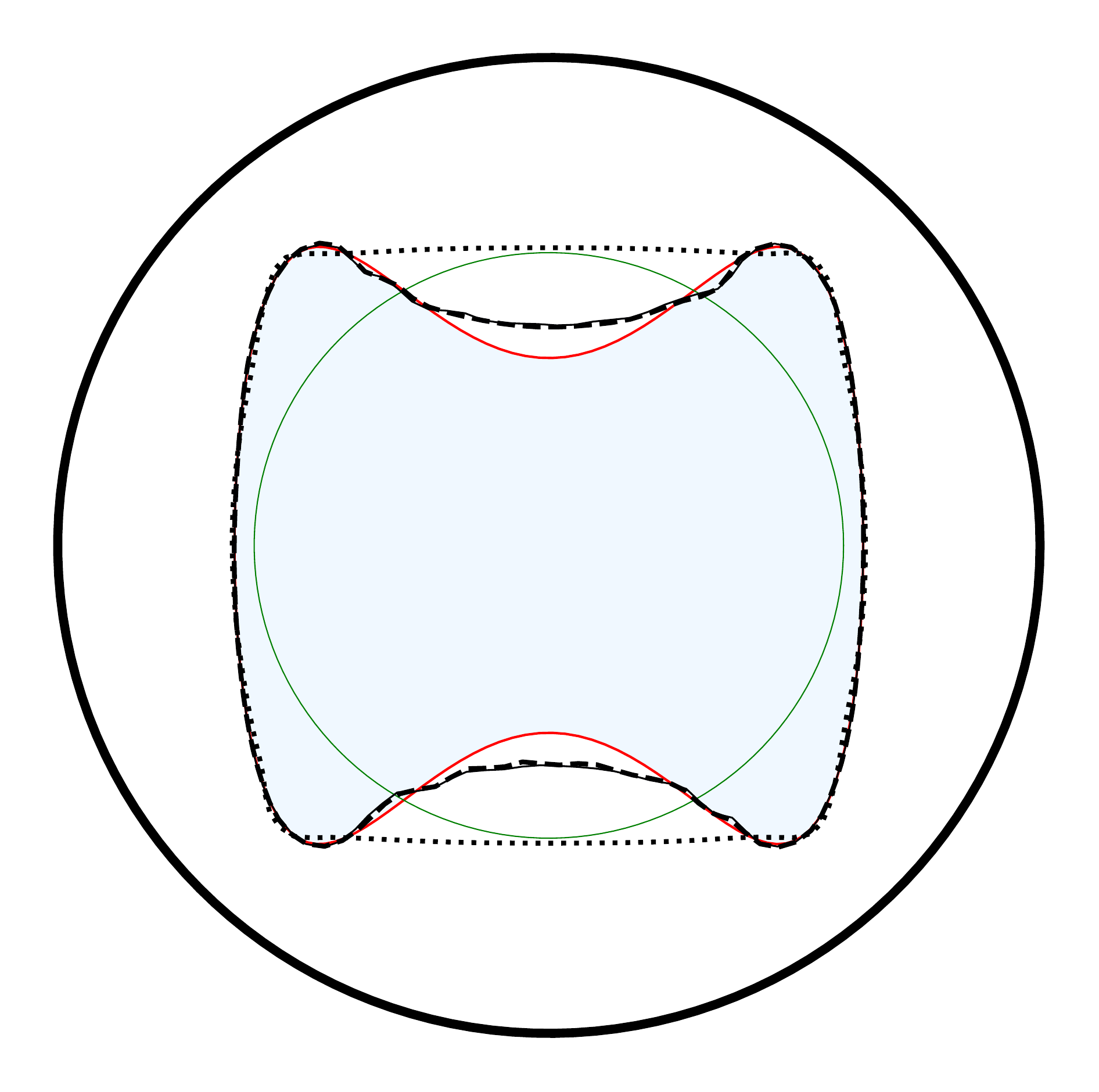}} \
\resizebox{0.235\linewidth}{!}{\includegraphics{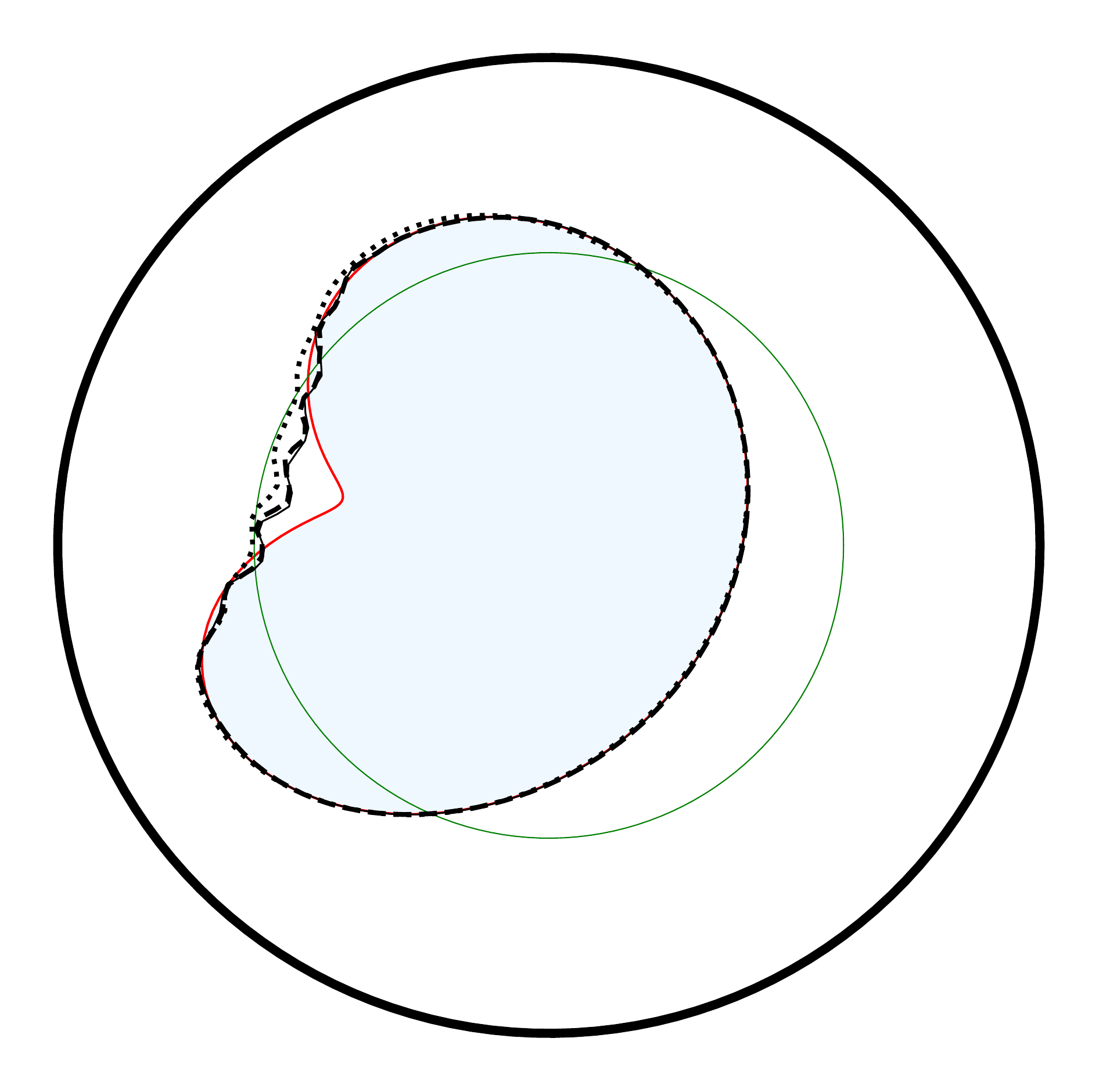}} \
\resizebox{0.235\linewidth}{!}{\includegraphics{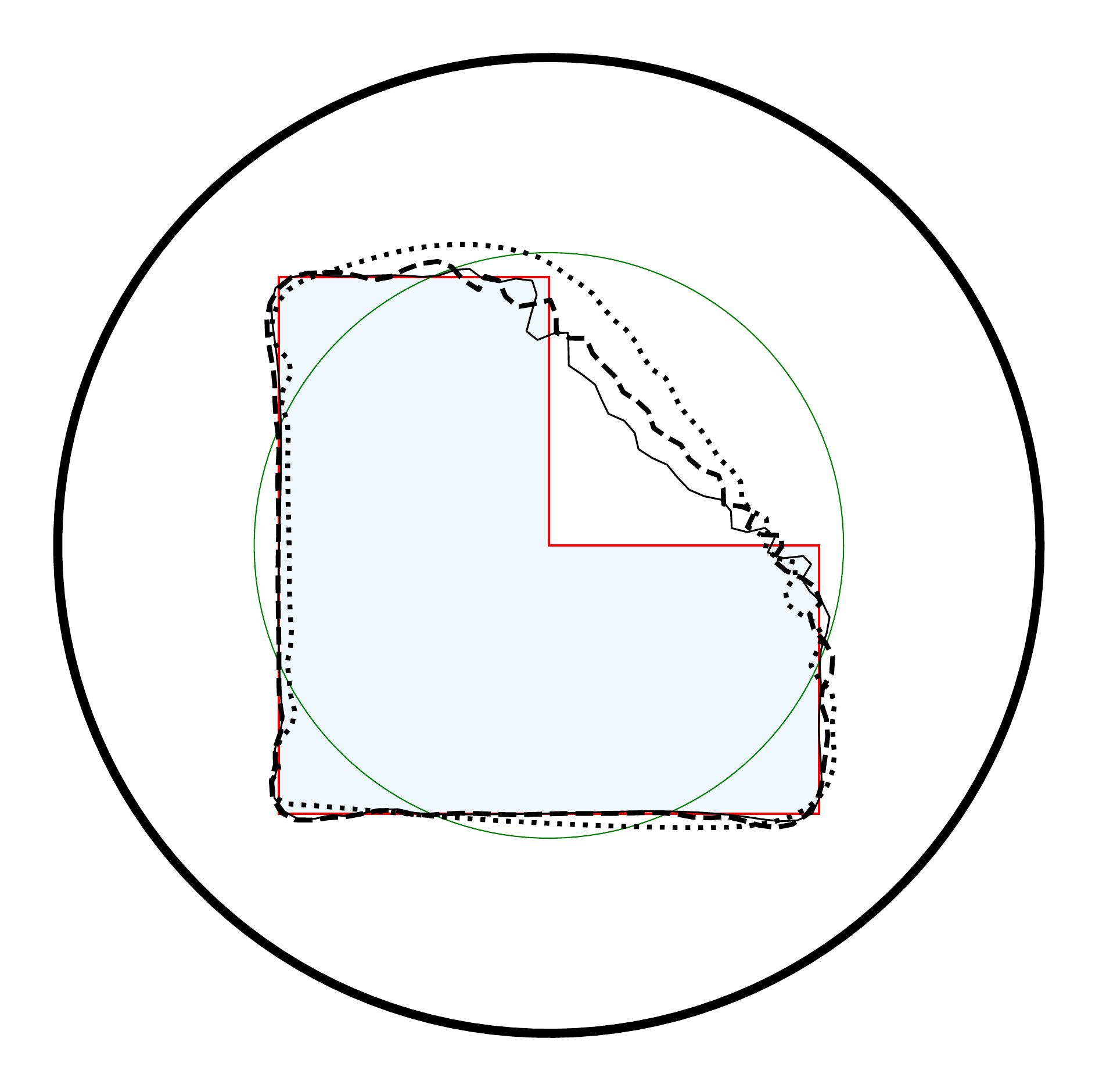}} \\
\resizebox{0.8\linewidth}{!}{\includegraphics{legend2.pdf}} 
\caption{Results of the detections with multiple boundary measurements ($\Gamma^{(0)} = C(\vect{0},0.6)$)}
\label{fig2:figure7}
\end{figure}
%
%
%------------------------------------------------------------------------------------
% 			FIGURE  4: HISTORIES OF VALUES
%------------------------------------------------------------------------------------         
\begin{figure}[htp!]
\centering
\resizebox{0.325\linewidth}{!}{\includegraphics{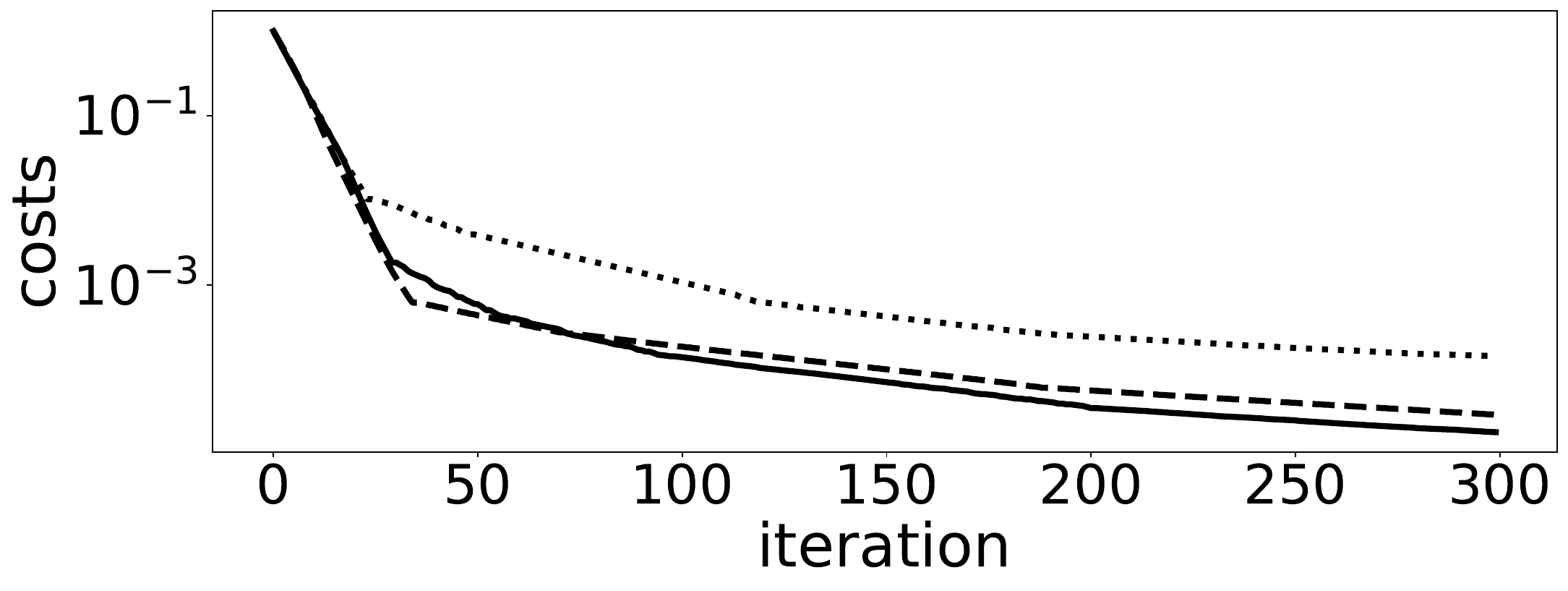}} \hfill
\resizebox{0.325\linewidth}{!}{\includegraphics{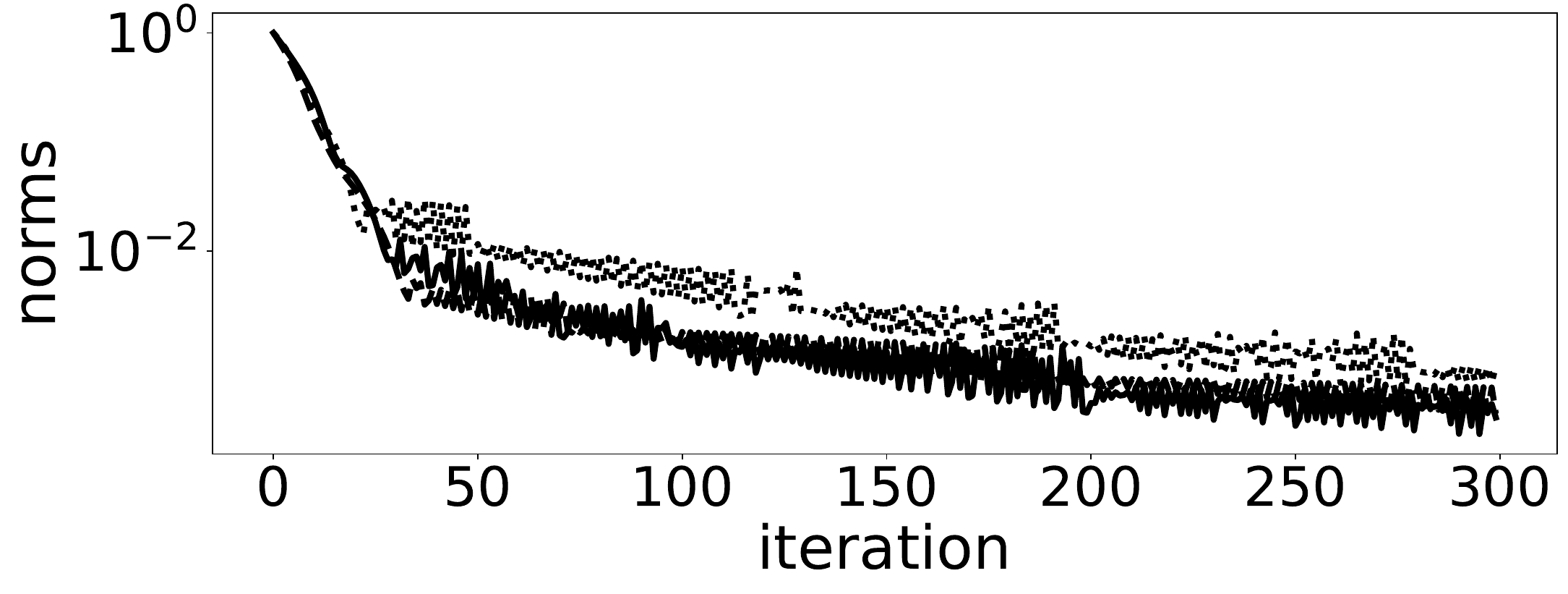}} \hfill
\resizebox{0.325\linewidth}{!}{\includegraphics{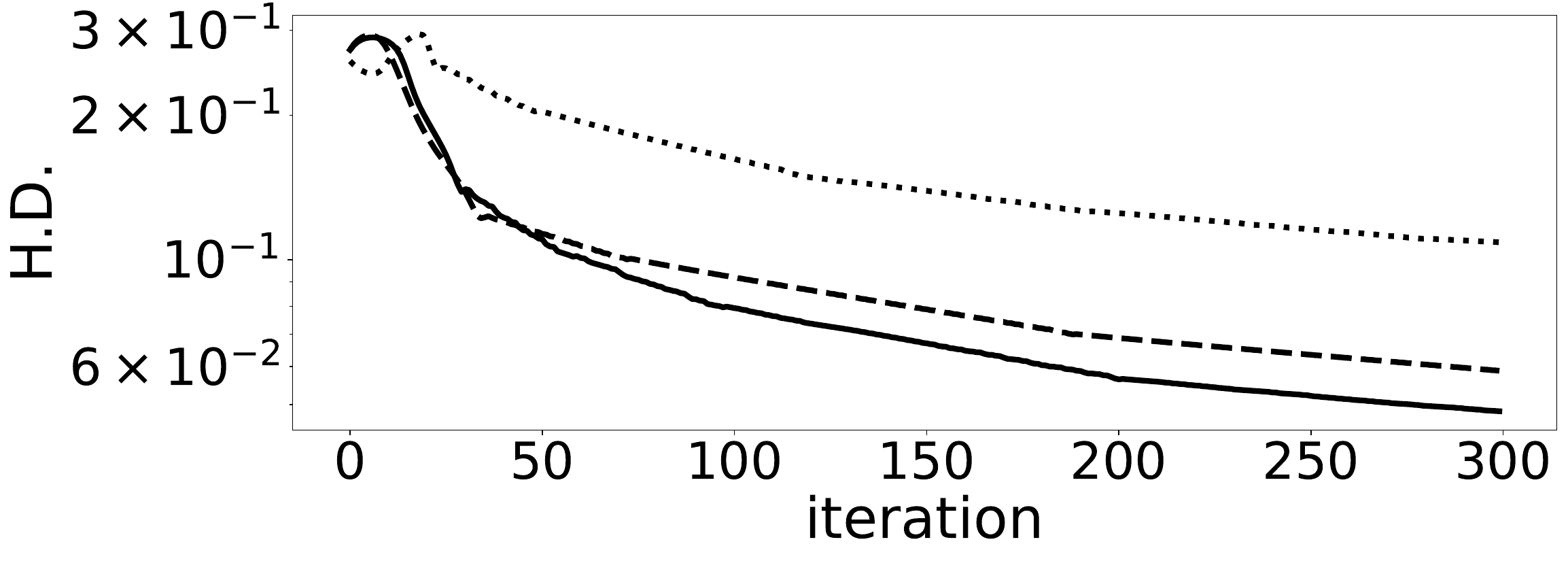}} \\[0.5em]
\resizebox{0.3\linewidth}{!}{\includegraphics{legendlines.pdf}} 
\caption{Histories of (normalized) cost values, Sobolev gradient norms, and Hausdorff distances $d_{H}(\Gamma^{(k)},\Gamma^{\ast})$ for \texttt{Test}($\Gamma_{\text{K}}$)}
\label{fig2:figure8}
\end{figure}%
%%%%%%%%%%%%%%%%%%%%%%%%%%%%%%%%%%%%%%%%%%%%%%%%%%%%%%%%%%%%%%%%%%%%%%%%%%%%%
\subsubsection{A test case in three dimension} %%% {\color{red} This section needs to be revised!}
For the final test case, we consider a problem setup in three dimensions focusing on examining the effect of different combinations of the Cauchy data (where one of the prescribed data is always chosen to be strictly positive). 
With the object's accessible surface given by a sphere of unit radius, the exact geometry of the unknown inclusion is depicted in Figure \ref{fig:3D_exact_view_all}.
The algorithm used to solve the problem is the same as in the two-dimensional case. 
However, this time, we stop the algorithm after a finite number of iterations or when the step size becomes very small (note that the step size is calculated using a backtracking procedure).

Moreover, the computational setup is similar to the two-dimensional case, but with modifications to fit the three-dimensional test case. 
Specifically, we choose the initial guess to be a sphere with radius $0.95$. 
The forward problem is solved with maximum and minimum mesh widths $h_{\max}^{\ast} = 0.08$ and $h_{\min}^{\ast} = 0.06$, and the exact solution is computed using ${P}{2}$ finite elements. 
For the inversion process, the initial computational mesh is set to have maximum and minimum mesh widths of $h{\max} = h_{\min} = 0.125$. 
Lastly, the variational problems corresponding to the state and adjoint problems are solved using ${P}{1}$ finite elements.

Table \ref{tab:table} summarizes the choice of the prescribed boundary data $g$ used in the experiments, and the corresponding numerical results are also depicted in the table.

\begin{table}[htp!]\resizebox{\columnwidth}{!}{%
{\renewcommand{\arraystretch}{1.2} %<- modify value to suit your needs
\begin{tabular}{c | c c c c c c c c c c c} 
 \toprule
 Input & Fig. \ref{fig:3D_case1} & Fig. \ref{fig:3D_case2} & Fig. \ref{fig:3D_case3} & Fig. \ref{fig:3D_case4} & Fig. \ref{fig:3D_case5} & Fig. \ref{fig:3D_case6} & Fig. \ref{fig:3D_case7} & Fig. \ref{fig:3D_case8} & Fig. \ref{fig:3D_case9} & Fig. \ref{fig:3D_case10} & Fig. \ref{fig:3D_case10} \\
 \midrule
 Case & \texttt{C1} & \texttt{C2} & \texttt{C3} & \texttt{C4} & \texttt{C5} & \texttt{C6} & \texttt{C7} & \texttt{C8} & \texttt{C9} & \texttt{C10} & \texttt{C11} \\  
 \midrule
 $g^{(1)}$ & $1$ & $1$ & $1$ & $1$ & $0.1$ & $0.1$ & $0.1$ & $0.1$ & $0.1$ & $0.1$ & $0.1$ \\ 
 \hline
 $g^{(2)}$ & $-$ & $\sin \theta$ & $\cos \theta$ & $\sin \theta$ & $-$ & $\sin \theta$ & $\sin \theta$ & $0.2 \cos \theta$ & $0.2 \sin \theta$ & $0.2 \sin \theta$ & $0.3 \sin \theta$ \\
 \hline
 $g^{(3)}$ & $-$ & $-$ & $-$ & $\cos \theta$ & $-$ & $-$ & $\cos \theta$ & $0.2 \cos \theta$ & $0.2 \cos \theta$ & $0.2 \cos \theta$ & $0.3 \cos \theta$ \\
 \hline
 $g^{(4)}$ & $-$ & $-$ & $-$ & $-$ & $-$ & $-$ & $-$ & $-$ & $0.6 \cos \phi$ & $0.6 \cos \phi$ & $0.5 \cos 2 \phi$ \\
 \hline
 $g^{(5)}$ & $-$ & $-$ & $-$ & $-$ & $-$ & $-$ & $-$ & $-$ & $-$ & $0.6 \sin \phi$ & $-$  \\\bottomrule
\end{tabular}
}}
\caption{Input data $\{g^{(i)}\}_{i=1}^{M}$, $M=1,2,3,4,5$, for testing in three dimensions}\label{tab:table}
\end{table}
Observing Figures \ref{fig:3D_case1} through \ref{fig:3D_case4}, \ref{fig:3D_case5} through \ref{fig:3D_case7}, and \ref{fig:3D_case8} through \ref{fig:3D_case10}, we note an improvement in detecting the unknown inclusion as the number of data used in the inversion increases. 
However, compared to the two-dimensional case, the detection of concave parts is less pronounced.

Using more pairs of Cauchy data results in higher accuracy in detecting concave parts of the unknown interior boundary, although not entirely satisfactory compared to just using a single measurement, as expected. With a single measurement, only the convex hull of the unknown inclusion is detected, as depicted in Figures \ref{fig:3D_case1} and \ref{fig:3D_case5}. The choice of prescribed data $\{g^{(i)}\}_{i=1}^{M}$ greatly influences the detection results, as seen in the comparison between Figures \ref{fig:3D_case2} and \ref{fig:3D_case6}, Figures \ref{fig:3D_case4}, \ref{fig:3D_case7}, and \ref{fig:3D_case8}, and Figures \ref{fig:3D_case9} and \ref{fig:3D_case10}.
Nevertheless, regions with concavities in the inclusion were satisfactorily detected by our scheme, particularly with the input data $(g^{(1)}, g^{(2)}, g^{(3)}, g^{(4)}) = (0.1,0.3\cos \theta, 0.3\sin \theta, 0.5\cos \phi)$, as shown in Figure \ref{fig:3D_case10_view_all}. Additionally, plots in Figure \ref{fig:3D_case11_cost_and_norm} depict the histories of costs and gradient norms obtained from the data set $(g^{(1)}, g^{(2)}, g^{(3)}, g^{(4)}, g^{(5)}) = (0.1,0.2\cos \theta, 0.2\sin \theta, 0.6\cos \phi, 0.6\sin \phi)$. Notably, schemes with multiple measurements converge faster than those with only one measurement, thereby improving both detection quality and convergence behavior.

In conclusion, while finer meshes and adaptive remeshing could enhance our detections, our current results with coarse meshes are satisfactory. We reserve refinements for future investigations. Overall, our proposed strategy of utilizing multiple boundary measurements significantly enhances the detection of unknown interior boundaries. Although effectiveness depends on the choice and combinations of boundary data, the methodology proves reliable in addressing the challenge of detecting concave parts in unknown inclusions. We anticipate its effectiveness in addressing more complex and general inverse problems.

%
%
%
%------------------------------------------------------------------------------------
% 			FIGURE  17: EXACT 3D GEOMETRY
%------------------------------------------------------------------------------------         
\begin{figure}[htp!]
\centering
\resizebox{0.225\linewidth}{!}{\includegraphics{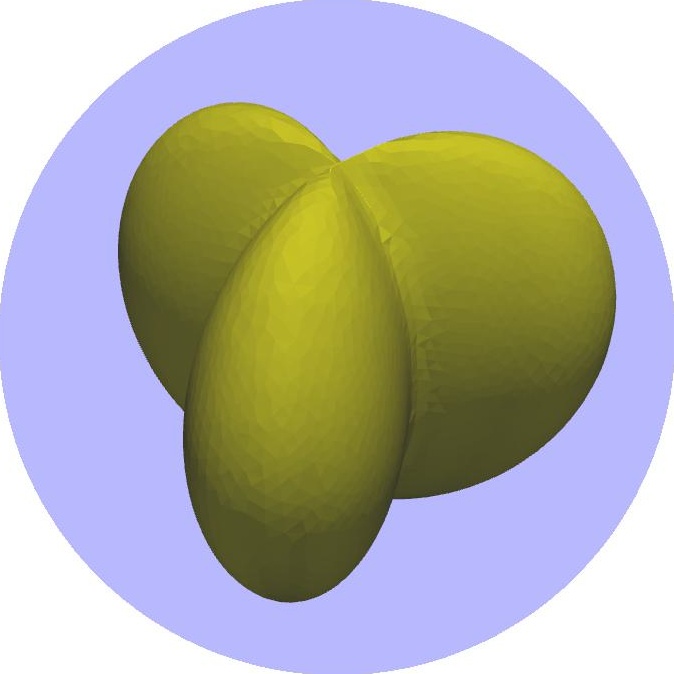}} \hfill
\resizebox{0.225\linewidth}{!}{\includegraphics{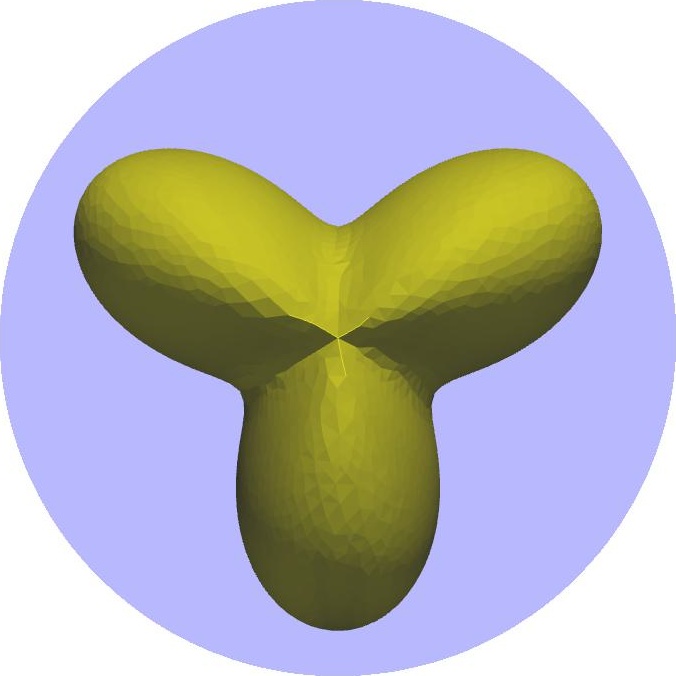}} \hfill
\resizebox{0.225\linewidth}{!}{\includegraphics{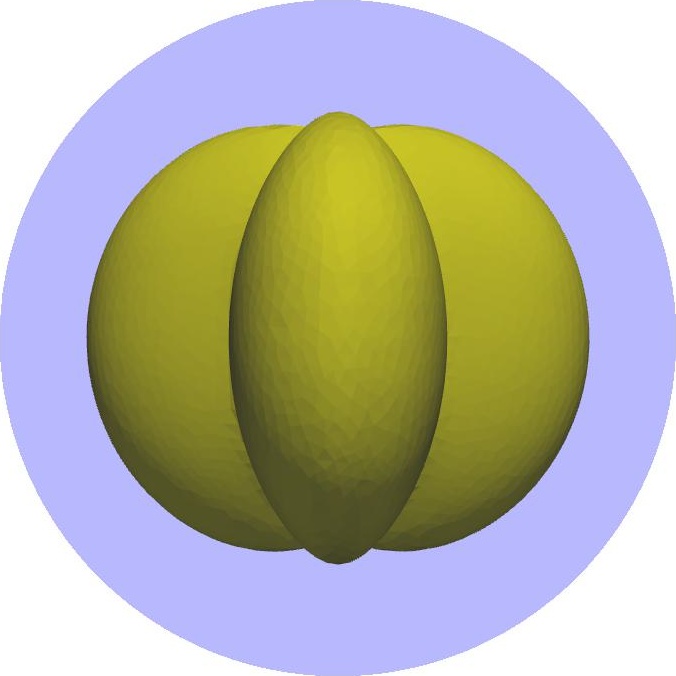}} \hfill
\resizebox{0.225\linewidth}{!}{\includegraphics{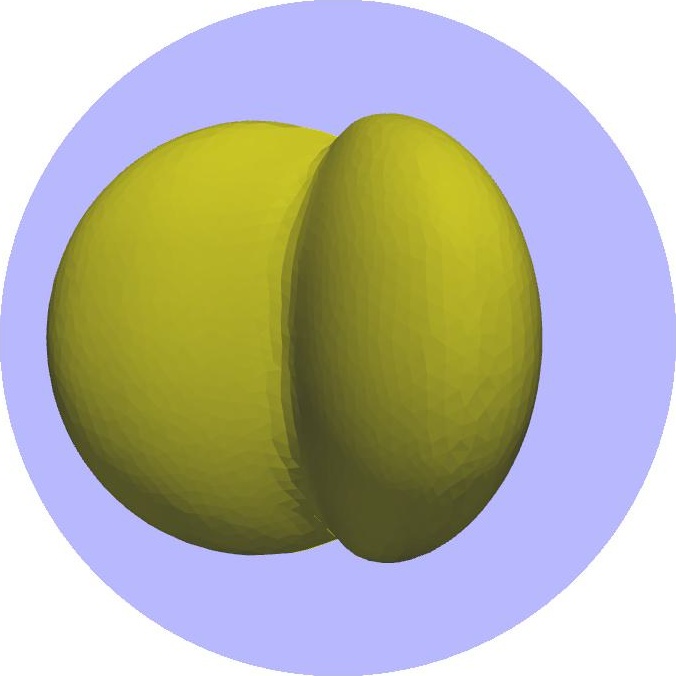}} \hfill
\caption{Exact geometric profile of the unknown inclusion}
\label{fig:3D_exact_view_all}
\end{figure}
\begin{figure}
  \centering
  \begin{subfigure}[t]{0.16\linewidth}
    \resizebox{\linewidth}{!}{\includegraphics{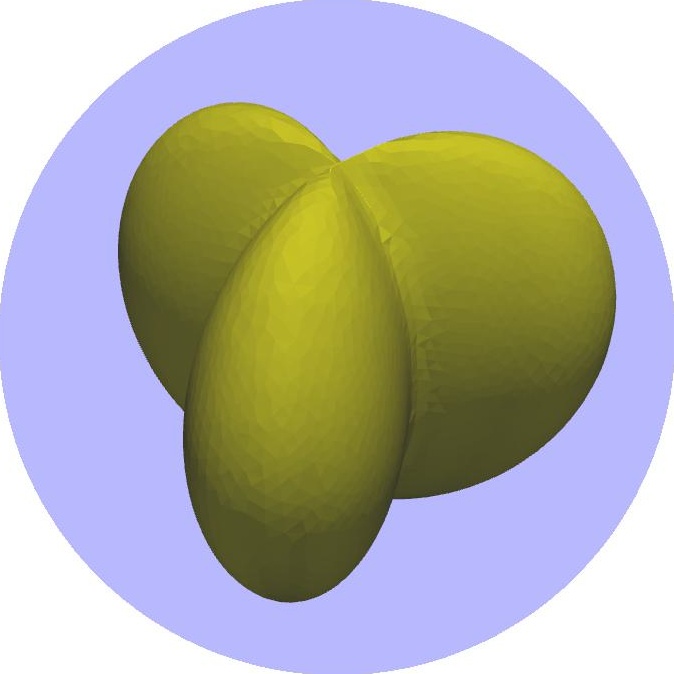}}
    \caption{Exact shape}\label{fig:3D_exact_0}
  \end{subfigure}  
  \begin{subfigure}[t]{0.16\linewidth}
    \resizebox{\linewidth}{!}{\includegraphics{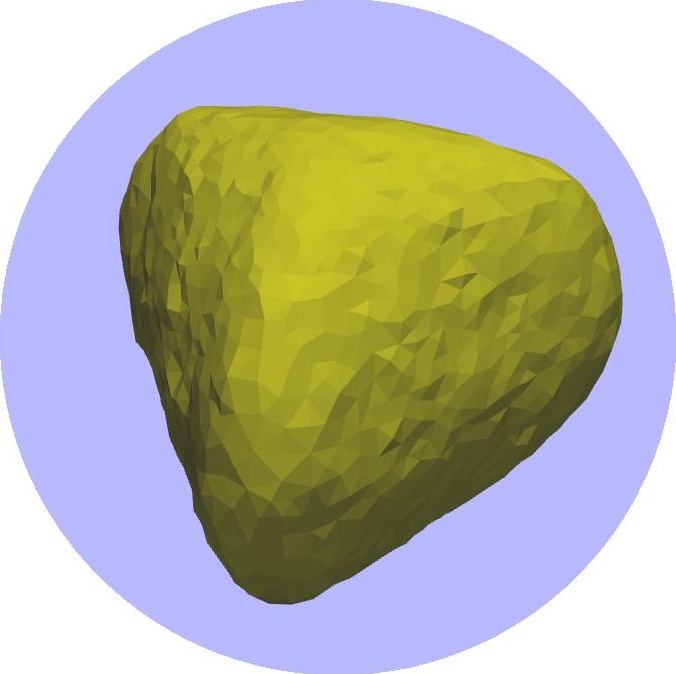}}
    \caption{Case \texttt{C1}}\label{fig:3D_case1}
  \end{subfigure}
  \begin{subfigure}[t]{0.16\linewidth}
    \resizebox{\linewidth}{!}{\includegraphics{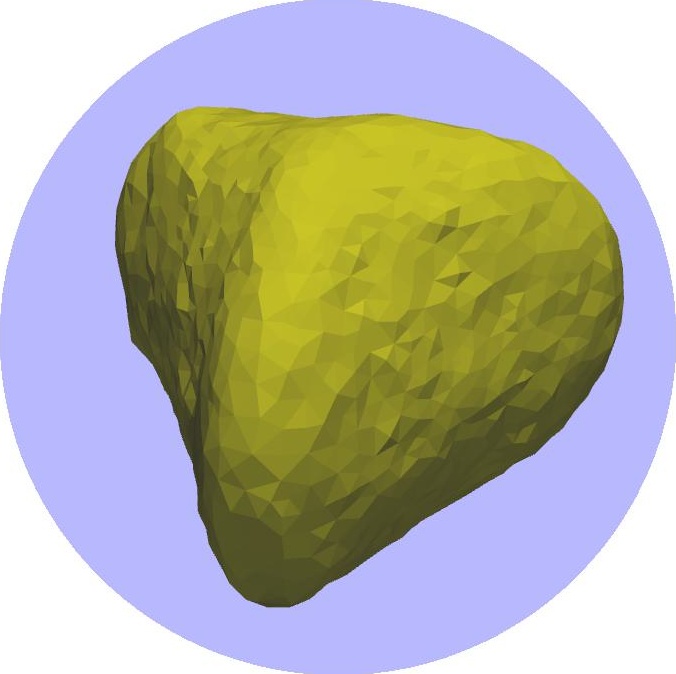}}
    \caption{Case \texttt{C2}}\label{fig:3D_case2}
  \end{subfigure}
  \begin{subfigure}[t]{0.16\linewidth}
    \resizebox{\linewidth}{!}{\includegraphics{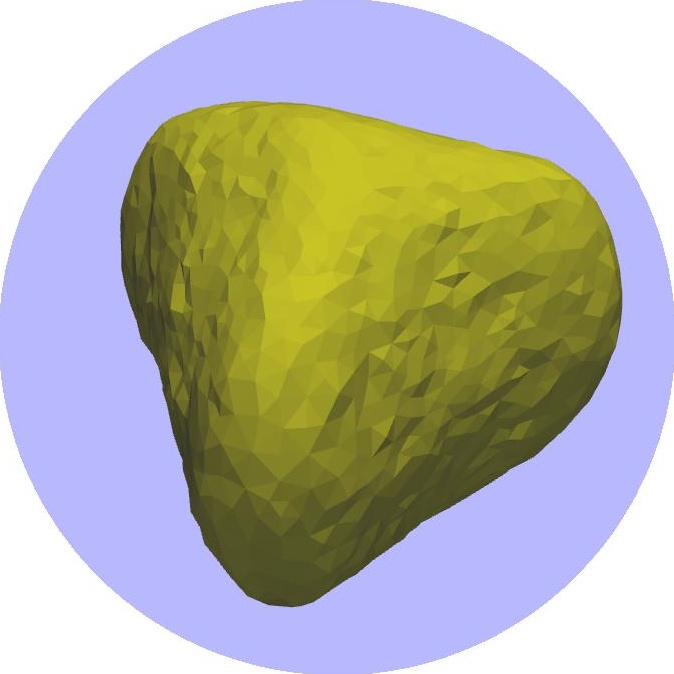}}
    \caption{Case \texttt{C3}}\label{fig:3D_case3}
  \end{subfigure}
  \begin{subfigure}[t]{0.16\linewidth}
    \resizebox{\linewidth}{!}{\includegraphics{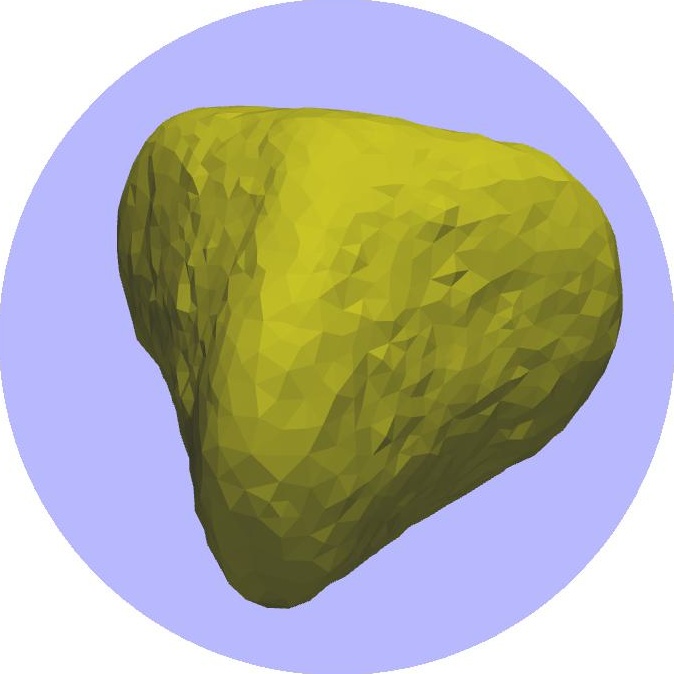}}
    \caption{Case \texttt{C4}}\label{fig:3D_case4}
  \end{subfigure}
  \begin{subfigure}[t]{0.16\linewidth}
    \resizebox{\linewidth}{!}{\includegraphics{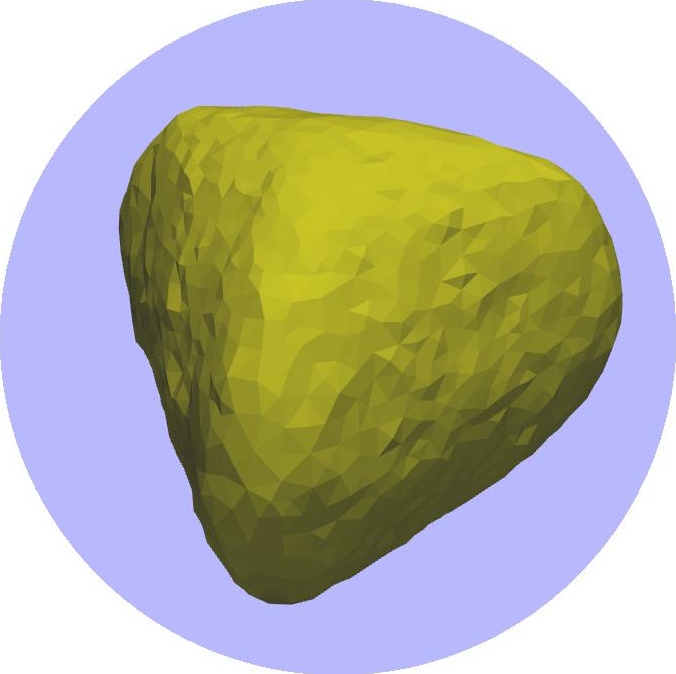}}
    \caption{Case \texttt{C5}}\label{fig:3D_case5}
  \end{subfigure}
  \par\bigskip%%%%%%%%%%%%%%%%%%%%%%%    
  \begin{subfigure}[t]{0.16\linewidth}
    \resizebox{\linewidth}{!}{\includegraphics{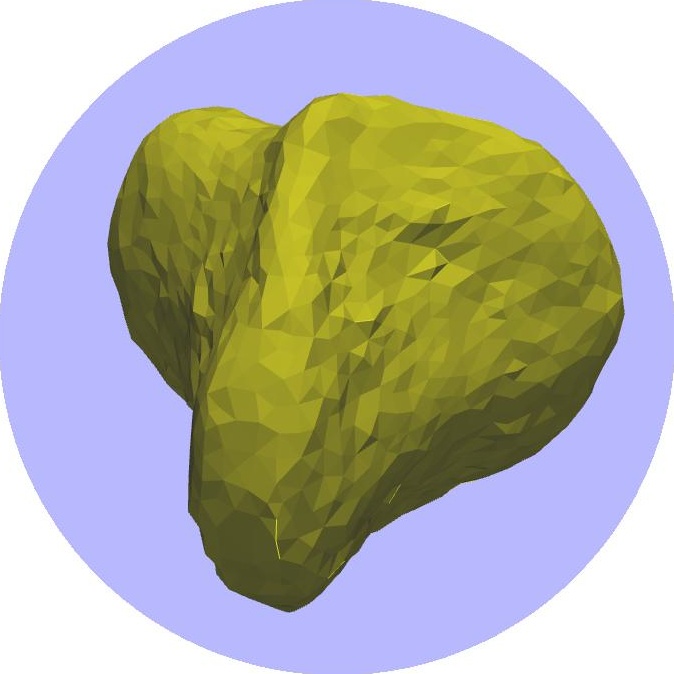}}
    \caption{Case \texttt{C6}}\label{fig:3D_case6}
  \end{subfigure}  
  \begin{subfigure}[t]{0.16\linewidth}
    \resizebox{\linewidth}{!}{\includegraphics{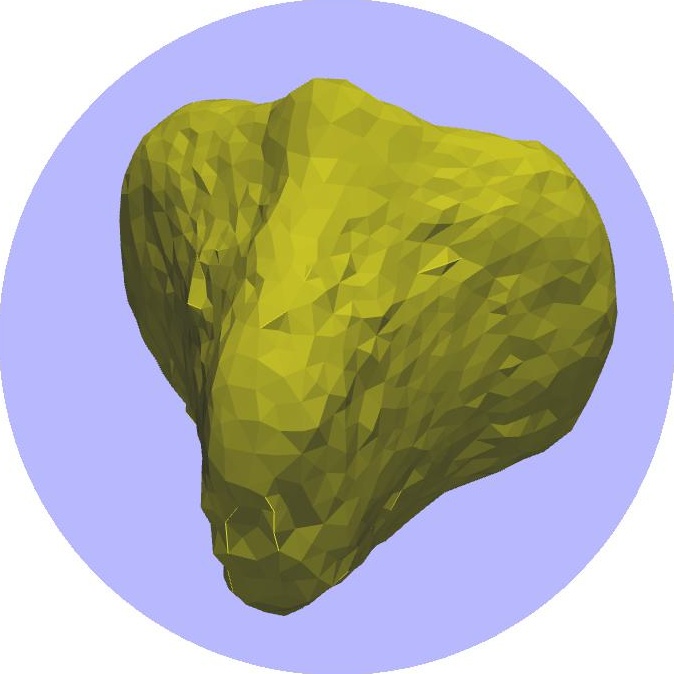}}
    \caption{Case \texttt{C7}}\label{fig:3D_case7}
  \end{subfigure}  
  \begin{subfigure}[t]{0.16\linewidth}
    \resizebox{\linewidth}{!}{\includegraphics{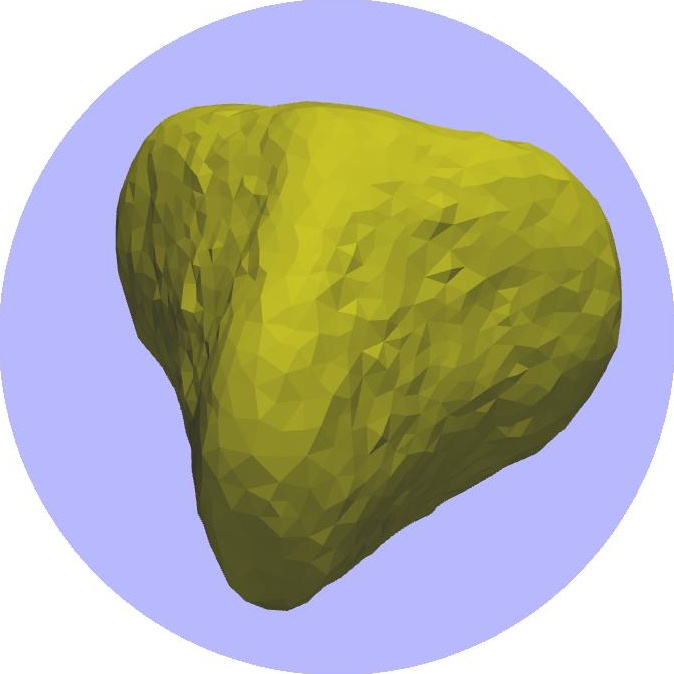}}
    \caption{Case \texttt{C8}}\label{fig:3D_case8}
  \end{subfigure}   
  \begin{subfigure}[t]{0.16\linewidth}
    \resizebox{\linewidth}{!}{\includegraphics{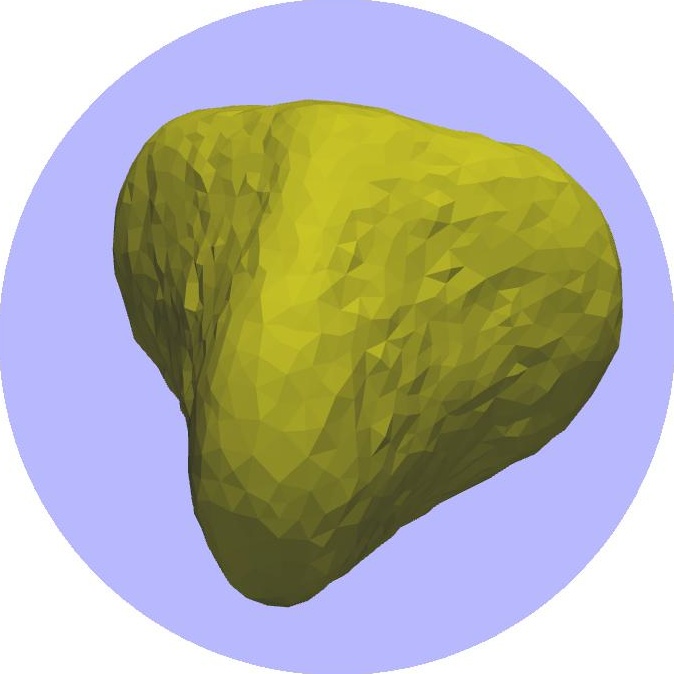}}
    \caption{Case \texttt{C9}}\label{fig:3D_case9}
  \end{subfigure}     
  \begin{subfigure}[t]{0.16\linewidth}
    \resizebox{\linewidth}{!}{\includegraphics{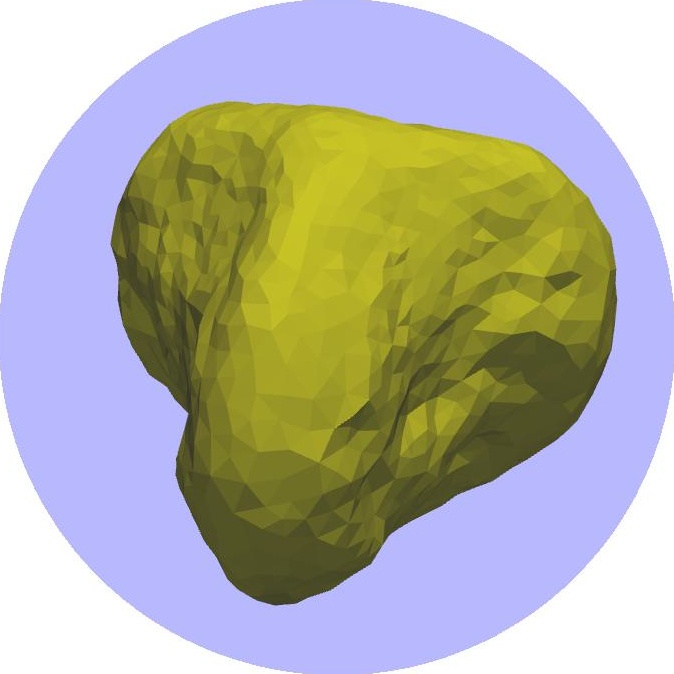}}
    \caption{Case \texttt{C10}}\label{fig:3D_case10}
  \end{subfigure}    
  \begin{subfigure}[t]{0.16\linewidth}
    \resizebox{\linewidth}{!}{\includegraphics{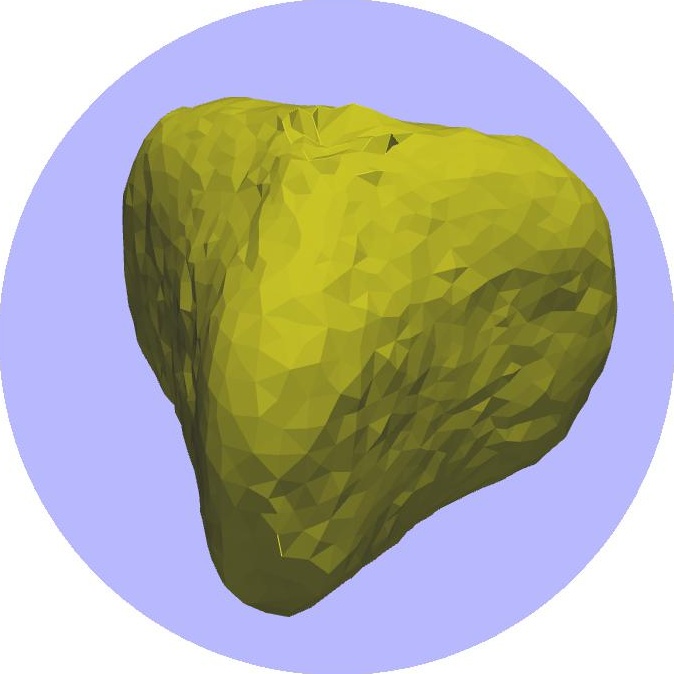}}
    \caption{Case \texttt{C11}}\label{fig:3D_case11}
  \end{subfigure}      
   \caption{Results of approximation with different number/choices of boundary measurements}
\end{figure}
%
%
%
%
%
%
%
%------------------------------------------------------------------------------------
% 			FIGURE  28: NUMERICAL RESULT
%------------------------------------------------------------------------------------         
\begin{figure}[htp!]
\centering
\resizebox{0.225\linewidth}{!}{\includegraphics{st4XYZ.jpg}} \hfill
\resizebox{0.225\linewidth}{!}{\includegraphics{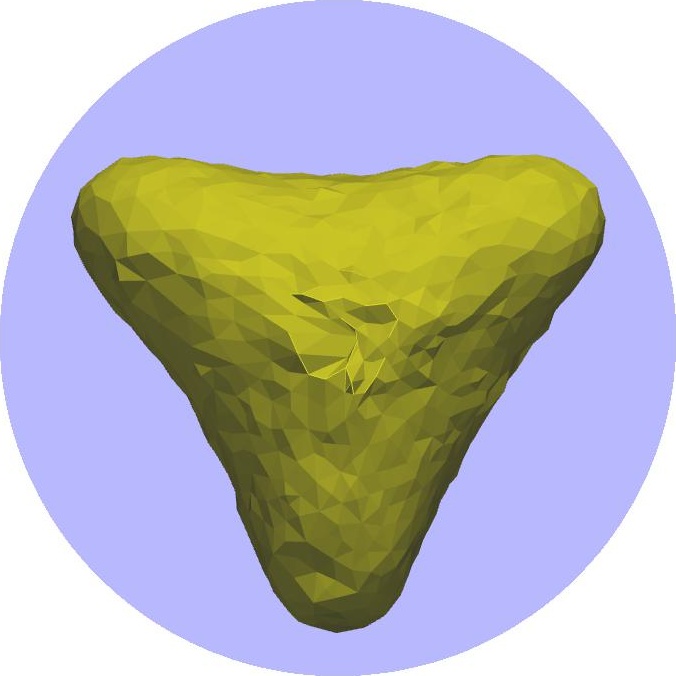}} \hfill
\resizebox{0.225\linewidth}{!}{\includegraphics{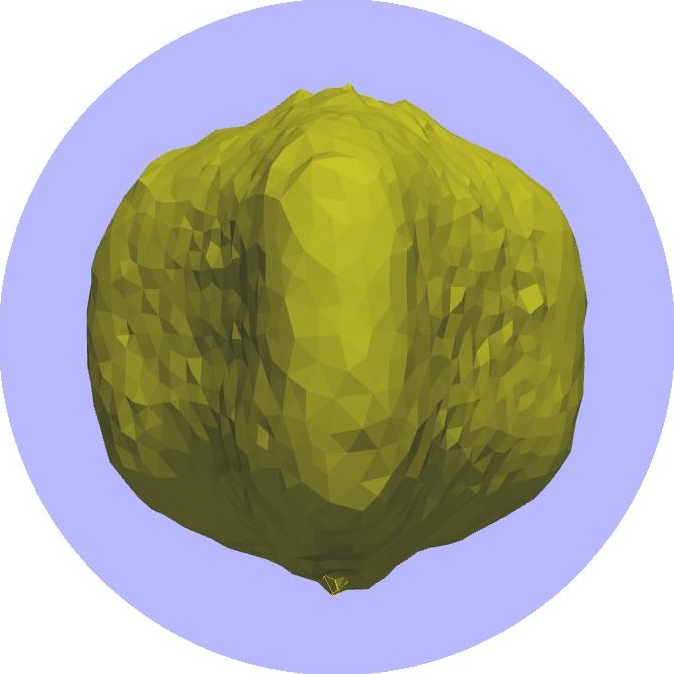}} \hfill
\resizebox{0.225\linewidth}{!}{\includegraphics{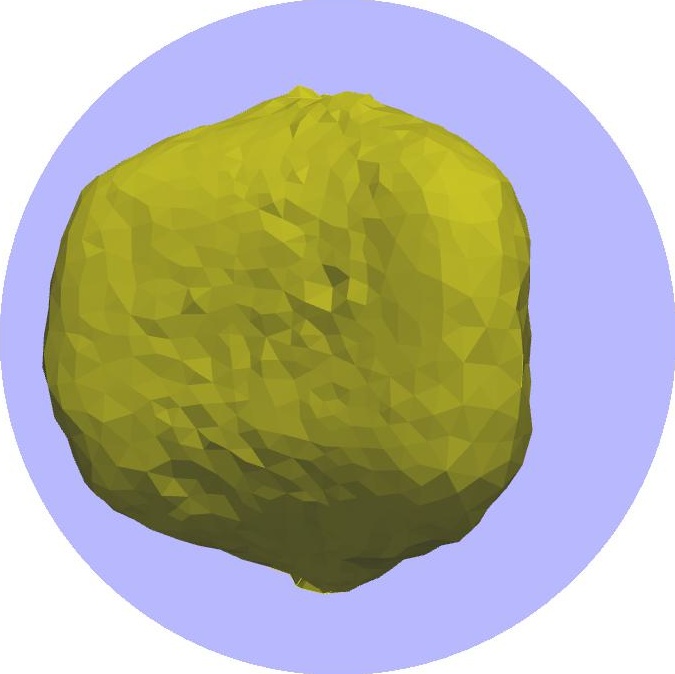}}
\caption{Result of Case \texttt{C11}: $(g^{(1)}, g^{(2)}, g^{(3)}, g^{(4)}) = (0.1,0.3\cos \theta, 0.3\sin \theta, 0.5\cos 2\phi)$}
\label{fig:3D_case10_view_all}
\end{figure}
%------------------------------------------------------------------------------------
% 			FIGURE  29: NUMERICAL RESULT
%------------------------------------------------------------------------------------         
\begin{figure}[htp!]
\centering
\resizebox{0.49\linewidth}{!}{\includegraphics{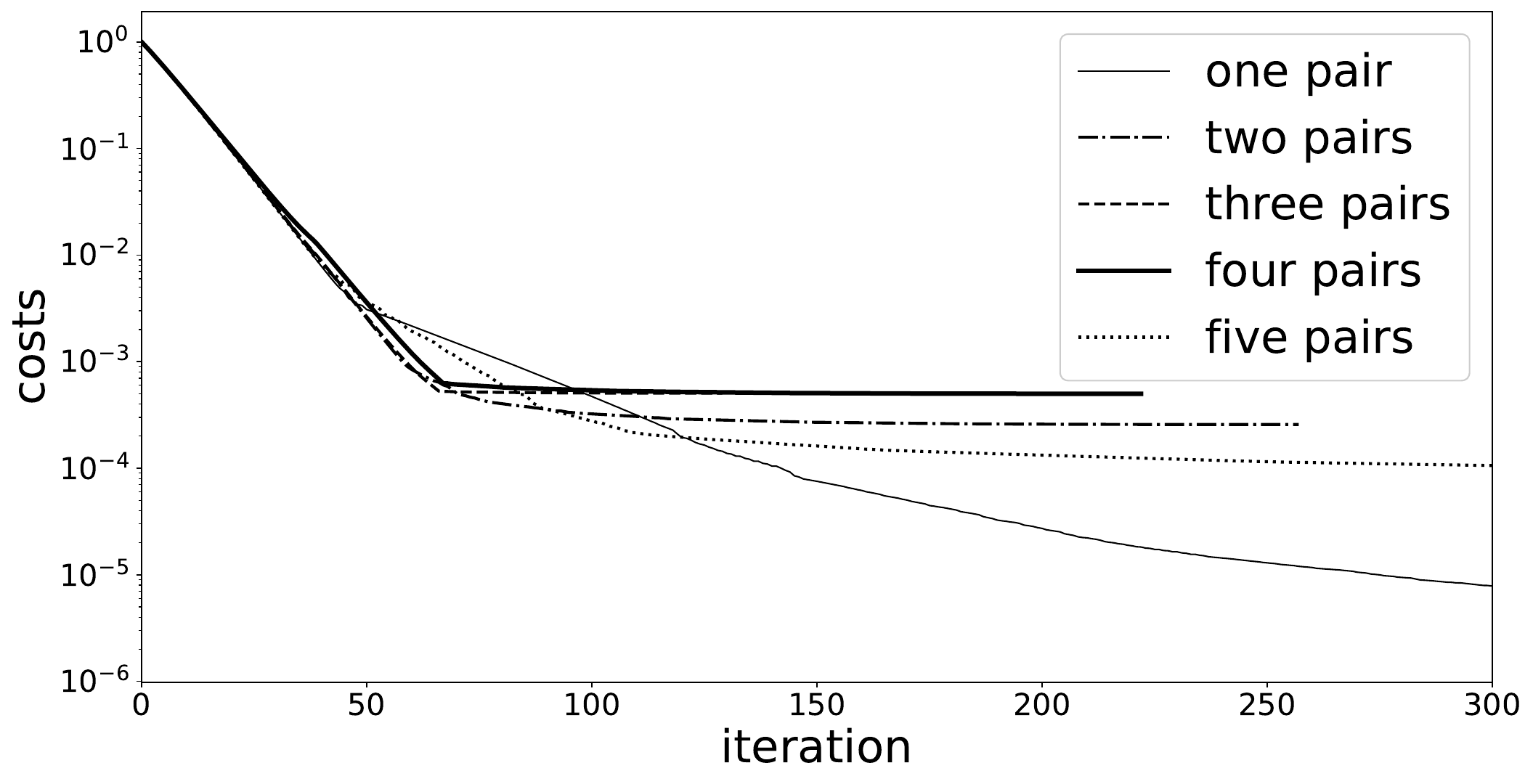}} \hfill
\resizebox{0.49\linewidth}{!}{\includegraphics{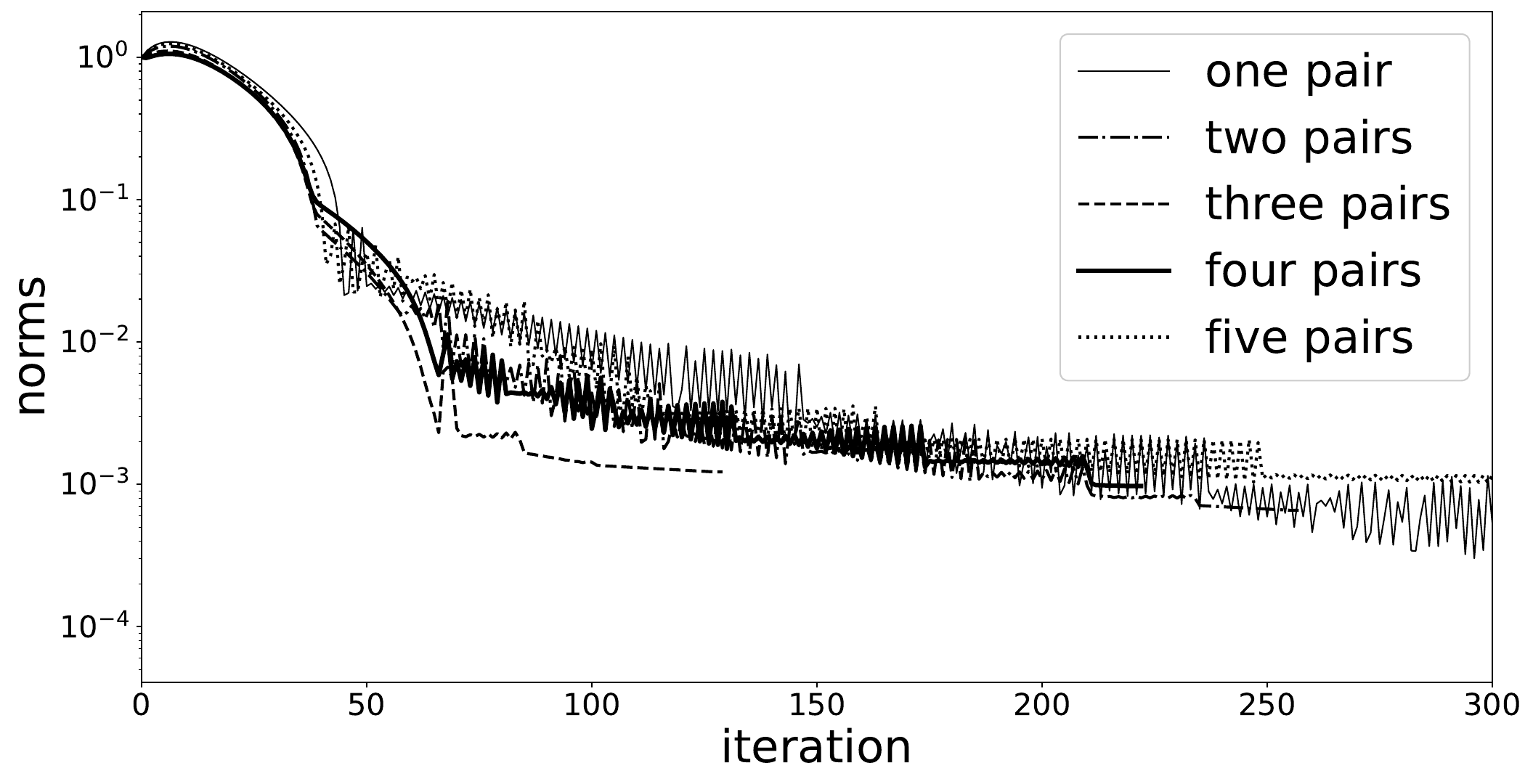}}
\caption{Histories of costs and gradient norms corresponding to the data set $(g^{(1)}, g^{(2)}, g^{(3)}, g^{(4)}, g^{(5)}) = (0.1,0.2\cos \theta, 0.2\sin \theta, 0.6\cos \phi, 0.6\sin \phi)$}
\label{fig:3D_case11_cost_and_norm}
\end{figure}
%
%
%
%
%-------------------------------------- CONCLUSION -------------------------------------- 
%-------------------------------------- CONCLUSION -------------------------------------- 
%-------------------------------------- CONCLUSION --------------------------------------  
\section{Conclusion}\label{sec:conclusion}
We revisited two classical shape optimization approaches for solving the shape inverse problem with the Robin condition for the Laplace equation, focusing on boundary data tracking in a least-squares sense. 
The problem is highly ill-posed, as demonstrated by the instability of the shape Hessian in the optimization setting.

Our numerical findings show that using multiple boundary measurements improves the accuracy of detecting unknown inclusions, especially in two dimensions. 
However, reconstructions suffer when the concave part of the interior boundary is distant from the measurement region, likely due to the severe ill-posed nature of the problem. Addressing this challenge is a focus of our future work. 
Additionally, the smoothness of the interior boundary affects detection accuracy.
In three-dimensional cases, our proposed strategy does not significantly improve detection unless the prescribed data is carefully chosen. This prompts further exploration of alternative strategies for improved boundary reconstruction. 
Future research will involve developing more advanced computational techniques and schemes.

Furthermore, our computational strategy for solving the shape inverse problem can be adapted to other shape optimization formulations. Employing multiple boundary measurements is expected to significantly enhance detection capabilities in such cases.
%-------------------------------------- ACKNOWLEDGMENT --------------------------------------  
\section*{Acknowledgements}
JFTR is partially supported by the Japan Science and Technology Agency under CREST Grant Number JPMJCR2014 and by the Japan Society for the Promotion of Science (JSPS) Grant-in-Aid for Early-Career Scientists under Japan Grant Number JP23K13012.
The authors gratefully acknowledge the anonymous reviewer for their valuable comments, which significantly enhanced the final article. 
Additionally, special thanks to John Sebastian Simon for his correction in the earlier version of the paper.


%% \section*{Conflict of interest}
%%
%% The authors declare that they have no conflict of interest.

\bibliographystyle{alpha} 
\bibliography{main}   % name your BibTeX data base

\begin{thebibliography}{DMNV07}

\bibitem[ADEK07]{AfraitesDambrineEpplerKateb2007}
L.~Afraites, M.~Dambrine, K.~Eppler, and D.~Kateb.
\newblock Detecting perfectly insulated obstacles by shape optimization
  techniques of order two.
\newblock {\em Discrete Contin. Dyn. Syst. Ser. B}, 8(2), 2007.

\bibitem[ADK08]{AfraitesDambrineKateb2008}
L.~Afraites, M.~Dambrine, and D.~Kateb.
\newblock On second order shape optimization methods for electrical impedance
  tomography.
\newblock {\em SIAM J. Control Optim.}, 47(3):1556--1590, 2008.

\bibitem[AF03]{AdamsFournier2003}
R.~A. Adams and J.~J.~F. Fournier.
\newblock {\em Sobolev {S}paces}, volume 140 of {\em Pure and Applied
  Mathematics}.
\newblock Academic Press, Amsterdam, 2003.

\bibitem[Afr22]{Afraites2022}
L.~Afraites.
\newblock A new coupled complex boundary method ({C}{C}{B}{M}) for an inverse
  obstacle problem.
\newblock {\em Discrete Contin. Dyn. Syst. Ser. S}, 15(1):23 -- 40, 2022.

\bibitem[AMN22]{AfraitesMasnaouiNachaoui2022}
L.~Afraites, C.~Masnaoui, and M.~Nachaoui.
\newblock Shape optimization method for an inverse geometric source problem and
  stability at critical shape.
\newblock {\em Discrete Contin. Dyn. Syst. Ser. S}, 15(1):1--21, 2022.

\bibitem[AMR09]{AlvesMartinsRoberty2009}
C.~S. Alves, N.~F.~M. Martins, and N.~C. Roberty.
\newblock Full identification of acoustic sources with multiple frequencies and
  boundary measurements.
\newblock {\em Inverse Prob. Imaging}, 3(2):275--294, 2009.

\bibitem[APR03]{AlessandriniDelPieroRondi2003}
G.~Alessandrini, L.~Del Piero, and L.~Rondi.
\newblock Stable determination of corrosion by a single electrostatic boundary
  measurement.
\newblock {\em Inverse Problems}, 19:973--984, 2003.

\bibitem[AR22]{AfraitesRabago2022}
L.~Afraites and J.~F.~T. Rabago.
\newblock Shape optimization methods for detecting an unknown boundary with the
  {R}obin condition by a single measurement.
\newblock {\em Discrete Contin. Dyn. Syst. Ser. S}, 2022.

\bibitem[AS07]{AlessandriniSincich2007}
G.~Alessandrini and E.~Sincich.
\newblock Solving elliptic {C}auchy problems and identification of non-linear
  corrosion.
\newblock {\em J. Comput. Appl. Math.}, 198:307--320, 2007.

\bibitem[Bac09]{Bacchelli2009}
V.~Bacchelli.
\newblock Uniqueness for the determination of unknown boundary and impedance
  with the homogeneous {R}obin condition.
\newblock {\em Inverse Problems}, 25:Art. 015004 (4pp), 2009.

\bibitem[BC07]{BoulkhemairChakib2007}
A.~Boulkhemair and A.~Chakib.
\newblock On the uniform {P}oincar\'{e} inequality.
\newblock {\em Comm. Partial Differential Equations}, 32:1439--1447, 2007.

\bibitem[BNC08]{Boulkhemairetal2008}
A.~Boulkhemair, A.~Nachaoui, and A.~Chakib.
\newblock Uniform trace theorem and application to shape optimization.
\newblock {\em Appl. Comput. Math.}, 7:192--205, 2008.

\bibitem[BNC13]{Boulkhemairetal2013}
A.~Boulkhemair, A.~Nachaoui, and A.~Chakib.
\newblock A shape optimization approach for a class of free boundary problems
  of {B}ernoulli type.
\newblock {\em Appl. Math.}, 58:205--221, 2013.

\bibitem[CDK13]{CaubetDambrineKateb2013}
F.~Caubet, M.~Dambrine, and D.~Kateb.
\newblock Shape optimization methods for the inverse obstacle problem with
  generalized impedance boundary conditions.
\newblock {\em Inverse Problems}, 29:Art. 115011 (26pp), 2013.

\bibitem[Che75]{Chenais1975}
D.~Chenais.
\newblock On the existence of a solution in a domain identification problem.
\newblock {\em J. Math. Anal. Appl.}, 52:189--219, 1975.

\bibitem[CJ99]{ChaabaneaJaoua1999}
S.~Chaabane and M.~Jaoua.
\newblock Identification of {R}obin coefficients by means of boundary
  measurements.
\newblock {\em Inverse Problems}, 15:1425--1438, 1999.

\bibitem[CK07]{CakoniKress2007}
F.~Cakoni and R.~Kress.
\newblock Integral equations for inverse problems in corrosion detection from
  partial cauchy data.
\newblock {\em Inverse Prob. Imaging}, 1:229--245, 2007.

\bibitem[CKS10]{CakoniKressSchuft2010a}
F.~Cakoni, R.~Kress, and C.~Schuft.
\newblock Integral equations for inverse problems in corrosion detection from
  partial {C}auchy data.
\newblock {\em Inverse Problems}, 26:Art. 095012 24pp., 2010.

\bibitem[Dam02]{Dambrine2002}
M.~Dambrine.
\newblock On variations of the shape {H}essian and sufficient conditions for
  the stability of critical shapes.
\newblock {\em Rev. R. Acad. Cienc. Exactas Fis. Nat. Ser. A. Mat.},
  96:95--121, 2002.

\bibitem[DD97]{DancerDaners1997}
E.~N. Dancer and D.~Daners.
\newblock Domain perturbation for elliptic equations subject to {R}obin
  boundary conditions.
\newblock {\em J. Differential Equations}, 138:86--132, 1997.

\bibitem[DMNV07]{Doganetal2007}
G.~Do\v{g}an, P.~Morin, R.H. Nochetto, and M.~Verani.
\newblock Discrete gradient flows for shape optimization and applications.
\newblock {\em Comput. Methods Appl. Mech. Engrg.}, 196:3898--3914, 2007.

\bibitem[DP00]{DambrinePierre2000}
M.~Dambrine and M.~Pierre.
\newblock About stability of equilibrium shapes.
\newblock {\em Model Math. Anal. Numer.}, 34:811--834, 2000.

\bibitem[DSZ03]{DambrineSokolowskiZochowski2003}
M.~Dambrine, J.~Sokolowski, and A.~Zochowski.
\newblock On analysis in shape optimisation: critical shapes for {N}eumann
  problem.
\newblock {\em Control Cybern.}, 30(3):503--528, 2003.

\bibitem[DZ11]{DelfourZolesio2011}
M.~C. Delfour and J.-P. Zol\'{e}sio.
\newblock {\em {S}hapes and {G}eometries: {M}etrics, {A}nalysis, {D}ifferential
  {C}alculus, and {O}ptimization}, volume~22 of {\em Adv. Des. Control}.
\newblock SIAM, Philadelphia, 2nd edition, 2011.

\bibitem[EH05]{EpplerHarbrecht2005}
K.~Eppler and H.~Harbrecht.
\newblock A regularized {N}ewton method in electrical impedance tomography
  using shape hessian information.
\newblock {\em Control Cybern.}, 34:203--225, 2005.

\bibitem[EHS07]{EpplerHarbrechtSchneider2007}
K.~Eppler, H.~Harbrecht, and R.~Schneider.
\newblock On convergence in elliptic shape optimization.
\newblock {\em SIAM J. Control Optim.}, 46(1):61--83, 2007.

\bibitem[Epp00]{Eppler2000a}
K.~Eppler.
\newblock Boundary integral representations of second derivatives in shape
  optimization.
\newblock {\em Discuss. Math. Differ. Incl. Control. Optim.}, 20:487--516,
  2000.

\bibitem[Fan22]{Fang2022}
W.~Fang.
\newblock Simultaneous recovery of {R}obin boundary and coefficient for the
  {L}aplace equation by shape derivative.
\newblock {\em J. Comput. Appl. Math.}, 413:Art. 114376 13 pp, 2022.

\bibitem[FI07]{FasinoInglese2007}
D.~Fasino and G.~Inglese.
\newblock Recovering nonlinear terms in an inverse boundary value problem for
  {L}aplace's equation: a stability estimate.
\newblock {\em J. Comput. Appl. Math.}, 198:460--470, 2007.

\bibitem[FL04]{FangLu2004}
W.~Fang and M.~Lu.
\newblock A fast collocation method for an inverse boundary value problem,.
\newblock {\em Internat. J. Numer. Methods Engrg.}, 59:1563--1585, 2004.

\bibitem[FLM19]{FangLinMa2019}
W.~Fang, F.~Lin, and Y.~Ma.
\newblock Fast algorithms for boundary integral equations on elliptic domains
  and related inverse problems.
\newblock {\em East Asian J. Appl. Math.}, 9:485--505, 2019.

\bibitem[FZ09]{FangZeng2009}
W.~Fang and S.~Zeng.
\newblock Numerical recovery of {R}obin boundary from boundary measurements for
  the {L}aplace equation.
\newblock {\em J. Comput. Appl. Math.}, 224:573--580, 2009.

\bibitem[GLZ22]{GongLiZhu2022}
Wei Gong, Jiajie Li, and Shengfeng Zhu.
\newblock Improved discrete boundary type shape gradients for
  {P}{D}{E}-constrained shape optimization.
\newblock {\em SIAM J. Sci. Comput.}, 44:A2464--A2505, 2022.

\bibitem[GPT17]{GiacominiPantzaTrabelsi2017}
M.~Giacomini, O.~Pantz, and K.~Trabelsi.
\newblock Certified descent algorithm for shape optimization driven by
  fully-computable a posteriori error estimators.
\newblock {\em ESAIM - Control Optim. Calc. Var.}, 23:977--1001, 2017.

\bibitem[Gri85]{Grisvard1985}
P.~Grisvard.
\newblock {\em Elliptic Problems in Nonsmooth Domains}.
\newblock Pitman Publishing, Marshfield, Massachusetts, 1985.

\bibitem[GZ21]{GongZhu2021}
Wei Gong and Shengfeng Zhu.
\newblock On discrete shape gradients of boundary type for
  {P}{D}{E}-constrained shape optimization.
\newblock {\em SIAM J. Numer. Anal.}, 59:1510--1541, 2021.

\bibitem[Hec12]{Hecht2012}
F.~Hecht.
\newblock New development in {F}ree{F}em++.
\newblock {\em J. Numer. Math.}, 20:251--265, 2012.

\bibitem[HKKP03]{HKKP2003}
J.~Haslinger, T.~Kozubek, K.~Kunisch, and G.~H. Peichl.
\newblock Shape optimization and fictitious domain approach for solving
  free-boundary value problems of {B}ernoulli type.
\newblock {\em Comput. Optim. Appl.}, 26(3):231--251, 2003.

\bibitem[HKKP04]{HKKP2004a}
J.~Haslinger, T.~Kozubek, K.~Kunisch, and G.~H. Peichl.
\newblock An embedding domain approach for a class of 2-d shape optimization
  problems: {M}athematical analysis.
\newblock {\em J. Math. Anal. Appl.}, 209(2):665--685, 2004.

\bibitem[HM03]{HaslingerMakinen2003}
J.~Haslinger and R.~A.~E. M\"{a}kinen.
\newblock {\em Introduction to Shape Optimization: Theory, Approximation, and
  Computation}.
\newblock SIAM, Philadelphia, 2003.

\bibitem[Hol01]{Holzleitner2001}
L.~Holzleitner.
\newblock Hausdorff convergence of domains and their boundaries for shape
  optimal design.
\newblock {\em Control Cybern.}, 30(1):23--44, 2001.

\bibitem[HP18]{HenrotPierre2018}
A.~Henrot and M.~Pierre.
\newblock {\em Shape Variation and Optimization: A Geometrical Analysis},
  volume~28 of {\em Tracts in Mathematics}.
\newblock European Mathematical Society, Z\"{u}rich, 2018.

\bibitem[IK06]{IvanyshynKress2006}
O.~Ivanyshyn and R.~Kress.
\newblock Nonlinear integral equations for solving inverse boundary value
  problems for inclusions and cracks.
\newblock {\em J. Integral Equ. Appl.}, 18(1):13--38, 2006.

\bibitem[IM04]{IngleseMariani2004}
G.~Inglese and F.~Mariani.
\newblock Corrosion detection in conducting boundaries.
\newblock {\em Inverse Problems}, 20:1207--1215, 2004.

\bibitem[Ing97]{Inglese1997}
G.~Inglese.
\newblock An inverse problem in corrosion detection.
\newblock {\em Inverse Problems}, 13:977--994, 1997.

\bibitem[KC98]{ColtonKress2013}
R.~Kress and D.~Colton.
\newblock {\em Inverse Acoustic and Electromagnetic Scattering Theory},
  volume~93 of {\em Applied Mathematical Sciences}.
\newblock Springer, New York, 3rd edition, 1998.

\bibitem[KMI17]{KurahashiMaruokaIyama2017}
T.~Kurahashi, K.~Maruoka, and T.~Iyama.
\newblock Numerical shape identification of cavity in three dimensions based on
  thermal non-destructive testing data.
\newblock {\em Eng. Optim.}, 49(3):434--448, 2017.

\bibitem[KR05]{KressRundell2005}
R.~Kress and W.~Rundell.
\newblock Nonlinear integral equations and the iterative solution for an
  inverse boundary value problem.
\newblock {\em Inverse Problems}, 21:1207--1223, 2005.

\bibitem[Kre89]{Kreyszig1989}
Erwin Kreyszig.
\newblock {\em Introductory Functional Analysis with Applications}.
\newblock Wiley, 1989.

\bibitem[KS95]{KaupSantosa1995}
P.~G. Kaup and F.~Santosa.
\newblock Nondestructive evaluation of corrosion damage using electrostatic
  measurements.
\newblock {\em J. Nondestr. Eval.}, 14:127--136, 1995.

\bibitem[KSV96]{KaupSantosaVogelius1996}
P.~G. Kaup, F.~Santosa, and M.~Vogelius.
\newblock Method for imaging corrosion damage in thin plates from electrostatic
  data.
\newblock {\em Inverse Problems}, 12:279--293, 1996.

\bibitem[Loh87]{Loh1987}
W.~H. Loh.
\newblock Modeling and measurement of contact resistances.
\newblock Tech. Rep. No. G 830-1, Stanford Electronics Labs., 1987.

\bibitem[Mef17]{Meftahi2017}
H.~Meftahi.
\newblock Stability analysis in the inverse {R}obin transmission problem.
\newblock {\em Math. Meth. Appl. Sci.}, 40(7):2505--2521, 2017.

\bibitem[MS76]{MuratSimon1976}
F.~Murat and J.~Simon.
\newblock Sur le contr\^{o}le par un domaine g\'{e}om\'{e}trique.
\newblock Research report 76015, Univ. Pierre et Marie Curie, Paris, 1976.

\bibitem[Neu97]{Neuberger1997}
J.~W. Neuberger.
\newblock {\em Sobolev Gradients and Differential Equations}.
\newblock Springer-Verlag, Berlin, 1997.

\bibitem[NR00]{NovruziRoche2000}
A.~Novruzi and J.-R. Roche.
\newblock Newton's method in shape optimisation: a three-dimensional case.
\newblock {\em BIT Numer. Math.}, 40:102--120, 2000.

\bibitem[Pir84]{Pironneau1984}
O.~Pironneau.
\newblock {\em Optimal {S}hape {D}esign for {E}lliptic {S}ystems}.
\newblock Springer series in {C}omputational {P}hysics. Springer-Verlag, 1984.

\bibitem[PP09]{PaganiPierotti2009}
C.~D. Pagani and D.~Peerotti.
\newblock Identifiability problems of defects with {R}obin condition.
\newblock {\em Inverse Problems}, 25:Art. 055007 12pp, 2009.

\bibitem[RA18]{RabagoAzegami2018}
J.~F.~T. Rabago and H.~Azegami.
\newblock Shape optimization approach to defect-shape identification with
  convective boundary condition via partial boundary measurement.
\newblock {\em Japan J. Indust. Appl. Math.}, 31(1):131--176, 2018.

\bibitem[RA19]{RabagoAzegami2019b}
J.~F.~T. Rabago and H.~Azegami.
\newblock A new energy-gap cost functional cost functional approach for the
  exterior {B}ernoulli free boundary problem.
\newblock {\em Evol. Equ. Control Theory}, 8(4):785--824, 2019.

\bibitem[RA20]{RabagoAzegami2020}
J.~F.~T. Rabago and H.~Azegami.
\newblock A second-order shape optimization algorithm for solving the exterior
  {B}ernoulli free boundary problem using a new boundary cost functional.
\newblock {\em Comput. Optim. Appl.}, 77(1):251--305, 2020.

\bibitem[Run08]{Rundell2008}
W.~Rundell.
\newblock Recovering an obstacle and its impedance from {C}auchy data.
\newblock {\em Inverse Problems}, 24:1--22, 2008.

\bibitem[Sim80]{Simon1980}
J.~Simon.
\newblock Differentiation with respect to the domain in boundary value.
\newblock {\em Numer. Funct. Anal. Optim.}, 2:649--687, 1980.

\bibitem[Sim89]{Simon1989}
J.~Simon.
\newblock Second variation for domain optimization problems.
\newblock In F.~Kappel, K.~Kunisch, and W.~Schappacher, editors, {\em Control
  and Estimation of Distributed Parameter Systems}, volume~91 of {\em
  International Series of Numerical Mathematics}, pages 361--378, Basel, 1989.
  Birkh\"{a}user.

\bibitem[Sin10]{Sincich2010}
E.~Sincich.
\newblock Stability for the determination of unknown boundary and impedance
  with a {R}obin boundary condition.
\newblock {\em SIAM J. Math. Anal.}, 42:2922--2943, 2010.

\bibitem[SZ92]{SokolowskiZolesio1992}
J.~Soko\l{}owski and J.-P. Zol\'{e}sio.
\newblock {\em {I}ntroduction to {S}hape {O}ptimization: {S}hape {S}ensitivity
  {A}nalysis}.
\newblock Springer Series in Computational Mathematics. Springer-Verlag,
  Berlin, Heidelberg, 1992.

\end{thebibliography}
\end{document}